\documentclass{article}

\usepackage{amsmath}
\usepackage{amsthm}
\usepackage{amsxtra}
\usepackage{amsfonts,amssymb}
\newtheorem{Th}{Theorem}
\newtheorem{Prop}[Th]{Proposition}
\newtheorem{Lm}[Th]{Lemma}
\newtheorem{Co}[Th]{Corollary}

\theoremstyle{definition}
\newtheorem{Def}[Th]{Definition}

\newtheorem{Rem}[Th]{Remark}
\date{}

\begin{title}
 {The classification of the stable factor-representations of   the infinite hyperoctahedral  group.}
\end{title}
\author{N.I. Nessonov\footnote{B. Verkin  Institute for Low Temperature Physics and Engineering
of the National Academy of Sciences of Ukraine, \ \ \ \ \ \ \ \ \ \ \ \ \ \ \ \ \ \ \ \ \ \ \ \ \ \ \ \ \ \ \ \ \ \ \ \ \ \ \ \ \ \ \ \ \ \ \ \ \ \ \ \ \ \ \ \ \ \ \ \ \ \ \ \ \ \ \ \ \ \ \ \ \ \ \ \  \ \ \ \ \ Institute of Mathematics Polish Academy of Sciences. \ \ \ \ \ \ \ \ \ \ \ \ \ \ \ \ \ \ \  \ \ \ \ \ \ \ \ \ \ \ \ \ \ \  \ \ \ \ \ \ \ \ \ \ \ This project was funded by Long-term program of support of the Ukrainian research teams at the Polish Academy of Sciences carried out in collaboration with the U.S. National Academy of Sciences with the financial support of external partners}}

\begin{document}
\maketitle

\begin{abstract}
The full description of the stable factor-representations of   the infinite hyperoctahedral  group  up to quasi-equivalence obtained.
\end{abstract}
\section{Introduction}
In this paper we describe the stable representations  of   the infinite hyperoctahedral  group.

The notion of stability for a representation of an arbitrary group
(or algebra) with respect to  their  automorphisms group was introduced in \cite{Ver-Nes}.   We recall the corresponding definition.

Let $G$ be a countable group, and let ${\rm Aut}\,G$ be its automorphisms group. We endow ${\rm Aut}\,G$ with the topology, in which a base of neighborhoods of the identity of ${\rm Aut}\,G$ consists of the sets
  \begin{eqnarray}\label{topology}
  \mathfrak{U}_g=\left\{ \theta\in {\rm Aut}\,G:\theta(g)=g \right\}, \;\;g\in G.
  \end{eqnarray}
Automorphism $\theta\in{\rm Aut}\,G$ define the automorphisms of the group $\mathbf{C}^\star$-algebra $\mathbf{C}^\star[G]$.  We will identify further these two automorphisms. For convenience, we denote $\mathbf{C}^\star[G]$ by $\mathfrak{A}$.  Let $\mathfrak{A}^*$ be the dual space  of $\mathfrak{A}$. For every functional $\varphi\in\mathfrak{A}^*$ and an arbitrary subgroup $\mathcal{K}\subset{\rm Aut}\,G$ we denote by ${\rm Orb}_\varphi\mathcal{K}$ the set $\left\{ \varphi\circ\theta  \right\}_{\theta \in \mathcal{K}}$.
\begin{Def}\label{k_stable}
A functional $\varphi$ is said to be $\mathcal{K}$-stable  if the map $\mathcal{K}\ni\theta\stackrel{O_\varphi}{\mapsto}\varphi\circ\theta\in{\rm Orb}_\varphi\mathcal{K}\subset\mathfrak{A}^*$ is continuous if we consider the topology (\ref{topology}) on $\mathcal{K}$ and the norm topology of the dual space on ${\rm Orb}_\varphi\mathcal{K}$.
\end{Def}
Denote by ${\rm Ad}\,G$ a subgroup of ${\rm Aut}\,G$ consisting of all automorphisms of the view: $G\ni x\stackrel{{\rm Ad}\,g}{\mapsto}gxg^{-1}\in G$, $g\in G$. A functional $\varphi$ is said to be just {\it stable} if it is ${\rm Ad}\,G$-stable.

Let $\pi$ be a unitary representation of $G$ acting on a Hilbert space $\mathcal{H}$.
\begin{Def}[\cite{Ver-Nes}]\label{stable_Def}
A representation $\pi$ is said to be stable if the functional $\omega_\eta$ on $\mathbf{C}^*[G]$, defined by  $\omega_\eta\left( a \right)=\left( \pi(a)\eta,\eta \right)$, $a\in\mathbf{C}^*[G]$, is stable for every $\eta\in \mathcal{H}$.
\end{Def}
It is obvious that the representations  of finite type\footnote{The representation $\pi$ of a group $G$ has a finite type, if $w^*$-algebra $\pi(G)^{\prime\prime}$, generated by the operators $\pi(g)$, $g\in G$, has a finite type.} are the natural examples of the stable representations.

It follows easily that the next statement holds.
\begin{Prop}\label{GNS_stable}
Let $f$ be a positive functional on $\mathbf{C}^*[G]$ and let $\pi_f$ be the corresponding GNS-representation. If $f$ is {\it stable} then $\pi_f$ is stable too.
\end{Prop}
It is easy to check that $\pi$ is quasi-equivalent to $\pi\circ\theta$ for any $\theta\in{\rm Aut}\,G$ in the completion of ${\rm Ad}\,G$ with respect to the topology (\ref{topology}).

\begin{Def}\label{def_tame}
The representation $\pi$ is said to be {\it tame} if the map $G\ni x\mapsto \pi(x)\eta\in\mathcal{H}$ is continuous for any $\eta\in\mathcal{H}$, if we consider on $G$ the topology, in which a base of neighborhoods of the identity  consists of the sets ${\rm Ad}^{-1}\left(\mathfrak{U}_g \right)$, $g\in G$, and the norm topology on $\mathcal{H}$.
\end{Def}
It is obviously that a {\it tame} representation is {\it stable}.

In the following, we will denote by $B(\mathcal{H})$ the set  of all bounded linear operators on $\mathcal{H}$. For each subset $\mathcal{S}\subset B(\mathcal{H})$, let $\mathcal{S}^\prime=\left\{A\in B(\mathcal{H}):AB=BA \text{ for all }\right.$ $\left. B\in \mathcal{S}\right\}$. If $\mathcal{S}$ is invariant under the $^*$-operation then $\mathcal{S}^{\prime\prime}=\left( \mathcal{S}^\prime\right)^\prime$, the double commutant of $\mathcal{S}$, is the smallest von Neumann algebra containing $\mathcal{S}$, and it is called the von Neumann algebra {\it generated} by $\mathcal{S}$ \cite{TAKES1}.   The closed subspace, spanned by $\mathcal{S}\mathfrak{H}$, where $\mathfrak{H}$ is a subset in $\mathcal{H}$, we will denote by $[\mathcal{S}\mathfrak{H}]$.

Let $M$ stand for $\pi(G)^{\prime\prime}$. Denote by $M_*$ the space  of all weakly continuous linear functionals on $M$ ({\it predual} of $M$). Let $M_*^+\subset M_*$ be the cone of the positive functionals from $M_*$.
\begin{Def}
Let $\widehat{\pi}$ be a representation of $G$ and let $\widehat{M}$ be the von Neumann algebra generated by $\widehat{\pi}(G)$. The representations $\pi$  and $\widehat{\pi}$  are said to be  quasi-equivalent if there is an isomorphism $M\stackrel{\theta}{\rightarrow}\widehat{M}$ such that $\theta\left( \pi(g) \right)=\widehat{\pi}(g)$ for all $g\in G$.
\end{Def}
\begin{Def}[\cite{TAKES1}]
  Let $\phi\in M_*^+$. The smallest unique orthogonal projection $E\in M$ such that $\phi(Ex)=\phi(x)$ for all $x\in M$ is called the {\it support} of $\phi$ and denoted by ${\rm supp}\,\phi$.
\end{Def}
  For each pair $\xi,\eta\in\mathcal{H}$, define functional $\omega_{\xi\eta}\in M_*$ by $\omega_{\xi\eta}(A)=\left( A\xi,\eta \right)$, $A\in M$. Then the norm closure of the subspace, spanned by $\left\{ \omega_{\xi\eta}\right\}_{\xi,\eta\in\mathcal{H}}$, coincides  with $M_*$. It follows the next statement:
\begin{Prop}
  Let $\varphi\in M_*$ be a positive linear functional. Define functional $\varphi_\pi$ on $\mathbf{C}^*[G]$ by $\varphi_\pi(a)=\varphi(\pi(a))$, $a\in\mathbf{C}^*[G]$. If $\pi$ is stable representation then $\varphi_\pi$ is a stable functional.
\end{Prop}
\begin{Co}\label{stable_quasi}
  Let $M\stackrel{\theta}{\mapsto}N$ be an isomorphism between von Neumann algebras $M$ and $N$. Then the representation $G\ni g\stackrel{\theta\circ\pi}{\mapsto}\theta(\pi(g))\in N$ is stable. Thus the notation of the stability is invariant of the quasi-equivalence of the representations.
\end{Co}
The symmetric group $\mathfrak{S}_n$ is the group of  all bijections of $\left\{1,2,\ldots, n \right\}$ to itself. Consider the natural action of $\mathfrak{S}_n$ on $\mathbb{Z}_2^n$ by permutation:
\begin{eqnarray*}
s(z_1,z_2,\ldots,z_n)=\left(z_{s^{-1}(1)},z_{s^{-1}(2)},\ldots z_{s^{-1}(n)} \right), \text{ where } s\in\mathfrak{S}_n, z_j\in\mathbb{Z}_2=\{0,1\}.
\end{eqnarray*}
The corresponding semidirect product $\mathbb{Z}_2^n\rtimes\mathfrak{S}_n$ is the group made up of the elements $\mathbb{Z}_2^n\times\mathfrak{S}_n$ and with internal law of composition $\left(z,s \right)\cdot\left(z^\prime,s^\prime \right)=\left(z+sz^\prime,ss^\prime \right)$. This group arise in many contexts, where it is usually called the Weyl group of type $B$ or the hyperoctahedral group.

Now we identify the group $\mathbb{Z}_2^n\rtimes\mathfrak{S}_n$  with the subgroup of $\mathbb{Z}_2^{n+1}\rtimes\mathfrak{S}_{n+1}$, using the embeddings
\begin{eqnarray*}
\mathfrak{S}_n\ni s\mapsto \tilde{s}\in\mathfrak{S}_{n+1}, \text{ where } \tilde{s}(k)=s(k) \text{ for } k\leq n \text{ and }\tilde{s}(n+1)=n+1;\\
\mathbb{Z}_2^n\ni(z_1,z_2,\ldots,z_n)\mapsto(z_1,z_2,\ldots,z_n,0)\in \mathbb{Z}_2^{n+1}.
\end{eqnarray*}
The corresponding union $\bigcup\limits_{n=1}^\infty \mathbb{Z}_2^n\rtimes\mathfrak{S}_n$ is denoted by $\mathbb{Z}_2\wr\mathfrak{S}_\infty$ and is said the {\it wreath product} $\mathbb{Z}_2$ by $\mathfrak{S}_\infty$. This group known as infinite Weyl group.

Here we give the full description of the stable factor-representations of  $\mathbb{Z}_2\wr\mathfrak{S}_\infty$. In particular,  it follows from this the comprehensive classification of the stable representations of infinite symmetric inverse semigroup.
\subsection{Finite type factor-representation of the infinite hyper\-octahedral  group.}\label{characters_of_hyperocthedral}
The full description of ${\rm II}_1$-factor-representations of the infinite wreath product $G\wr\mathfrak{S}_\infty$ was found in \cite{RB,Hirai} in the case when  $G$ is finite group. In these papers, the authors used the ergodic method of finding characters for locally finite groups, proposed by Vershik and Kerov in \cite{Ver_Ker,Ker_Ver,Ker}. The classification of ${\rm II}_1$-factor-representation of $G\wr\mathfrak{S}_\infty$ for arbitrary countable group $G$, by using Okounkov approach \cite{Ok1,Ok2}, was obtained in \cite{DN}. We recall the corresponding results for the infinite Weyl group $\mathbb{Z}_2\wr\mathfrak{S}_\infty$.

First, we will identify $\bigcup\limits_{n=1}^\infty\mathbb{Z}_2^n=\;_0\mathbb{Z}_2^\infty$ \label{page_Z_2}and $\bigcup\limits_{n=1}^\infty\mathfrak{S}_n=\mathfrak{S}_\infty$ with the corresponding subgroups of $\mathbb{Z}_2\wr\mathfrak{S}_\infty$. Under this condition, every $g\in \mathbb{Z}_2\wr\mathfrak{S}_\infty$ we can write as $g=sz$, where $s\in\mathfrak{S}_\infty$, $z=(z_1,z_2,\ldots,z_n,\ldots)\in \;_0\mathbb{Z}_2^\infty$. Let $s=c_1\cdot c_2\cdots c_l$ be a decomposition of $s$ in the product of the  independent cycles. Denote by ${\rm supp}\,s$ the set $\left\{k:s(k)\neq k \right\}$. Then $sz=c_1\cdot \,^{1}\!z\cdot c_2\cdot \,^{2}\!z\cdots c_l\cdot \,^{l}\!z\cdot z^\prime$, where
$^{j}\!z=\left(\,^{j}\!z_1,\,^{j}\!z_2,\ldots,\,^{j}\!z_n,\ldots \right), z^\prime =(z^\prime_1,z^\prime_2,\ldots,z^\prime_n,\ldots)\in\;_0\mathbb{Z}_2^\infty$, $\,^j\!z_n=\left\{
 \begin{array}{rl}
 z_n,&\text{ if } n\in{\rm supp}\,c_j\\
 0,&\text{ if }  n\notin{\rm supp}\,c_j
 \end{array}\right.$
 and $z^\prime_n=\left\{
 \begin{array}{rl}
 0,&\text{ if } n\in{\rm supp}\,s\\
 z_n,&\text{ if }  n\notin{\rm supp}\,s.
 \end{array}\right.$
 We will call  $c_j\cdot \,^{j}\!z$ the quasi-cycle. In particular, the element $\zeta^{(n)}\in\;_0\mathbb{Z}_2^\infty$ of the form $(0,0,\ldots,0,\underbrace{\zeta}_{n},0,\ldots)$ is a quasi-cycle too. We call the set $${\rm supp}\,(sz)=\left\{k: s(k)\neq k \,\text{ or } z_k\neq0 \right\}$$ the support of $sz$.

Let $\pi$ be the unitary ${\rm II}_1$-factor-representation of $\mathbb{Z}_2\wr\mathfrak{S}_\infty$, and let $\mathcal{M}$ be $w^*$-algebra (factor), generated by the operators $\pi(g)$, $g\in \mathbb{Z}_2\wr\mathfrak{S}_\infty$. Denote by ${\rm tr}$ the unique faithful normal trace on $\mathcal{M}$ such that ${\rm tr}(I)=1$. We will find the formula for the {\it character} $\chi(g)={\rm tr}(\pi(g))$.

Since $\pi$ is factor-representations, the property of the multiplicativity holds. Namely,
\begin{eqnarray}\label{II_1-mult_general}
\chi(sz)={\rm tr}(\pi(sz))=\left(\prod\limits_{j=1}^l {\rm tr}(\pi(c_j\cdot z^{(j)}))\right)\cdot\prod\limits_{j\notin {\rm supp}\,s}{\rm tr}(\pi(z_j^{(j)})).
\end{eqnarray}
Thus it is sufficiently to find the value of $\chi$ on the quasi-cycles $cz$, where $c=\left( n_1\;n_2\;\ldots\;n_k \right)\in\mathfrak{S}_\infty$, $z=(z_1,z_2,\ldots)\in\;_0\mathbb{Z}_2^\infty$. We recall that $z_j=0$ if $j\notin\left\{n_1,n_2,\ldots,n_k \right\}$.

It follows from(\ref{II_1-mult_general}) that the restriction of $\chi$ to $\mathfrak{S}_\infty$ is Thoma character \cite{Thoma}. Let $\alpha=\left(\alpha_1\geq\alpha_2\geq \ldots >0\right)$, $\beta=\left( \beta_1\geq\beta_2\geq\ldots >0\right)$ be the corresponding collections of Thoma parameters. Each of these collections ($\alpha$ or $\beta$) can be empty. Using the description of the indecomposable characters on the infinite wrath product \cite{DN}, we find  nonnegative numbers $\gamma_0$ and $\gamma_1$ and a function $\alpha\cup\beta\stackrel{\sigma}{\mapsto}\{0,1\}$ such that $\gamma_0+\gamma_1=1-\sum \alpha_i-\sum\beta_i$ and
\begin{eqnarray}\label{character_formula}
\chi(cz)=\left\{\begin{array}{rl}
 &\sum\alpha_i^k(-1)^{\sigma(\alpha_i)\sum z_i}+(-1)^{k-1}\sum\beta_i^k(-1)^{\sigma(\beta_i)\sum z_i},\text{ if } k>1;\\
&\gamma_0+\gamma_1(-1)^{ z_{n_1}}+ \sum\alpha_i(-1)^{\sigma(\alpha_i) z_{n_1}}+\sum\beta_i\cdot(-1)^{\sigma(\beta_i) z_{n_1}},\text{ if } k=1.
 \end{array}\right.
\end{eqnarray}

We will denote this character by $\chi_{\alpha\beta\gamma}^\sigma$, where $\gamma=(\gamma_0,\gamma_1)$.
\subsection{The stable representations of $\mathfrak{S}_\infty$.}
The full description of the stable representations of the infinite symmetric group was obtained in \cite{Ver-Nes}. We recall the corresponding results.
Let $\left( \alpha,\beta\right)$ be Thoma parameters \cite{Thoma} and let $\pi_{\alpha\beta}$ be the corresponding ${\rm II}_1$-factor-representation of $\mathfrak{S}_\infty$. Denote by $\chi_{\alpha\beta}$ the corresponding character. For each $n\in \mathbb{N}\cup0$, let $\mathfrak{S}_{n\infty}$ denote the subgroup of $\mathfrak{S}_\infty$ consisting of the elements,  which fix the numbers in $\left\{1,2,\ldots, n \right\}$. \label{Subgroup_n_infty} Thus we have a decreasing sequence of infinite subgroups $\mathfrak{S}_\infty=\mathfrak{S}_{0\infty}\supset\mathfrak{S}_{1\infty}\supset\ldots$, all isomorphic to $\mathfrak{S}_\infty$  and having infinite index at each inclusion.  Set $\mathfrak{S}_n=\left\{s\in\mathfrak{S}_\infty:s(k)=k \text{ for all } k>n \right\}$. We will identify a group $\mathfrak{S}_n\times\mathfrak{S}_{n\;\infty}$ with the Young subgroup of $\mathfrak{S}_\infty$, consisting of all permutations,  fixing every part of the partition $\left\{1,2,\ldots, n \right\}\cup\left\{n+1,n+2,\ldots \right\}$. Let $\lambda$ be the partition of $n$ $(\lambda\vdash n)$ and let $\chi_\lambda$  be the  character of the corresponding irreducible representation of $\mathfrak{S}_n$.  Denote by $\tau_{\lambda\alpha\beta}$ a function on $\mathfrak{S}_\infty$, which defined as follows
\begin{eqnarray}
\begin{split}
&\tau_{\lambda\alpha\beta}(s_1s_2)=\chi_\lambda(s_1)\cdot\chi_{\alpha\beta}(s_2), \text{ where } s_1\in\mathfrak{S}_n,\;s_2\in\mathfrak{S}_{n\;\infty};\\
&\tau_{\lambda\alpha\beta}(s)=0, \text{ if } s\notin\mathfrak{S}_n\times\mathfrak{S}_{n\;\infty}.
\end{split}
\end{eqnarray}
We emphasize that $\tau_{\lambda\alpha\beta}$ is the positive definite function  on $\mathfrak{S}_\infty$.
\begin{Prop}
Let $\pi$ be a stable factor-representation of $\mathfrak{S}_\infty$. Than there exist $n\in\mathbb{N}$, the partition $\lambda$ of $n$ and Thoma parameters $(\alpha,\beta)$ such that $\pi$ is quasi-equivalent to the GNS representation $\Pi_{\lambda \alpha \beta}$ associated to  $\tau_{\lambda\alpha\beta}$.
\end{Prop}
The results from the  next statement is proved in \cite{Ver-Nes}.
\begin{Th}
The following  hold:
\begin{itemize}
  \item {\rm i)} $\Pi_{\lambda \alpha \beta}$ is factor representation of type ${\rm II}$ for any triplet $(\lambda, \alpha, \beta)$;
  \item {\rm ii)} if $\sum\alpha_i+\sum\beta_i<1$ then $\Pi_{\lambda \alpha \beta}$ is quasi-equivalent to ${\rm II}_1$-factor-representation $\pi_{\alpha\beta}$;
\item {\rm iii)} if $\sum\alpha_i+\sum\beta_i=1$ and $n> 0$ then $\Pi_{\lambda \alpha \beta}$ is ${\rm II}_\infty$-factor representation;
  \item {\rm iv)}  if $\sum\alpha_i+\sum\beta_i=1$  and $(\lambda, \alpha, \beta)\neq (\lambda^\prime, \alpha^\prime, \beta^\prime)$ then the representations $\Pi_{\lambda \alpha \beta}$ and $\Pi_{\lambda^\prime \alpha^\prime \beta^\prime}$ are disjunct.
\end{itemize}

\end{Th}

\subsection{The examples}
{\bf Example 1.}
We will consider the vectors from the Hilbert space $l^2(\mathbb{N})$ as the functions on $\mathbb{N}$.
Let the operators $T(s)$, $s\in\mathfrak{S}_\infty$ and $T(z)$, $z=(z_1,z_2,\ldots)$ $\in \bigcup\limits_{n=1}^\infty\mathbb{Z}_2^n=\;_0\mathbb{Z}_2^\infty$ act on $f\in l^2(\mathbb{N})$ as follows
\begin{eqnarray*}
\left(T(s)f\right)(x)=f(s^{-1}x),\;\left(T(z)f\right)(x)=(-1)^{z_x}f(x).
\end{eqnarray*}
It is obviously that $T$ extends to the stable irreducible  representation of  $\mathbb{Z}_2\wr\mathfrak{S}_\infty$.

Let $1^{[1,n]})=(\underbrace{1,1,\ldots,1}_n,0,0,\ldots)\in\;_0\mathbb{Z}_2^\infty\subset\mathbb{Z}_2\wr\mathfrak{S}_\infty$.
Clearly,  $$\lim\limits_{k\to\infty}\left( T\left(1^{[1,2k]})  \right)f,f \right)=\left( f,f \right)\;\text{ and } \lim\limits_{k\to\infty}\left( T\left(1^{[1,2k+1]})  \right)f,f \right)=-\left( f,f \right)$$
 for each $f\in l^2(\mathbb{N})$. Therefore, since the sequence ${\rm Ad}\, 1^{[1,n]}$ tends to the identical automorphism of the group $\mathbb{Z}_2\wr\mathfrak{S}_\infty$,
  $T$ is not  a {\it tame} representation (see definition \ref{def_tame}). In particular, the restriction $T$ to the subgroup $\mathfrak{S}_\infty$ is the irreducible {\it tame} representation \cite{Ol1}.
\medskip

\noindent{\bf Example 2.} Let $B(\mathbf{H})$ be the set of all bounded linear operators on the Hilbert space $\mathbf{H}$, and let ${\rm Tr}$ be the ordinary trace on $B(\mathbf{H})$. We fix an orthonormal basis in $\mathbf{H}$ and will to identify the operators from $B(\mathbf{H})$ with the corresponding matrices. Denote by $\mathcal{J}$ an ideal of the Hilbert-Schmidt operators in $B(\mathbf{H})$. The inner product $\left( x,y \right)={\rm Tr}\left( y^*x \right)$, where $x,y\in \mathcal{J}$, define on $\mathcal{J}$ the structure of the Hilbert space. The group $\mathbb{Z}_2\wr\mathfrak{S}_\infty$ is embedded in the unitary subgroup of $B(\mathbf{H})$ by
\begin{eqnarray*}
&\mathfrak{S}_\infty\ni s\mapsto \left[ s_{ij} \right]=\left[\delta_{is(j)}\right]\in B(\mathbf{H}), \text{ where } \delta_{ij} \text{ is Kronecker delta};\\
&\;_0\mathbb{Z}_2^\infty\ni z=(z_1,z_2,\ldots)\mapsto \left[ (-1)^{z_i} \delta_{ij}\right]\in B(\mathbf{H}).
\end{eqnarray*}
Define the unitary operators $R(s)$ and $R(z)$ on $\mathcal{J}$ by
  \begin{eqnarray*}
  \mathcal{J}\ni \left[x_{ij}\right]\stackrel{R(s)}{\mapsto}\left[ s_{ij} \right]\cdot\left[ x_{ij} \right]\cdot\left[ s_{ij} \right]^{-1}\in\mathcal{J};\\
 \mathcal{J}\ni \left[x_{ij}\right]\stackrel{R(z)}{\mapsto} \left[ (-1)^{z_i} \delta_{ij}\right]\cdot\left[ x_{ij} \right]\cdot \left[ (-1)^{z_i} \delta_{ij}\right]^{-1}\in\mathcal{J}.
  \end{eqnarray*}
  Set $R(sz)=R(s)R(z)$ for any $s\in\mathfrak{S}_\infty$, $z\in\;_0\mathbb{Z}_2^\infty$. It is clear that $R$ is an unitary {\it tame} representation of   $\mathbb{Z}_2\wr\mathfrak{S}_\infty$.

  Let $\left\{v_1,v_2, \ldots   \right\}$ be the collection of the vectors from $\mathcal{J}$. Denote by $\left[  v_1,v_2, \ldots  \right]$ the closure of ${\rm span}\,\left\{v_1,v_2, \ldots   \right\}=\left\{ \sum\limits_i \lambda_iv_i \left|\lambda_i\in\mathbb{C} \right. \right\}$. Let $\mathcal{J}_D$ stand for the subspace of the diagonal matrices in $\mathcal{J}$. Set  $\mathcal{J}_+=\left\{x=\left[x_{ij} \right]\in\mathcal{J}:x_{kl}= x_{lk}\;\text{ and }\;x_{kk}=0 \right\}$, $v_+=\left[ \begin{matrix}0&1&0&0&\ldots\\1&0&0&0&\ldots\\0&0&0&0&\ldots\\ \vdots&\vdots&\vdots&\vdots&\ldots\end{matrix} \right]$ and  $\mathcal{J}_-=\left\{x=\left[x_{kl} \right]\in\mathcal{J}:x_{kl}= -x_{lk}\right\}$,  $v_-=\left[ \begin{matrix}0&1&0&0&\ldots\\-1&0&0&0&\ldots\\0&0&0&0&\ldots\\ \vdots&\vdots&\vdots&\vdots&\ldots\end{matrix} \right]$. A trivial verification shows that the subspaces $\mathcal{J}_D$, $J_+=\left[ R\left( \mathbb{Z}_2\wr\mathfrak{S}_\infty  \right)v_+ \right]$ and $J_-=\left[ R\left( \mathbb{Z}_2\wr\mathfrak{S}_\infty  \right)v_- \right]$ are pairwise orthogonal and $\mathcal{J}=\mathcal{J}_D\oplus \mathcal{J}_+\oplus\mathcal{J}_-$. Furthermore, the restriction of $R$ to each of these subspaces is the irreducible tame representation of $\mathbb{Z}_2\wr\mathfrak{S}_\infty$. In particular, these restrictions are pairwise disjunct.
\medskip

\noindent{\bf Example 3.} Define the measure $\nu$ on the set $\left\{ 0,1  \right\}$ by $\nu(0)=p$, $\nu(1)=q$, where $p+q=1$ and $p,q>0$. Endow the space $X=\{0,1\}^\mathbb{N}$ with the product measure $\mu=\nu^\mathbb{N}$. The group $\mathfrak{S}_\infty$ acts natural on $X$ by
  \begin{eqnarray*}
  X\ni(x_1,x_2,\ldots)=x\mapsto xs=(x_{s^{-1}(1)},x_{s^{-1}(2)},\ldots)\in X, s\in\mathfrak{S}_\infty.
  \end{eqnarray*}
The corresponding unitary representation $\Pi$ of $\mathbb{Z}_2\wr\mathfrak{S}_\infty$  is define on $L^2(X,\mu)$ as follows
\begin{eqnarray*}
\left(\Pi(s)\eta\right)=\eta(xs), \text{ where } \eta\in L^2(X,\mu), s\in\mathfrak{S}_\infty;\\
\left( \Pi(z)\eta \right)(x)=(-1)^{\sum x_iz_i}\eta(x),  \text{ where } z=(z_1,z_2,\ldots)\in\;_0\mathbb{Z}_2^\infty.
\end{eqnarray*}
It is easy to prove that $\Pi$ is irreducible representation of $\mathbb{Z}_2\wr\mathfrak{S}_\infty$.
Now we prove that $\Pi$ is nonstable representation.

Denote by $\mathbb{I}$ the function from $L^2(X,\mu)$ identically equal to one. Let $(k\;l)\in\mathfrak{S}_\infty$ be  permutation which exchanges only two elements $k$, $l$ and keeps  all others fixed. Set $1^{(n)}=(\underbrace{0,0,\ldots,0}_{n-1},1,0,\ldots)\in \;_0\mathbb{Z}_2^\infty$. An easy computation shows that
  \begin{eqnarray}\label{nonstable_equality}
  \left(\Pi\left( 1^{(n)} \right) \Pi\left( \left( n\;n+1 \right) \right)\Pi\left( 1^{(n)} \right)\mathbb{I},\mathbb{I} \right)-\left( \Pi\left( \left( n\;n+1 \right) \right)\mathbb{I},\mathbb{I} \right)=(p-q)^2-1.
  \end{eqnarray}
Since the sequence ${\rm Ad}\, 1^{(n)}$ tends to the identical automorphism of the group $\mathbb{Z}_2\wr\mathfrak{S}_\infty$, it follows from (\ref{nonstable_equality}) that the state $\omega_\mathbb{I}$ (see definition \ref{stable_Def}) is nonstable. But the restriction of $\Pi$ to $\mathfrak{S}_\infty$ is the reducible tame  representation.
\medskip

\noindent{\bf Example 4.} Let $\pi_{\alpha\beta}$ be  ${\rm II}_1$-factor-representation of $\mathfrak{S}_\infty$ with Thoma parameters $\alpha=\left(\alpha_1\geq\alpha_2\geq \ldots \geq 0\right)$, $\beta=\left( \beta_1\geq\beta_2\geq\ldots\geq 0\right)$ \cite{Thoma}. Suppose that $\alpha_i<1$ $\left( \beta_i<1 \right)$ for all $\alpha_i\in\alpha$ $\left( \text{for all }\beta_i\in\beta \right)$. Denote by $\mathcal{F}$ ${\rm II}_1$-factor, generated by $\pi_{\alpha\beta}\left(  \mathfrak{S}_\infty \right)$. We will suppose that $\pi_{\alpha\beta}$ acts in $L^2(\mathcal{F},{\rm tr})$, where ${\rm tr}$ is the unique normal trace on $\mathcal{F}$, by the left multiplication: $\mathcal{F}\ni a\mapsto \pi_{\alpha\beta}(s)\cdot a$. Define the representation $\widetilde{\pi}_{\alpha\beta}$ of $\mathbb{Z}_2\wr\mathfrak{S}_\infty$ in $L^2(\mathcal{F},{\rm tr})$ by
  \begin{eqnarray*}
  \widetilde{\pi}_{\alpha\beta}(s)=\pi_{\alpha\beta}(s), \text{ when } s\in\mathfrak{S}_\infty,\text{ and }\widetilde{\pi}_{\alpha\beta}(z)={\rm id} \text{ for }z\in\;_0\mathbb{Z}_2^\infty.
  \end{eqnarray*}
Define the representation $\widetilde{T}_{\alpha\beta}$ in $l^2(\mathbb{N})\otimes L^2(\mathcal{F},{\rm tr})$ by
$\widetilde{T}_{\alpha\beta}=T\otimes\widetilde{\pi}_{\alpha\beta}$, where $T$ is the same as in {\bf Example 1}. Clearly, $\widetilde{T}_{\alpha\beta}$ is a stable representation.
The next statement is true.
\begin{Prop}
$\widetilde{T}_{\alpha\beta}$  is a stable ${\rm II}_\infty$-factor-representation of $\mathbb{Z}_2\wr\mathfrak{S}_\infty$.
\end{Prop}
\begin{proof}
Let one dimensional projection $p_n$ acts on  $f\in l^2(\mathbb{N})$ by \newline
$
\left( p_nf\right)(x)=\left\{\begin{array}{rl}
 f(n),&\text{ if } x=n\\
0,&\text{ if } x\neq n.
 \end{array}\right.
$
Clearly, that  $\widetilde{p}_n=p_n\otimes I=\frac{I-\widetilde{T}_{\alpha\beta}\left( 1^{(n)} \right)}{2}$.

Let $\mathfrak{C}$ be an element from the center of $\widetilde{T}_{\alpha\beta}\left( \mathbb{Z}_2\wr\mathfrak{S}_\infty\right)^{\prime\prime}$. Then $\mathfrak{C}_n=\widetilde{p}_n\mathfrak{C}  \widetilde{p}_n$ lies in the center of $w^*$-algebra $\widetilde{p}_n\cdot\widetilde{T}_{\alpha\beta}\left( \mathbb{Z}_2\wr\mathfrak{S}_\infty\right)^{\prime\prime}\cdot \widetilde{p}_n=\widetilde{p}_n\cdot\widetilde{T}_{\alpha\beta}\left( \mathfrak{S}_{\mathbb{N}\setminus n}\right)^{\prime\prime}\cdot\widetilde{p}_n$, where $\mathfrak{S}_{\mathbb{N}\setminus n}=\left\{ s\in\mathfrak{S}_\infty:s(n)=n \right\}$. Since $w^*$-algebra $\widetilde{p}_n\cdot\widetilde{T}_{\alpha\beta}\left( \mathfrak{S}_{\mathbb{N}\setminus n}\right)^{\prime\prime}\cdot\widetilde{p}_n$ is ${\rm II}_1$-factor, we obtain that $\mathfrak{C}_n=c_n\cdot\widetilde{p}_n$, where $c_n\in\mathbb{C}$. Hence, using the relations
$\mathfrak{C}=\sum\limits_{n=1}^\infty\mathfrak{C}_n$, $\widetilde{T}_{\alpha\beta}(s)\cdot\widetilde{p}_n\cdot\widetilde{T}_{\alpha\beta}(s^{-1})=\widetilde{p}_{s(n)}$, we have
$c_j=c_k=c$ for all $k,j$. Therefore, $\mathfrak{C}=cI$. Since the projections $\left\{\widetilde{p}_n \right\}_{n=1}^\infty$ are mutually orthogonal and equivalent,  then $w^*$-algebra $\widetilde{T}_{\alpha\beta}\left( \mathbb{Z}_2\wr\mathfrak{S}_\infty\right)^{\prime\prime}$ is ${\rm II}_\infty$-factor.
\end{proof}
\noindent{\bf Example 5.}\label{example_5} Let $(X,\mu)$ be the same as in {\bf Example 3}.
Define the representation $\widetilde{\Pi}$ of $\mathbb{Z}_2\wr\mathfrak{S}_\infty$ in the space $L^2(X,\mu)\otimes l^2(\mathfrak{S}_\infty)$ by
\begin{eqnarray*}
\left(\widetilde{\Pi}(s)\eta\right)(x,g)=\eta(xs,gs),\;\;\widetilde{\Pi}(z)=\Pi(z)\otimes I, \text{ where } s\in\mathfrak{S}_\infty, z \in \;_0\mathbb{Z}_2^\infty \text{ and } \Pi
\end{eqnarray*}
is defined  in {\bf Example 3}. Let $e$ be identity  of $\mathfrak{S}_\infty$. Set $\xi_e(x,g)=\left\{\begin{array}{rl}
 1,&\text{ if } g=e\\
0,&\text{ if } g\neq e.
 \end{array}\right.$
 If  $p,q>0$ then it is easy to show that $w^*$-algebra $\widetilde{\Pi}\left(\mathbb{Z}_2\wr\mathfrak{S}_\infty\right)^{\prime\prime}$ is ${\rm II}_1$-factor. In particular, the state ${\rm tr}$, defined by  ${\rm tr}(A)=(A\xi_e,\xi_e)$, is unique normal trace on $\widetilde{\Pi}\left(\mathbb{Z}_2\wr\mathfrak{S}_\infty\right)^{\prime\prime}$. The commutant $\widetilde{\Pi}\left(\mathbb{Z}_2\wr\mathfrak{S}_\infty\right)^{\prime}$ of $\widetilde{\Pi}\left(\mathbb{Z}_2\wr\mathfrak{S}_\infty\right)^{\prime\prime}$ is generated by the operators $\widetilde{\Pi}^{\prime}\left( s \right)$ $(s\in\mathfrak{S}_\infty)$ and $\widetilde{\Pi}^{\prime}\left( f\right)$ $(f\in L^\infty(X,\mu))$, which act on  $\eta\in L^2(X,\mu)\otimes l^2(\mathfrak{S}_\infty)$ by
 \begin{eqnarray}\label{Commutant_operators}
 \left( \widetilde{\Pi}^{\prime}_f\eta\right)(x,g)=f\left(xg^{-1}\right)\cdot\eta(x,g),\;\;\;
 \left( \widetilde{\Pi}^{\prime}(s)\eta\right)(x,g)=\eta\left(x,s^{-1}g\right).
 \end{eqnarray}
\medskip

\noindent{\bf Example 6.} Let $\widetilde{R}=T\otimes\widetilde{\Pi}$, where $T$ and $\widetilde{\Pi}$ are defined in {\bf Examples} {\bf1} and {\bf 5} respectively. The next statement is true.
\begin{Prop}\label{Prop_example_6}
  The representations $\widetilde{R}$ and $\widetilde{\Pi}$ are quasi-equivalent.
\end{Prop}
\begin{proof}[Outline of the proof of Proposition \ref{Prop_example_6}] Let $\,^j\!\xi_t\in l^2(\mathbb{N})\otimes L^2(X,\mu)\otimes l^2\left(\mathfrak{S}_\infty\right)$ is defined as follows
\begin{eqnarray}
\,^j\!\xi_t(n,x,g)=\left\{\begin{array}{rl}
 1,&\text{ if }\{n=j\}\&\{ g=t\}\\
0,&\text{ if } \{n\neq j\}\vee\{ g\neq t\}.
 \end{array}\right.
\end{eqnarray}
  Denote by $H_k$ the subspace $\left[\widetilde{R}\left( \mathbb{Z}_2\wr\mathfrak{S}_\infty\right)\,^1\!\xi_{_{(1\;k)}}\right]$, where transposition $(1\;k)$ swaps only two elements $1$, $k$ and leaves everything else fixed.   It is easy to check that the subspaces $H_k$ are pairwise orthogonal for different $k$ and
  \begin{eqnarray}
  \bigoplus\limits_{k=1}^\infty H_k=l^2(\mathbb{N})\otimes L^2(X,\mu)\otimes l^2\left(\mathfrak{S}_\infty\right).
  \end{eqnarray}
  Let $\widetilde{R}_k$ be the restriction of $\widetilde{R}$ to $H_k$. Since the operators $I\otimes\widetilde{\Pi}^{\prime}(\mathfrak{S}_\infty)$ (see (\ref{Commutant_operators})) lie in $\widetilde{R}(\mathbb{Z}_2\wr\mathfrak{S}_\infty)^\prime$ and $\widetilde{\Pi}^{\prime}\left( (1\;k) \right)H_1=H_k$, the representations $\widetilde{R}_k$ are pairwise unitary equivalent.

  Let us prove that $\widetilde{R}_1$ and $\widetilde{\Pi}$ are unitary equivalent. Indeed, an easy computation shows that with $A=\displaystyle{\frac{q}{2p}}\left(I+\widetilde{\Pi}((1,0,0,\ldots))\right)+\frac{p}{2q}\left(I-\widetilde{\Pi}((1,0,0,\ldots))\right)$ we have
  \begin{eqnarray*}
  \left(\widetilde{\Pi}(g)A\xi_e,\xi_e  \right)=\left(\widetilde{R}\left( g \right)\,^1\!\xi_e, \,^1\!\xi_e \right)\text{ for all } g\in \mathbb{Z}_2\wr\mathfrak{S}_\infty.
   \end{eqnarray*} Therefore, the map $L^2(X,\mu)\otimes l^2(\mathfrak{S}_\infty)\ni \widetilde{\Pi}(g)\sqrt{A}\xi_e\stackrel{U}{\mapsto}\widetilde{R}\left( g \right)\,^1\!\xi_e\in H_1$ is extended to an isometry of the space $L^2(X,\mu)\otimes l^2(\mathfrak{S}_\infty)$ onto $H_1$ and $U\cdot \widetilde{\Pi}(g)$ $=\widetilde{R}_1\left( g \right)\cdot U$ for all $g\in \mathbb{Z}_2\wr\mathfrak{S}_\infty$.
\end{proof}

\section{ Stable representations induced from the representations of the finite type.}\label{inducing_stable_representation}

 Let $\mathbb{Z}_2^{n\infty}$ $\left(  \mathbb{Z}_2^{n} \right)$ be the subgroup of $\;_0\mathbb{Z}_2^\infty$, consisting of all elements $z$ of the view  $\left( 0,\ldots,0,z_{n+1},z_{n+2},\ldots \right)$ $\left( \left( z_1,\ldots,z_{n},0,0,\ldots \right) \right)$. Denote by $B_{n\infty}$ $\left( B_{n} \right)$ the subgroup $\mathbb{Z}_2^{n\infty}\cdot\mathfrak{S}_{n\infty}$ $\left( \mathbb{Z}_2^{n}\cdot\mathfrak{S}_{n} \right)$(see page \pageref{Subgroup_n_infty}). It is clear that $B_n=\left\{ b\in B: {\rm supp}\,b\subset \overline{1,n}\right\}$ and $B_{n\infty}=\left\{ b\in B: {\rm supp}\,b\subset \overline{n+1,\infty}\right\}$, where $\overline{1,n}=\{1,2,\ldots,n\}$ and  $\overline{n+1,\infty}$ $=\{n+1.n+2,\ldots\}$. Set $B_\infty^{(n)}=B_n\cdot B_{n\infty}$. Normal subgroup $\;_0\mathbb{Z}_2^\infty$ lies in $B_\infty^{(n)}$ for all $n$. With this notation, we have $B_{0\infty}=\mathbb{Z}_2\wr\mathfrak{S}_\infty$.  For convenience, we denote  group $B_{0\infty}$ by $B$. Let us introduce the notation $\mathfrak{S}_{kn}$ $\left( \mathbb{Z}_2^{\widetilde{kn}}\right)$, where $k<n$, for subgroup $\mathfrak{S}_n\cap\mathfrak{S}_{k\infty}$ $\left(\mathbb{Z}_2^n\cap \mathbb{Z}_2^{k\infty} \right)$. It is clear that $\mathbb{Z}_2^{\widetilde{kn}}$ is natural isomorphic to $\mathbb{Z}_2^{n-k}$. We denote by $B_{kn}$ subgroup $\mathfrak{S}_{kn}\cdot \mathbb{Z}_2^{\widetilde{kn}}$. It is clear that  $\mathbb{Z}_2^k=\mathbb{Z}_2^{\widetilde{0k}}$ and $B_n=B_{0n}$. Now we consider subset $\mathbb{A}\subset\mathbb{N}$ and introduce subgroups $B_\mathbb{A}=\left\{b\in B: {\rm supp}\,b\subset \mathbb{A} \right\}$, $\mathfrak{S}_\mathbb{A}=\left\{s\in \mathfrak{S}_\infty: {\rm supp}\,s\subset \mathbb{A} \right\}$ and $B_\infty^{\mathbb{A}}=B_\mathbb{A}\cdot B_{\mathbb{N}\setminus \mathbb{A}}$.  In particular,  $B_{_{\overline{n+1,\infty}}}=B_{n\infty}$, $B_{_{\overline{1,n}}}=B_n$,   $B_\infty^{^{\overline{1,n}}}=B_\infty^{(n)}$.

 The purpose of this section is to formulate  some properties of the representations of $B$, which are induced
 from the finite factor-representations of group  $B_\infty^{(n)}$.
\subsection{Irreducible representations of $B_n$.}\label{irr_B_n}\label{irreducible_repr_of_B_n}
Let  $1^{(j)}=(\underbrace{0,0,\ldots,0}_{j-1},1,0,\ldots)$ be an element of $\mathbb{Z}_2^n\subset B_n$, where integer $j>0$. Take $\hat{z}=\left(\hat{z}_1, \hat{z}_2,\ldots,\hat{z}_n  \right)\in\mathbb{Z}_2^n$ and define the character $\,^{\Omega{z}}\!\chi$ of $\mathbb{Z}_2^n$ at follows
\begin{eqnarray}\label{mult_character}
\,^{\hat{z}}\!\Omega((z_1,z_2,\ldots,z_n))=(-1)^{\sum z_j\hat{z}_j}.
\end{eqnarray}
For each integer $k$ with $0\leq k\leq n$ define a multiplicative character $\chi_{_{kn}}$ of $\mathbb{Z}_2^n$ by $$\Omega_{_{kn}}\left(1^{(j)}\right)=\left\{\begin{array}{rl}
 1,&\text{ if }j\leq k;\\
-1,&\text{ if } j> k.
 \end{array}\right.$$
 Clearly, $\Omega_{_{kn}}=\,^{\hat{z}}\!\Omega$, where $\hat{z}=(\underbrace{0,0,\ldots,0}_{k},1\ldots,1)$, $\Omega_{_{nn}}(z)=1$ for all $z\in\mathbb{Z}_2^n$ and $\Omega_{_{0n}}(z)$ $=(-1)^{\sum_{j=1}^n z_j}$ for all $z=(z_1,z_2,\ldots,z_n)\in\mathbb{Z}_2^n$.
 It easy to show that any multiplicative  character $\Omega$ of $\mathbb{Z}_2^n$ is conjugate by an element of $\mathfrak{S}_n$ to precisely one $\Omega_{_{kn}}$; i. e. there exists $s\in \mathfrak{S}_n$ such that  $\Omega_{_{kn}}(z)=\Omega_{_{kn}}(szs^{-1})$ for all $z\in \mathbb{Z}_2^n$. Let
 \begin{eqnarray}\label{subgroup_G_chi}
 G_{\Omega_{_{kn}}}=\left\{g\in \mathbb{Z}_2\wr\mathfrak{S}_n: \Omega_{_{kn}}\left( gzg^{-1} \right)=\Omega_{_{kn}}(z) \text{ for all } z\right\}=B_k\cdot B_{kn}
 \end{eqnarray}
 be the stabilizer of $\Omega_{_{kn}}$, and let $\rho$ be an unitary irreducible representations of the Young subgroup $G_{\Omega_{_{kn}}}\cap\mathfrak{S}_n$. Clearly,  $G_{\Omega_{_{kn}}}=\left(G_{\Omega_{_{kn}}}\cap\mathfrak{S}_n\right)\mathbb{Z}_2^n$.   Define the irreducible representation $\,^{^{\Omega_{_{kn}}}}\!\!\!\rho$ of $G_{\Omega_{_{kn}}}$ by
 \begin{eqnarray}\label{repr_rho}
 \,^{^{\Omega_{_{kn}}}}\!\!\!\rho(zs)=\Omega_{_{kn}}(z)\rho(s),\;\text{ where }\; z\in\mathbb{Z}_2^n, s\in G_{\Omega_{_{kn}}}\cap\mathfrak{S}_n.
 \end{eqnarray}
  The irreducible representations of  the group $\mathbb{Z}_2\wr\mathfrak{S}_n$, which is the semidirect product of the subgroup $\mathfrak{S}_n$ by the abelian normal subgroup $\mathbb{Z}_2^n$,   are all uniquely ${\rm Ind}^{B_n}_{G_{\Omega_{_{kn}}}}\left(\,^{^{\Omega_{_{kn}}}}\!\!\!\rho\right)$. Since group $G_{\Omega_{_{kn}}}\cap\mathfrak{S}_n$ is naturally isomorphic to Young subgroup $\mathfrak{S}_k\times\mathfrak{S}_{n-k}$, representation $\rho$ has the form $\rho=\left( {\rm Irr}_{\,^0\!\lambda}\otimes {\rm Irr}_{\,^1\!\!\lambda} \right)$, where ${\rm Irr}_{\,^0\!\lambda}$ and ${\rm Irr}_{\,^1\!\!\lambda}$ are irreducible representations of $\mathfrak{S}_k$ and $\mathfrak{S}_{kn11}$ in the spaces $\mathcal{H}_{\,^0\!\lambda}$ and $\mathcal{H}_{\,^1\!\!\lambda}$, respectively; ${\;^0\!\lambda}\vdash k$, ${\;^1\!\!\lambda}\vdash (n-k)$ are the corresponding partitions.

\subsection{The representations of $B_\infty^{(n)}$ of the finite type.}\label{repr_subgroup}
Let $\,^n\!\pi$ be the factor-representation of   $B_\infty^{(n)}$ of the finite type in the Hilbert space $\,^n\!\mathcal{H}$ with the cyclic and separating vector $\,^n\!\xi$. Set $\,^n\!M=\,^n\!\pi\left( B_\infty^{(n)}\right)^{\prime\prime}$.  Without loss of generality, we assume that $\left( m\,^n\!\xi,\,^n\!\xi\right)=\,^n\!{\rm tr}(m)$ for all $m\in\,^n\!M$, where $\,^n\!{\rm tr}$ is a normal, faithful trace on $\,^n\!M$ such that  $\,^n\!{\rm tr}(I)=1$. Let $\mathcal{S}$ be subset of the elements from $\,^n\!\mathcal{H}$. Denote by $\left[\mathcal{S} \right]$ the closure of the linear span of $\mathcal{S}$. Then $\left[ \,^n\!M\,^n\!\xi\right]=\left[\left( \,^n\!M\right)^\prime\,^n\!\xi\right]=\,^n\!\mathcal{H}$, where $\left( \,^n\!M\right)^\prime$ is commutant of $\,^n\!M$.
\begin{Prop}
For all $m_1\in\,^n\!\pi\left( B_n \right)^{\prime\prime}$ and $m_2\in\,^n\!\pi\left(B_{n\infty}\right)$ the equality $\,^n\!{\rm tr}(m_1\cdot m_2)=\,^n\!{\rm tr}(m_1)\cdot\,^n\!{\rm tr}(m_2)$ holds.
\end{Prop}
Let $t\in B_n$, $s\in B_{n\infty}$.
Set $\chi_n(t)=\,^n\!{\rm tr}\left(\,^n\!\pi(t)\right)$, $\chi_{n\infty}(s)=\,^n\!{\rm tr}\left(\,^n\!\pi(s)\right)$.
Since $\,^n\!\pi$ is a factor-representation, there exists the parameters $\alpha$, $\beta$, $\gamma=$$(\gamma_0$, $\gamma_1)$ such that $\chi_{n\infty}=\chi_{\alpha\beta\gamma}^\sigma$, where the last is defined by (\ref{II_1-mult_general}) and (\ref{character_formula}).
 For the same reason, there exists the irreducible representation $\mathfrak{Ir}_{_{\,^0\!\lambda\,^1\!\!\lambda}}$ ($\,^0\!\lambda\vdash k$, $\,^1\!\!\lambda\vdash(n-k)$) of $B_n$ such that $\left({\rm dim }\,\mathfrak{Ir}_{_{\,^0\!\lambda\,^1\!\!\lambda}}\right)\chi_n$ is the ordinary trace of $\mathfrak{Ir}_{_{\,^0\!\lambda\,^1\!\!\lambda}}$. We denote by $\pi_{\alpha\beta\gamma}^\sigma$ GNS-representation of $B_{n\infty}$ in the space $H_{\alpha\beta\gamma}^\sigma$ formed from character $\chi_{\alpha\beta\gamma}^\sigma$. Since, by the facts  from section \ref{irr_B_n}, the representation $\mathfrak{Ir}_{\,^0\!\lambda\,^1\!\!\lambda}$ is unitary equivalent to
${\rm Ind}^{B_n}_{G_{_{\Omega_{_{kn}}}}}\left(\,^{^{\Omega_{_{kn}}}}\!\!\rho\right)$ (see \ref{repr_rho}), we obtain the next statement.
\begin{Prop}\label{Prop_Ind}
Let $z\in \mathbb{Z}_2^n$, $s\in G_{\Omega_{_{kn}}}\cap\mathfrak{S}_n=\mathfrak{S}_k\cdot\mathfrak{S}_{kn}$ and $t\in B_{n\infty}$.
Define the representation of the group $H_{\infty\Omega_{_{kn}}}=G_{\chi_{_{kn}}}\cdot B_{n\infty}=B_k\cdot B_{kn}\cdot B_{n\infty}$ in the space $\mathcal{H}_{\,^0\!\lambda}\otimes \mathcal{H}_{\,^1\!\!\lambda}\otimes \mathcal{H}_{\alpha\beta\gamma}^\sigma$   as follows
\begin{eqnarray}\label{tilde_rho}
\Xi(zst)=\Omega_{_{kn}}(z)\left( {\rm Irr}_{\,^0\!\lambda}\otimes {\rm Irr}_{\,^1\!\!\lambda}  \right)(s)\otimes\pi_{\alpha\beta\gamma}^\sigma(t).
\end{eqnarray}
Then the representation $\,^n\!\pi$ is quasi-equivalent to ${\rm Ind}_{H_{\infty\Omega_{_{kn}}}}^{B_\infty^{(n)}}\,\Xi$.
\end{Prop}
\subsection{The results}
Let ${\rm Cl}(g)=\left\{ hgh^{-1}  \right\}_{h\in B}$ be the  conjugacy class of an element $g\in B$. We will denote by ${\rm Cl}_{n\infty}(g)$ the set ${\rm Cl}(g)\cap B_{n\infty}$. In Proposition \ref{asymp_char}  we introduced the notation of the asymptotical  character, which is an quasi-equivalence invariant of stable factor-representation.
\begin{Prop}
Take any sequence $g_n\in {\rm Cl}_{n\infty}(g)$.
If $\Pi$ is a stable factor-representation of $B$ then for any $g\in B$ there exists $w-\lim\limits_{n\to\infty}\Pi(g_n)$ in the weak operator topology, which is an scalar operator of the view $\chi_{_\Pi}^{as}(g)I$, where $\chi_{_\Pi}^{as}\in\mathbb{C}$. Besides this,  $\chi_{_\Pi}^{as}$
is an indecomposable character on $B$.
\end{Prop}
Let $\Pi$ be a stable factor-representation of $B$ in Hilbert space $\mathcal{H}_\Pi$, and let $M$ be $w^*$-algebra, generated by $\Pi(B)$. Denote by $M_*$ predual of $M$. Set $M_*^+=\left\{\varphi\in M_*: \varphi(m)\geq 0\right. $ for all nonnegative $\left. m\in M \right\}$. Take any $\varphi\in M_*^+$. The smallest orthogonal projection $E\in M$ such that $\varphi(Em)=\varphi(m)$ for all $m\in M$ is called the {\it support} of $\varphi$. Denote the support of $\varphi$ by ${\rm supp}\,\varphi$. The following fact we prove in section \ref{properties_of_stable_repr} (see Theorem \ref{induction_theorem} and Corollary \ref{collora_of_ind_Th}).
\begin{Th}\label{result_ind}
Let $M_*^+(n)=\left\{ \omega\in M_*^+\big|\omega\left( \Pi(g)\cdot a\cdot \Pi(g^{-1}) \right)=\omega\left( a \right)\text{ for all } \right.$ \newline $\left.g\in B_n\cdot B_{n\infty}\right.$ $\text{ and each }\left.  a\in M  \right\}$. Then the following hold:
     \begin{itemize}
       \item {\bf a}) there exists $n$ such that $M_*^+(n)\neq0$;
       \item {\bf b}) if ${\rm cd}\,\Pi=\min\left\{ n\big|M_*^+(n)\neq0\right\}$, then $\left({\rm supp}\,\omega \right)\cdot \pi(s)\cdot \left( {\rm supp}\,\omega \right)=0$ for all $s\notin B_{{\rm cd}\,\Pi}\cdot B_{({{\rm cd}\,\Pi})\;\infty}$ and each nonzero $\omega\in M_*^+({\rm cd}\,\Pi)$.
     \end{itemize}
\end{Th}
\begin{Co}
 Set $E={\rm supp}\,\varphi$. Then $E\in M\cap \Pi\left( B_{\infty}^{(n)}\right)^\prime$. Thus, the operators $\Pi_E(g)$ $=E\cdot\Pi(g)\cdot E$, where $g\in B_{\infty}^{(n)}= B_n\cdot B_{n\infty}$, generate the unitary finite type factor-representation of the group $B_{\infty}^{(n)}$ in the space $E\mathcal{H}_\Pi$. It follows from Theorem \ref{result_ind} that $\Pi$ induced  from the representation $\Pi_E$ of subgroup $B_{\infty}^{(n)}$; i. e. $\Pi={\rm Ind}_{B_{\infty}^{(n)}}^B \Pi_E$.
\end{Co}
We call ${\rm cd}\,\Pi$ the {\it central depth} of $\Pi$.
\begin{Rem}
Take  a finite type representation $\,^n\!\pi$ of  subgroup $B_{\infty}^{(n)}= B_n\cdot B_{n\infty}$. Using Definition \ref{stable_Def} of a stable representation and the construction of an induced representation, it is easy to check that $\Pi={\rm Ind}_{B_{\infty}^{(n)}}^B \,\,^n\!\pi$ is a stable representation.
\end{Rem}
In subsection \ref{Properties induced repr} we prove next statement.
\begin{Th}
If $\,^n\!\pi$ is a finite type factor-representation of subgroup $B_{\infty}^{(n)}$ then  ${\rm Ind}^B_{B_\infty^{(n)}}\,^n\!\pi$ is a stable factor-representation (see Theorem \ref{factor_ind}).
\end{Th}
Now we find the parameters $\alpha, \beta, \gamma, \sigma$ such that $\chi_{_\Pi}^{as}=\chi_{\alpha\beta\gamma}^\sigma$ (see \eqref{character_formula}).

Denote by $\,^n\!{\rm tr}$ a normal finite normalize trace on factor $\,^n\!M=\,^n\!\pi\left( B_\infty^{(n)}\right)^{\prime\prime}$. If $b_1\in B_n$ and $b_2\in B_{n\infty}$ then $\,^n\!{\rm tr}\left(\,^n\!\pi(b_1\,b_2) \right)=\,^n\!{\rm tr}\left(\,^n\!\pi(b_1) \right)\,\,^n\!{\rm tr}\left(\,^n\!\pi(b_2)\right)$.  Since $\,^n\!\pi$ is a finite type factor-representation of $B_\infty^{(n)}$, the function
\begin{eqnarray*}
B_n\ni b\stackrel{\chi_n}{\mapsto} \,^n\!{\rm tr}\left(\,^n\!\pi(b) \right)=\chi_n(b)\in \mathbb{C}~\text{ and }~ B_{n\infty}\ni b\stackrel{\chi_{n\infty}}{\mapsto} \,^n\!{\rm tr}\left(\,^n\!\pi(b) \right)=\chi_{n\infty}(b)\in \mathbb{C}
\end{eqnarray*}
are the indecomposable characters of $B_n$ and $B_{n\infty}$, respectively.
If $\check\alpha, \check\beta,\check\gamma,\check\sigma$ are the parameters from \eqref{character_formula}, corresponding to  $\chi_{n\infty}$, then it is easily to seen that $\check\alpha=\alpha$,  $\check\beta=\beta$,  $\check\gamma=\gamma$, $\check\sigma=\sigma$; i. e. the parameters of asymptotical character $\chi_{_\Pi}^{as}$ coincide with the parameters of character $\chi_{n\infty}(b)$.
In section \ref{Properties induced repr} we prove the following result (see Theorem \ref{II_1_theorem}).
\begin{Th}
Let $\chi_{\alpha\beta\gamma}^\sigma$ be the asymptotical character of the stable factor-representation $\Pi$. If $\gamma_0>0$ and $\gamma_1>0$ then $\Pi$ is ${\rm II}_1$-factor-representation of $B$; i. e. asymptotical character $\chi^{as}_{_\Pi}$ is a full invariant of the quasi-equivalence   in this case.
\end{Th}

\begin{Th}
Let $\Pi$ and $\breve\Pi$ be the stable factor-representations of $B$. Suppose that $\chi^{as}_{_\Pi}=\chi^{as}_{_{\breve\Pi}}=\chi_{\alpha\beta\gamma}^\sigma$ and $\sum\alpha_j+\sum\beta_j=1$ or,  what is the same, $\gamma_0=\gamma_1=0$. If $n\geq1$, $\Pi={\rm Ind}_{B_{\infty}^{(n)}}^B \,\,^n\!\pi$ and $\breve\Pi={\rm Ind}_{B_{\infty}^{(n)}}^B \,\,^n\!\breve\pi$, where $\,\,^n\!\pi$ and $\,^n\!\breve\pi$ are the finite type factor-representations of $B_{\infty}^{(n)}$,  then the following holds
\begin{itemize}
  \item i) ${\rm cd}\,\Pi={\rm cd}\,\breve\Pi=n$ (see Proposition  \ref{prop_cd});
  \item ii)  $\Pi$ and $\breve\Pi$ are semi finite factor-representation (see Remark \ref{one_dim} and Theorem \ref{alpha+beta=1});
  \item iii) in the case, when there exists $\alpha_j<1$ or $\beta_j<1$,  $\Pi$ and $\breve\Pi$ have ${\rm II}_\infty$-type;
\item vi) the representations $\Pi={\rm Ind}^B_{B_\infty^{(n)}}\,^n\!\pi$ and $\breve{\Pi}={\rm Ind}^B_{B_\infty^{(n)}}\,^n\!\breve{\pi}$ are quasi-equivalent if and only if the representations $\,^n\!\pi$ and $\,^n\!\breve{\pi}$ are quasi-equivalent (see Theorem \ref{quasi-equivalence _in_case_alpha+beta=1}).
\end{itemize}
\end{Th}
We recall the structure of irreducible representation of $B_n$ before considering the case $\gamma_0=0$, $\gamma_1>0$. Multiplicative character $\Omega_{kn}$ is defined on $z=(z_1,\ldots, z_k,\ldots,z_n)\in\mathbb{Z}_2^n$ as follows $\Omega_{kn}(z)=(-1)^{\sum_{j=k+1}^n z_j}$. The stabilizer of $\Omega_{kn}$ is Young subgroup $\mathfrak{S}_k\mathfrak{S}_{kn}\subset \mathfrak{S}_n$. Let ${\rm Irr}_{\,^0\!\lambda}$ and ${\rm Irr}_{\,^1\!\!\lambda}$ be the irreducible representations of $\mathfrak{S}_k$ and $\mathfrak{S}_{kn}$ in the spaces $\mathcal{H}_{\,^0\!\lambda}$ and $\mathcal{H}_{\,^1\!\!\lambda}$, respectively; ${\;^0\!\lambda}\vdash k$, ${\;^1\!\!\lambda}\vdash (n-k)$ are the corresponding partitions. Take $z\in \mathbb{Z}_2^n$,  $s_1\in \mathfrak{S}_n$, $s_2\in \mathfrak{S}_{kn}$. Define representation $\,^{^{\Omega_{_{kn}}}}\!\!\!\rho$ of $B_\infty^{(n)}$ in $\mathcal{H}_{\,^0\!\lambda}\otimes \mathcal{H}_{\,^1\!\lambda}$ by $\,^{^{\Omega_{_{kn}}}}\!\!\!\rho(zs_1s_2)=\Omega_{_{kn}}(z)\,{\rm Irr}_{\,^0\!\lambda}(s_1)\otimes{\rm Irr}_{\,^1\!\lambda}(s_2)$. Set  $\mathfrak{Ir}_{_{\,^0\!\lambda\,^1\!\!\lambda}}={\rm Ind}^{B_n}_{B_kB_{kn}}\left(\,^{^{\Omega_{_{kn}}}}\!\!\rho\right)$. Eve\-ry irreducible representation of $B_n$ is equivalent to $\mathfrak{Ir}_{_{\,^0\!\lambda\,^1\!\!\lambda}}$ for some ${\;^0\!\lambda}\vdash k$, ${\;^1\!\!\lambda}\vdash (n-k)$.

If $\,^n\!\pi$ is a finite type factor-representation of $B_\infty^{(n)}$ then $$\,^n\!\pi(b_1b_2)=\mathfrak{Ir}_{_{\,^0\!\lambda\,^1\!\!\lambda}}(b_1)\otimes\pi_{\alpha\beta\gamma}^\sigma (b_2)\;~\text{ for all }~ b_1\in B_n ~\text{ and }~b_2\in B_{n\infty},$$
where $\pi_{\alpha\beta\gamma}^\sigma$ is a representation of $B_{n\infty}$, corresponding to character $\chi_{\alpha\beta\gamma}^\sigma$.
Define  representation $\Xi$ of $B_k\cdot B_n\cdot B_{n\infty}$ as follows $$\Xi(zs_1s_2t)=\Omega_{_{kn}}(z) {\rm Irr}_{\,^0\!\lambda}(s_1)\otimes {\rm Irr}_{\,^1\!\!\lambda}(s_2)\otimes\pi_{\alpha\beta\gamma}^\sigma(t), ~\text{ where }~t\in B_{n\infty}.$$
Hence we have
\begin{eqnarray}
{\rm Ind}_{B_{\infty}^{(n)}}^B \,\,^n\!\pi={\rm Ind}_{B_kB_{kn}B_{n\infty}}^B \,\Xi ={\rm Ind}_{B_kB_{k\infty}}^B\,\left({\rm Ind}_{B_kB_{kn}B_{n\infty}}^{B_kB_{k\infty}}\,\Xi\right).
\end{eqnarray}
In section \ref{gamma>0} we proved the following statement (see Theorem \ref{Ind^B_infty_(k)}).
\begin{Th}
If $\gamma_1>0$ then  $\,^k\!\pi={\rm Ind}_{B_kB_{kn}B_{n\infty}}^{B_kB_{k\infty}}\,\Xi$ is $\rm{II}_1$-factor-representation of $B_\infty^{(k)}=B_kB_{k\infty}$. More precisely, $\,^k\!\pi$ is defined  on $b=zsb$, where $z\in \mathbb{Z}_2^k, s\in\mathfrak{S}_k$ and $b\in B_{k\infty}$, as follows $\,^k\!\pi(b)=\,{\rm Irr}_{\,^0\!\lambda}(s)\otimes\pi_{\alpha\beta\gamma}^\sigma$.
\end{Th}

\begin{Th}\label{finite factor}
Let $\Pi={\rm Ind}_{B_{\infty}^{(n)}}^B \,\,^n\!\pi$ and $\breve\Pi={\rm Ind}_{B_{\infty}^{(n)}}^B \,\,^n\!\breve\pi$ be the stable factor-representations of $B$, where $\,^n\!\pi$ and $\,^n\!\breve\pi$ are the finite type factor representations of $B_\infty^{(n)}$.   Suppose that $\chi^{as}_{_\Pi}=\chi^{as}_{_{\breve\Pi}}=\chi_{\alpha\beta\gamma}^\sigma$. If $\gamma_0=0$ and $\gamma_1>0$ then following holds:
\begin{itemize}
  \item 1) if $k\geq0$ then $\Pi$ and $\breve\Pi$ are ${\rm II}_\infty$-factor-representations (see Theorem \ref{II_infty_case});
  \item 2) the representations $\Pi$ and $\breve{\Pi}$ are quasi-equivalent if and only if the corresponding representations $\,^k\!\pi$ and $\,^k\!\breve{\pi}$, defined by Theorem \ref{finite factor}, are quasi-equivalent; more precisely, the asymptotical characters and corresponding to $\Pi$ and $\breve\Pi$   partitions $\,^0\!\lambda$ and $\,^0\!\breve\lambda$ are the full system of the quasi-equivalence  invariants (see Theorem \ref{quasi equivalence gamma=0 gamma=1}).
\end{itemize}
\end{Th}

\subsection{The properties of the representations ${\rm Ind}^B_{B_\infty^{(n)}}\,^n\!\pi$.}\label{Properties induced repr}
Here we study the representations of $B$ of the form ${\rm Ind}^B_{B_\infty^{(n)}}\,^n\!\pi$, where $\,^n\!\pi$ is the same as in section \ref{repr_subgroup}.
\begin{Th}\label{factor_ind}
${\rm Ind}^B_{B_\infty^{(n)}}\,^n\!\pi$ is a stable factor-representation.
\end{Th}
\begin{proof}
Without loss of generality we can to assume that $\,^n\!\pi$ is realized on $L^2\left( \,^n\!M, \,^n\!{\rm tr} \right)$ by the operators of left multiplication and $\,^n\!\xi$ is the identity element   of factor $\,^n\!M$; i.e. $\,^n\!\xi=I$.
Denote by $X$ the space of left cosets of the subgroup $B_\infty^{(n)}$ in $B$. It is easily seen that there exists a map $X\ni x\stackrel{\mathfrak{s}}{\mapsto}\mathfrak{s}(x)\in \mathfrak{S}_\infty\subset B$ with the properties $x=\mathfrak{s}(x)B_\infty^{(n)}$ and $\mathfrak{s}\left( B_\infty^{(n)} \right)=e$, where $e$ is identity  of $\mathfrak{S}_\infty$. Denote by $\widetilde{g}$ the class in $X$ containing $g\in B$.

The operators of the representation $\Pi={\rm Ind}^B_{B_\infty^{(n)}}\,^n\!\pi$ act on $v\in \mathcal{H}_\Pi=l^2\left(L^2\left( \,^n\!M, \,^n\!{\rm tr} \right),X\right)$ by the formula
\begin{eqnarray}\label{ind_action}
\left(\Pi(g)v\right)(x)=c(g,x)v\left(g^{-1}x\right),
\end{eqnarray}
where $c(g,x)= \,^n\!\pi\left(\left( \mathfrak{s}(x) \right)^{-1}g\;\mathfrak{s}\left(g^{-1}x\right)\right)$. Notice that $\left( \mathfrak{s}(x) \right)^{-1}g\;\mathfrak{s}\left(g^{-1}x\right)$ lies in $B_\infty^{(n)}$.

 Set
\begin{eqnarray}\label{trace_vector}
\xi(x)=\left\{
\begin{array}{rl}
I, \text{ if } x=\widetilde{e}\\
0,\text{ if } x\neq\widetilde{e}.
\end{array}
\right.
\end{eqnarray}
Since $\,^n\!\xi=I$ is cyclic in  $L^2\left( \,^n\!M, \,^n\!{\rm tr} \right)$, then  $\left[\Pi(B)\xi \right]=\mathcal{H}_\Pi$.

Every  $m^\prime\in\,^n\!M^\prime$ defines operator $\widetilde{m^\prime}\in \Pi(B)^\prime$ as follows
\begin{eqnarray}\label{const_commutant}
\left(\widetilde{m^\prime}\eta\right)(x)=m^\prime\eta(x).
\end{eqnarray}
Now we prove that any operator  $\mathcal{C}\in \left( \Pi(B) \right)^{\prime\prime}\cap\left( \Pi(B) \right)^{\prime}$ acts as a scalar operator (a scalar multiple of the identity).

Let $P_y$ be an orthogonal projection of  $l^2\left(L^2\left( \,^n\!M, \,^n\!{\rm tr} \right),X\right)$ onto subspace $\left\{v\in l^2\left(L^2\left( \,^n\!M, \,^n\!{\rm tr} \right),X\right): v(x)=0, \text{ if } x\neq y\right\}$. Set $\mathcal{C}_{xy}=P_x\mathcal{C}P_y$. Since $\,^n\!M$ is a factor and $\widetilde{m^\prime}$ lies in $\left( \Pi(B) \right)^{\prime}$ for every $m^\prime\in\,^n\!M^\prime$, it follows from  (\ref{ind_action}) that
\begin{eqnarray}\label{matrix_el_scalar}
\mathcal{C}_{\tilde{e}\tilde{e}}=c_{\tilde{e}\tilde{e}}P_{\tilde{e}\tilde{e}}, \text{ where } c_{\tilde{e}\tilde{e}}\in\mathbb{C}.
\end{eqnarray}
Define operator  $E_{yz}$ $(y,z\in X)$  on $v\in l^2\left(L^2\left( \,^n\!M, \,^n\!{\rm tr} \right),X\right)$ as follows
\begin{eqnarray*}
\left( E_{yz}v\right)(x)=\left\{
\begin{array}{rl}
v(z), \text{ if } x=y\\
0,\text{ if } x\neq y.
\end{array}
\right.
\end{eqnarray*}
Then $\mathcal{C}_{yz}=P_y\mathcal{C}P_z=m_{yz}E_{yz}$, where $m_{yz}$ is a bounded operator in $L^2\left( \,^n\!M, \,^n\!{\rm tr} \right)$. Hence, using (\ref{const_commutant}) and the equalities $\mathcal{C}\widetilde{m^\prime}=\widetilde{m^\prime}\mathcal{C}$, $m'\in\,^n\!M'$, we have
\begin{eqnarray*}
m_{yz}\,m'=m'\,m_{yz}~\text{for all }~ m'\in\,^n\!M' ~\text{ and }~ y,z\in X.
\end{eqnarray*}
Therefore,
\begin{eqnarray}
m_{yz}\in \,^n\!M ~\text{for all }~ y,z\in X.
\end{eqnarray}
Further, using the relation $\Pi(g)\mathcal{C}\Pi(g^{-1})v=\mathcal{C}v$, where $g\in B$, $v\in l^2\left(L^2\left( \,^n\!M, \,^n\!{\rm tr} \right),X\right)$, we obtain
\begin{eqnarray*}
c(g,x)\sum\limits_{y\in X} m_{\left( g^{-1}x \right)y}\;c\left(g^{-1},y\right)v(gy)=\sum\limits_{y\in X} m_{xy}v(y).
\end{eqnarray*}
Therefore, if $v(x)=0$ for $x\neq \tilde{e}$, then
\begin{eqnarray*}
c(g,x)\;m_{\left( g^{-1}x \right)\left(g^{-1}\tilde{e}\right)}\;c\left(g^{-1},g^{-1}\tilde{e}\right)\;
v(\tilde{e})=m_{x\tilde{e}}v(\tilde{e}).
\end{eqnarray*}
Hence, in the case $g\in B_\infty^{(n)}$, we have
\begin{eqnarray*}
c(g,x)\;m_{\left( g^{-1}x \right)\tilde{e}}\;\,^n\!\pi\left(g^{-1} \right)v(\tilde{e})=m_{x\tilde{e}}\;v(\tilde{e}).
\end{eqnarray*}
Thus we conclude that
\begin{eqnarray}\label{equality}
\left\| m_{x\tilde{e}}\;\xi(\tilde{e})\right\|_{L^2\left( \,^n\!M, \,^n\!{\rm tr} \right)}^2=\left\| m_{x\tilde{e}}\right\|_{L^2\left( \,^n\!M, \,^n\!{\rm tr} \right)}^2
=  \left\| m_{\left( g^{-1}x \right)\tilde{e}} \right\|_{L^2\left( \,^n\!M, \,^n\!{\rm tr} \right)}^2\;
\end{eqnarray}
for all  $g\in B_\infty^{(n)} $ and $ x\in X$.
Denote by $X/B_\infty^{(n)}$ the set of all orbits of $X$ under the action of $B_\infty^{(n)}$ on $X$:
\begin{eqnarray*}
X\ni x\mapsto tx\in X, t\in B_\infty^{(n)}.
\end{eqnarray*}
Define function $X/B_\infty^{(n)}\ni o\stackrel{\ell}{\mapsto}\ell(o)\in X$ such that $\ell(o)\in o$. Clearly, $B_\infty^{(n)}\;\ell(o)=o$. It is easy to check that
\begin{eqnarray}\label{infty}
\#o=\infty, \;\text{ if } \tilde{e}\notin o.
\end{eqnarray}
Since
\begin{eqnarray*}
\left( \mathcal{C}\xi,\mathcal{C}\xi \right)_{l^2\left(L^2\left( \,^n\!M, \,^n\!{\rm tr} \right),X\right)}=\sum\limits_{x\in X}\left( m_{x\tilde{e}}, m_{x\tilde{e}}\right)_{L^2\left( \,^n\!M, \,^n\!{\rm tr} \right)}\\
\stackrel{(\ref{matrix_el_scalar})}{=}\left| c_{\tilde{e}\tilde{e}}\right|^2+\sum\limits_{x\in X,x\neq\tilde{e}}\left\| m_{x\tilde{e}}\right\|^2_{L^2\left( \,^n\!M, \,^n\!{\rm tr} \right)}
\stackrel{(\ref{equality})}{=}\left| c_{\tilde{e}\tilde{e}}\right|^2\\
+\sum\limits_{o\in X/B_\infty^{(n)}, \tilde{e}\notin o}(\#o)\left\|m_{\left( \ell(o) \right)\tilde{e}}\;\right\|_{L^2\left( \,^n\!M, \,^n\!{\rm tr} \right)},
\end{eqnarray*}
then, using (\ref{infty}), we obtain that $\mathcal{C}_{x\tilde{e}}=0$ for all $x\neq \tilde{e}$. Therefore, $\mathcal{C}\xi=c_{\tilde{e}\tilde{e}}\xi$. Now our theorem follows from the cyclisity of vector $\xi$.
\end{proof}
\begin{Rem}
By the results of the section \ref{repr_subgroup}, representation $\,^n\!\pi$ has a form $\,^n\!\pi=\mathfrak{Ir}_{\,^0\!\lambda\,^1\!\!\lambda}\otimes\pi_{\alpha\beta\gamma}^\sigma$, where $\mathfrak{Ir}_{\,^0\!\lambda\,^1\!\!\lambda}$ is the irreducible representation of the finite group  $B_n$ and  $\pi_{\alpha\beta\gamma}^\sigma$ is the factor-representation of a finite type of the infinite group $B_{n\infty}$. We recall that character $\chi_{\alpha\beta\gamma}^\sigma$ of the representation $\pi_{\alpha\beta\gamma}^\sigma$ is defined by (\ref{II_1-mult_general}) and (\ref{character_formula}).
\end{Rem}
\begin{Th}\label{II_1_theorem}
If $\gamma_0>0$  and $\gamma_1>0$ then $\Pi={\rm Ind}^B_{B_\infty^{(n)}}\,^n\!\pi$ is   ${\rm II}_1$-factor-representation.
\end{Th}
We first prove  the auxiliary lemmas needed to prove theorem \ref{II_1_theorem}.

Let $\Pi_{\alpha\beta\gamma}^\sigma$ be a ${\rm II}_1$-factor-representation of $B$ in Hilbert space $\mathcal{H}_{\alpha\beta\gamma}^\sigma$ with the same parameters $\alpha$, $\beta$, $\gamma$ as $\pi_{\alpha\beta\gamma}^\sigma$. Denote by ${\rm tr}$ a normal trace on factor $M=\left(\Pi_{\alpha\beta\gamma}^\sigma(B)\right)^{''}$ such that ${\rm tr}(I)=1$. We assume below that $\mathcal{H}_{\alpha\beta\gamma}^\sigma=L^2(M,{\rm tr})$ and  $\Pi_{\alpha\beta\gamma}^\sigma$ acts by the  operators of the left multiplication.
\begin{Rem}Let $g,h\in B$.
  If $l,m\notin ({\rm supp}\,g)\cup ({\rm supp}\,h)$, then $${\rm tr}\left(\Pi_{\alpha\beta\gamma}^\sigma\left( h^* \,(k\;m)\, g\right)  \right)={\rm tr}\left(\Pi_{\alpha\beta\gamma}^\sigma\left( h^* \,(k\;l) g\right)  \right).$$
\end{Rem}
This implies
\begin{Lm}\label{asym_II_1}
For all $k\in\mathbb{N}$ the weak limits
\begin{eqnarray}\label{asymp_transposition}
\mathcal{O}_k=w-\lim\limits_{j\to\infty}\Pi_{\alpha\beta\gamma}^\sigma\left( (k\;j) \right)
\end{eqnarray}
exist and satisfy the relations $\mathcal{O}_k\mathcal{O}_l=\mathcal{O}_l\mathcal{O}_k$ and $\Pi_{\alpha\beta\gamma}^\sigma(s)\mathcal{O}_k\Pi_{\alpha\beta\gamma}^\sigma(s^{-1})=\mathcal{O}_{s(k)}$, $s\in\mathfrak{S}_\infty$.
\end{Lm}
Denote by $\mathfrak{A}_k$ the $w^*$-algebra generated by $\mathcal{O}_k$ and $\Pi_{\alpha\beta\gamma}^\sigma\left( 1^{(k)} \right)$, where $1^{(k)}=(\underbrace{0,0,\ldots,0}_{k-1},1,0,\ldots)$.

 Next statement is a particular case of the general lemma 15 from \cite{DN}.
 \begin{Lm}\label{lemma_commutative}
 $w^*$-algebra $\mathfrak{A}=\left( \left\{\mathfrak{A}_k\right\}_{k=1}^\infty \right)^{''}$ is commutative.
 \end{Lm}
\begin{proof}
For completeness, we give below the proof of this lemma.

Denote operator $\Pi_{\alpha\beta\gamma}^\sigma\left( 1^{(j)} \right)$ briefly  by $A_j$. Clearly,
\begin{eqnarray*}
\mathcal{O}_kA_j=A_j\mathcal{O}_k \;\text{ for all }\; j\neq k.
\end{eqnarray*}
Therefore,  to prove our lemma, we must show that
\begin{eqnarray}\label{commutative_relation}
\mathcal{O}_kA_k=A_k\mathcal{O}_k.
\end{eqnarray}
To do this, take an arbitrary pair of the elements $g,h\in B$. Now (\ref{commutative_relation}) follows from the next chain of the equalities
\begin{eqnarray*}
\left(\mathcal{O}_k A_k \Pi_{\alpha\beta\gamma}^\sigma(g),\Pi_{\alpha\beta\gamma}^\sigma(h) \right)_{L^2(M,{\rm tr})}={\rm tr}\left( \Pi_{\alpha\beta\gamma}^\sigma(h^{-1})\; \mathcal{O}_k \,A_k\; \Pi_{\alpha\beta\gamma}^\sigma(g)\right)
\end{eqnarray*}
\begin{eqnarray*}
=\lim\limits_{j\to\infty}{\rm tr}\left( \Pi_{\alpha\beta\gamma}^\sigma(h^{-1})\; \Pi_{\alpha\beta\gamma}^\sigma((k\;j)) \,A_k\; \Pi_{\alpha\beta\gamma}^\sigma(g)\right)
\end{eqnarray*}
\begin{eqnarray*}
\stackrel{\rm Lemma \ref{asym_II_1}}{=}\lim\limits_{j\to\infty}{\rm tr}\left( \Pi_{\alpha\beta\gamma}^\sigma(h^{-1})\; \,A_j\; \Pi_{\alpha\beta\gamma}^\sigma((k\;j))\;\Pi_{\alpha\beta\gamma}^\sigma(g)\right)
\end{eqnarray*}
\begin{eqnarray*}
\stackrel{\rm Lemma \ref{asym_II_1}}{=}\lim\limits_{j\to\infty}{\rm tr}\left( \Pi_{\alpha\beta\gamma}^\sigma(h^{-1})\; \; \Pi_{\alpha\beta\gamma}^\sigma((k\;j))\;A_j\;\Pi_{\alpha\beta\gamma}^\sigma(g)\right)
\end{eqnarray*}
\begin{eqnarray*}
\stackrel{\rm Lemma \ref{asym_II_1}}{=}\lim\limits_{j\to\infty}{\rm tr}\left( \Pi_{\alpha\beta\gamma}^\sigma(h^{-1})\; \;A_k\; \Pi_{\alpha\beta\gamma}^\sigma((k\;j))\;\Pi_{\alpha\beta\gamma}^\sigma(g)\right)
\end{eqnarray*}
\begin{eqnarray*}
\stackrel{\rm Lemma \ref{asym_II_1}}{=}{\rm tr}\left( \Pi_{\alpha\beta\gamma}^\sigma(h^{-1})\; \;A_k\; \mathcal{O}_k\;\Pi_{\alpha\beta\gamma}^\sigma(g)\right)=\left( A_k\; \mathcal{O}_k\;\Pi_{\alpha\beta\gamma}^\sigma(g), \Pi_{\alpha\beta\gamma}^\sigma(h)\right)_{L^2(M,{\rm tr})}.
\end{eqnarray*}
\end{proof}
\begin{Co}\label{Collorary19}
  Let $\mathbb{S}$ be a subset of $\mathbb{N}$, and let $\mathfrak{A}_\mathbb{S}$ be a $w^*$-algebra generated by $\left\{\mathfrak{A}_j \right\}_{j\in\mathbb{S}}$. Take $sz\in B$, where $s\in\mathfrak{S}_\infty$ and $z\in \;_0\mathbb{Z}_2^\infty$, such that ${\rm supp}\,s\in\mathbb{N}\setminus\mathbb{S}$. Then $\Pi_{\alpha\beta\gamma}^\sigma(sz)\;a\;\Pi_{\alpha\beta\gamma}^\sigma\left( zs^{-1} \right)=a$ for all $a\in\mathfrak{A}_\mathbb{S}$.
\end{Co}
Set  $\mathfrak{S}_{n\infty}=\left\{s\in\mathfrak{S}_\infty: s(j)=j\;\text{ for all }\;j\leq n \right\}$.
Fix  $D_j\in\mathfrak{A}_j$ such that
 \begin{eqnarray}\label{normalise}
 {\rm tr}(D_j)=1.
 \end{eqnarray}
 We conclude from lemma \ref{lemma_commutative} that
 \begin{eqnarray*}
 D_jD_k=D_kD_j\;\text{ for all } j,k.
 \end{eqnarray*}
 Let $\,^n\!D=\prod\limits_{j=1}^n D_j$. Define functional $\,^n\!\varphi$ on ${\rm II}_1$-factor $M$ as follows
 \begin{eqnarray}\label{definition_varphi}
 \,^n\!\varphi\left( m \right)={\rm tr}\left(\,^n\!D\;m\right),\;\;m\in M.
 \end{eqnarray}
 For further convenience, we introduce the axillary sequence $\left\{\,^n\!D_j \right\}_{j=1}^\infty$ as follows
\begin{eqnarray*}
\,^n\!D_j =\left\{
\begin{array}{rl}
D_j, \text{ if } j\leq n\\
I,\text{ if } j>n.
\end{array}
\right.
\end{eqnarray*}
  $\,^n\!D=\prod\limits_{j=1}^\infty \;^n\!D_j$. An important observation here is
 \begin{Lm}\label{multiplicative_lemma}
  If $g,h\in B$ and $\left( {\rm supp}\,g \right)\cap \left( {\rm supp}\,h \right)=\emptyset$, then
  \begin{eqnarray*}
  \,^n\!\varphi\left( \Pi_{\alpha\beta\gamma}^\sigma(gh) \right)=\,^n\!\varphi\left( \Pi_{\alpha\beta\gamma}^\sigma(g) \right)\;\,^n\!\varphi\left( \Pi_{\alpha\beta\gamma}^\sigma(h) \right).
  \end{eqnarray*}
 \end{Lm}
 \begin{proof}
 Let $\left\{n_k \right\}_{k\in\mathbb{N}}$ be a sequence of the natural numbers such that
 \begin{itemize}
   \item {\bf 1}) $n_{k+1}>n_k$ for all $k\in\mathbb{N}$;
   \item {\bf 2}) $n_k>\max\left\{j\in \left( {\rm supp}\,g \right)\cup \left( {\rm supp}\,h \right) \right\}$ for all $k\in\mathbb{N}$.
 \end{itemize}
 Set $\mathbb{N}_n=\left\{1,2,\ldots,n \right\}$, $\mathbb{E}=({\rm supp}\,h)\cup\left\{j\leq n:j\notin{\rm supp}\,g \right\}$.
 Since $\left( {\rm supp}\,g \right)\cap \left( {\rm supp}\,h \right)=\emptyset$, there exists a sequence $\left\{s_k \right\}\in \mathfrak{S}_\infty$ with the properties
 \begin{eqnarray}\label{sequence_s_k}
 \begin{split}
 \min\left\{j\in s_k(\mathbb{E}) \right\}>n_k \;\text{ for all }\; k;\\
 s_k(q)=q  \;\text{ for all }\;k \;\text{ and }\;q\in\, {\rm supp}\,g.
 \end{split}
 \end{eqnarray}
 In particular, $s_k({\rm supp}\,h)\cap\,{\rm supp}\,g=\emptyset$.
Hence, using lemma \ref{lemma_commutative}, corollary \ref{Collorary19} and (\ref{sequence_s_k}), we have
\begin{eqnarray*}
  \,^n\!\varphi\left( \Pi_{\alpha\beta\gamma}^\sigma(gh) \right)={\rm tr}\left(\Pi_{\alpha\beta\gamma}^\sigma(s_k)\;  \,^n\!D\; \Pi_{\alpha\beta\gamma}^\sigma(gh)\; \Pi_{\alpha\beta\gamma}^\sigma(s_k^{-1})  \right)\\
 ={\rm tr}\left(\left( \prod\limits_{(j\in{\rm supp}\,g )}\,^n\!D_j\right)\Pi_{\alpha\beta\gamma}^\sigma(g)\;\left( \prod\limits_{j\in s_k(\mathbb{N}_n\cap\,{\rm supp}\,h )}\,D_j\right) \;\Pi_{\alpha\beta\gamma}^\sigma(s_khs_k^{-1})\;R_{k,g,h}\right),
\end{eqnarray*}
where $R_{k,g,h}=\prod\limits_{j\in s_k\left(\mathbb{N}_n\setminus(({\rm supp}\,g)\cup \,{\rm supp}\,h) \right)}\,D_j$.
  Applying property {\bf 1}) and first inequality from  (\ref{sequence_s_k}), we obtain that the sequence $$A_k=\left( \prod\limits_{j\in s_k(\mathbb{N}_n\cap\,{\rm supp}\,h )}\,^n\!D_j\right) \;\Pi_{\alpha\beta\gamma}^\sigma(s_khs_k^{-1})\;R_{k,g,h}$$ tends in the weak operator topology to a scalar operator $cI$, where $$c={\rm tr}\left( \left( \prod\limits_{j\in {\rm supp}\,h}\,^n\!D_j\right) \;\Pi_{\alpha\beta\gamma}^\sigma(h)\;R_{g,h} \right),\;\; R_{g,h}=\prod\limits_{j\notin\left({\rm supp}\,g\right)\cup \left( {\rm supp}\,h \right)}\,^n\!D_j.$$
 Therefore,
 \begin{eqnarray*}
  \,^n\!\varphi\left( \Pi_{\alpha\beta\gamma}^\sigma(gh) \right)={\rm tr}\left( \left( \prod\limits_{j\in{\rm supp}\,g }\,^n\!D_j\right)\Pi_{\alpha\beta\gamma}^\sigma(g)\right)
  {\rm tr}\left( \left( \prod\limits_{j\in {\rm supp}\,h}\,^n\!D_j\right) \;\Pi_{\alpha\beta\gamma}^\sigma(h)\;R_{g,h} \right).
 \end{eqnarray*}
 By the same reason,
 \begin{eqnarray*}
 {\rm tr}\left( \left( \prod\limits_{j\in {\rm supp}\,h}\,^n\!D_j\right) \;\Pi_{\alpha\beta\gamma}^\sigma(h)\;R_{g,h} \right)={\rm tr}\left( \left( \prod\limits_{j\in {\rm supp}\,h}\,^n\!D_j\right) \;\Pi_{\alpha\beta\gamma}^\sigma(h) \right)\;{\rm tr}\left( R_{g,h} \right)\\
 ={\rm tr}\left( \left( \prod\limits_{j\in {\rm supp}\,h}\,^n\!D_j\right) \;\Pi_{\alpha\beta\gamma}^\sigma(h) \right)\;
 \prod\limits_{j\notin\left({\rm supp}\,g\right)\cup \left( {\rm supp}\,h \right)}{\rm tr}\left( \,^n\!D_j \right).
 \end{eqnarray*}
 Now we conclude from (\ref{normalise})  that
 \begin{eqnarray*}
   \,^n\!\varphi\left( \Pi_{\alpha\beta\gamma}^\sigma(gh) \right)={\rm tr}\left( \left( \prod\limits_{j\in{\rm supp}\,g }\,^n\!D_j\right)\Pi_{\alpha\beta\gamma}^\sigma(g)\right)
  {\rm tr}\left( \left( \prod\limits_{j\in {\rm supp}\,h}\,^n\!D_j\right) \;\Pi_{\alpha\beta\gamma}^\sigma(h) \right).
 \end{eqnarray*}
 Analogously, since
 \begin{eqnarray*}
 {\rm tr}\left( \left( \prod\limits_{(j\in{\rm supp}\,g )}\,^n\!D_j\right)\Pi_{\alpha\beta\gamma}^\sigma(g)\right)=\prod\limits_{(j\notin{\rm supp}\,g )}{\rm tr}\left(  \,^n\!D_j \right)\;{\rm tr}\left( \left( \prod\limits_{(j\in{\rm supp}\,g )}\,^n\!D_j\right)\Pi_{\alpha\beta\gamma}^\sigma(g)\right)\\
 ={\rm tr}\left( \,^n\!D \;\Pi_{\alpha\beta\gamma}^\sigma(g)\right)=\,^n\!\varphi\left(\Pi_{\alpha\beta\gamma}^\sigma(g)\right),
 \end{eqnarray*}
 then $ \,^n\!\varphi\left( \Pi_{\alpha\beta\gamma}^\sigma(gh) \right)=\,^n\!\varphi\left( \Pi_{\alpha\beta\gamma}^\sigma(g) \right)\;\,^n\!\varphi\left( \Pi_{\alpha\beta\gamma}^\sigma(h) \right)$.
 \end{proof}

\begin{Lm}\label{Main_lemma21}
  Let $q=cz$ be a quasi-cycle, where  $c=(n_1\;n_2\;n_3\;\ldots\;n_k)$ and $z=(z_1,z_2,z_3,\ldots)\in\;_0\mathbb{Z}_2^\infty$. Define $\tilde{z}=(\tilde{z}_1,\tilde{z}_2,\tilde{z}_3,\ldots)\in\;_0\mathbb{Z}_2^\infty$ as follows
   \begin{eqnarray*}\tilde{z}_j =\left\{
\begin{array}{rl}
z_j, &\text{ if } (j\neq n_2)\&(j\neq n_1);\\
0,&\text{ if } j=n_2;\\
z_{n_1}+z_{n_2}, &\text{ if } j=n_1.
\end{array}
\right.
\end{eqnarray*}
   Suppose that $n_1\leq n$.
  \begin{itemize}
    \item {\bf a}) If $n_2>n$ then $ \,^n\!\varphi\left(\Pi_{\alpha\beta\gamma}^\sigma(q)\right)={\rm tr}\left( \mathcal{O}_{n_1}\;^n\!D\,\Pi_{\alpha\beta\gamma}^\sigma\left(\tilde{c}\tilde{z}\right) \right)$, where $\tilde{c}=(n_1\;n_3\;\ldots\;n_k)$, $\tilde{z}=(\tilde{z}_1,\tilde{z}_2,\tilde{z}_3,\ldots)$
    \item {\bf b}) if $n_2\leq n$  then $\,^n\!\varphi\left(\Pi_{\alpha\beta\gamma}^\sigma(q)\right)={\rm tr}\left(\mathcal{O}_{n_1}D_{n_1}\;\left(\prod\limits_{j\neq n_2}\,^n\!D_j\right)  \;\Pi_{\alpha\beta\gamma}^\sigma(\tilde{c}\tilde{z}) \right)$.
  \end{itemize}
\end{Lm}
\begin{proof}
For simplicity of notation, we will write $\Phi(g)$ instead of  $\,^n\!\varphi\left(\Pi_{\alpha\beta\gamma}^\sigma(g)\right)$, $g\in B$. Let $z=\left( z_1, z_2, \ldots \right)$.

 Without loss of generality, we may assume that $z_{n_2}=0$. In opposite case we can to consider element
 $(\,^{n_2}\!z)\;sz \;(\,^{n_2}\!z)=sz^\prime$ instead $sz$, where $\,^{n_2}\!z=\left(\underbrace{0,\ldots,0}_{n_2-1},z_{n_2},0\ldots  \right)$, $z^\prime=\left( z^\prime_1, z^\prime_2,\ldots \right)$ with $z^\prime_{n_2}=0$, $z^\prime_{n_1}=z_{n_1}+z_{n_2}$. Indeed, then, by lemma \ref{lemma_commutative} and (\ref{definition_varphi}),
 \begin{eqnarray*}
 \Phi(cz)=\Phi(cz^\prime).
 \end{eqnarray*}
Let us prove {\bf a}). Since $c=(n_1\;n_k)(n_1\;n_{k-1})\cdots (n_1\;n_2)$, we have
\begin{eqnarray}
\Phi(cz)={\rm tr}\left( \;^n\!D\;\Pi_{\alpha\beta\gamma}^\sigma(\tilde{c})\;\Pi_{\alpha\beta\gamma}^\sigma((n_1\;n_2))\;\Pi_{\alpha\beta\gamma}^\sigma(z) \right).
\end{eqnarray}
Hence, applying corollary \ref{Collorary19}, we obtain
\begin{eqnarray*}
\Phi(cz)=\lim\limits_{N\to\infty}{\rm tr}\left( \Pi_{\alpha\beta\gamma}^\sigma((n_2\;N))\; ^n\!D\;\Pi_{\alpha\beta\gamma}^\sigma(\tilde{c})\;\Pi_{\alpha\beta\gamma}^\sigma((n_1\;n_2))\;\Pi_{\alpha\beta\gamma}^\sigma(z)\; \Pi_{\alpha\beta\gamma}^\sigma((n_2\;N)) \right)\\
\stackrel{(n_2>n)\&(z_{n_2}=0)}{=}\lim\limits_{N\to\infty}{\rm tr}\left(\;^n\!D\;\Pi_{\alpha\beta\gamma}^\sigma(\tilde{c})\;\Pi_{\alpha\beta\gamma}^\sigma((n_1\;N))\;\Pi_{\alpha\beta\gamma}^\sigma(z)\right)\\
\stackrel{(\ref{asymp_transposition})}{=}{\rm tr}\left(\;^n\!D\;\Pi_{\alpha\beta\gamma}^\sigma(\tilde{c})\;\mathcal{O}_{n_1}\;\Pi_{\alpha\beta\gamma}^\sigma(z)\right)\stackrel{\text{Lemma \ref{lemma_commutative}}}{=} {\rm tr}\left( \mathcal{O}_{n_1}\;^n\!D\,\Pi_{\alpha\beta\gamma}^\sigma\left(\tilde{c}z\right) \right).
\end{eqnarray*}
To prove {\bf b}), we notice that
\begin{eqnarray*}
\Phi(cz)={\rm tr}\left(\;^n\!D \;\Pi_{\alpha\beta\gamma}^\sigma(\tilde{c})\;\Pi_{\alpha\beta\gamma}^\sigma((n_1\;n_2))\;\Pi_{\alpha\beta\gamma}^\sigma(z)\right)\\
={\rm tr}\left( \left(\prod\limits_{j\neq n_2}\,^n\!D_j\right)\Pi_{\alpha\beta\gamma}^\sigma(\tilde{c})\;\Pi_{\alpha\beta\gamma}^\sigma((n_1\;n_2))\;D_{n_1}\;\Pi_{\alpha\beta\gamma}^\sigma(z) \right)\\
={\rm tr}\left( \left(\prod\limits_{j\neq n_2}\,^n\!D_j\right)\Pi_{\alpha\beta\gamma}^\sigma(\tilde{c})\;\mathcal{O}_{n_1}\;D_{n_1}\;\Pi_{\alpha\beta\gamma}^\sigma(z) \right)\\
={\rm tr}\left( \mathcal{O}_{n_1}\;D_{n_1}\;\left(\prod\limits_{j\neq n_2}\,^n\!D_j\right)\Pi_{\alpha\beta\gamma}^\sigma(\tilde{c}z)\right).
\end{eqnarray*}
\end{proof}
The next lemma is the special case of lemma \ref{Main_lemma21}.
\begin{Lm}\label{lemma_character_formula}
  Let $q=cz$ be a quasi-cycle, where $c=(1\;2\;\ldots\;p)\in\mathfrak{S}_\infty$ and $z=(z_1,\ldots,z_p,0,0,\ldots)\in\;_0\mathbb{Z}_2^\infty$. Set $\zeta=\left(\zeta_1,0,0,\ldots\right)\in\;_0\mathbb{Z}_2^\infty$, where $\zeta_1=\sum\limits_{j=1}^p z_j$. Then
  \begin{eqnarray*}
  {\rm tr}\left( \Pi_{\alpha\beta\gamma}^\sigma(q) \right)={\rm tr}\left( \mathcal{O}_1^{p-1} \;\Pi_{\alpha\beta\gamma}^\sigma(\zeta) \right).
  \end{eqnarray*}
\end{Lm}

Analysis similar to that in the proof of Theorems 1 and 2 from \cite{Ok2} implies
\begin{Lm}\label{spectral_decomposition}
 Let $S(\mathcal{O}_k)$ be a spectrum of the operator  $\mathcal{O}_k$ and let $\mu$ be a the spectral measure, corresponding to ${\rm tr}$. Then we conclude:
 \begin{itemize}\label{spectral_decomposition_Okounkov}
   \item {\bf 1}) the measure $\mu$ is discrete and its atoms can only accumulate to zero;
   \item  {\bf 2}) if $\;\mathcal{O}_k=\sum\limits_{\lambda\in S\left(\mathcal{O}_k \right)} \lambda E_k(\lambda)$ is the spectral decomposition of  $\mathcal{O}_k$ then ${\rm tr}\left(E_k(\lambda) \right)=m(\lambda)\cdot|\lambda|$, where $m(\lambda)\in\mathbb{N}\cup0$;
 \item   {\bf 3}) the values of  Thoma parameters belong  to  $S(\mathcal{O}_k)$;
   \item   {\bf 4}) if $\lambda\in S(\mathcal{O}_k)$ is positive  (negative) then there exists Thoma parameter such that   $\lambda=\alpha_k$ $\left(\lambda=-\beta_k \right)$ and $m(\lambda)=\#\left\{k:\alpha_k=\lambda \right\}\geq 1\;\;\left(\#\left\{k:-\beta_k=\lambda \right\}\geq 1 \right)$.
 \end{itemize}
  \end{Lm}
  Now we recall that, by Lemma \ref{lemma_commutative}, $w^*$-algebra $\mathfrak{A}_k$, generated  by $\mathcal{O}_k$ and $\Pi_{\alpha\beta\gamma}^\sigma\left( 1^{(k)} \right)$, is abelian.
 Next assertion was proved in \cite{DN} (Lemma 23).
 \begin{Lm}\label{integer_wreath_product}
 Let $P$ be a projection from $\mathfrak{A}_k$. If $\lambda$ lies in $S(\mathcal{O}_k)$ and $\lambda\neq0$ then ${\rm tr}\left( P\; E_k(\lambda)\right)=|\lambda|\nu_\lambda(P)$, where $\nu_\lambda(P)\in \mathbb{N}$.
\end{Lm}
\subsubsection{The proof of theorem \ref{II_1_theorem}.}
By lemma \ref{spectral_decomposition_Okounkov}, ${\rm tr}\left( E_j(0) \right)=\gamma_0+\gamma_1>0$. It follows from lemma \ref{lemma_commutative} that there exists a self-adjoint projection $F_j(\lambda)\in\mathfrak{A}_j$ with the properties: $F_j(\lambda)\leq E_j(\lambda)$ and
\begin{eqnarray}\label{E_k_lambda_z}
E_j(\lambda)\;\Pi_{\alpha\beta\gamma}^\sigma\left( \zeta^{(j)} \right)=F_j(\lambda)+(-1)^{\zeta_j}\,F^\perp_j(\lambda),
\end{eqnarray}
where $ F^\perp_j(\lambda)=E_j(\lambda)-F_j(\lambda)$, $\zeta^{(j)}=(\underbrace{0,0,\ldots,0}_{j-1},\zeta_j,0,\ldots)\in\;_0\mathbb{Z}_2^{\infty}$. Equality (\ref{E_k_lambda_z}) can be obtained also, using the evident realizations of ${\rm II}_1$-factor-representations of $B$ from \cite{DN}.

 To prove our theorem it is sufficiently to find nonnegative $D\in M$ with the property
\begin{eqnarray}\label{main_equality}
{\rm tr}\left( D \Pi_{\alpha\beta\gamma}^\sigma(b)\right)=\left( \Pi(b)\xi,\xi \right)\;\text{ for all }b\in B,
\end{eqnarray}
where $\xi$ is defined in (\ref{trace_vector}).

Indeed, if $\mathfrak{R}$ is GNS-representation of ${\rm II}_1$-factor $M$, corresponding to state ${\rm tr}\left( D \cdot\right)$, then the representations $\Pi_{\alpha\beta\gamma}^\sigma$ and $\mathfrak{R}\circ \Pi_{\alpha\beta\gamma}^\sigma$ are quasi-equivalent. But, by (\ref{main_equality}), $\Pi$  and $\mathfrak{R}\circ \Pi_{\alpha\beta\gamma}^\sigma$  are unitary equivalent. Therefore, $\Pi$ is  ${\rm II}_1$-factor-representation.

Let $\omega\in\{0,1\}=\mathbb{Z}_2$.
At first, we show that for $\lambda \neq 0$ the next equality holds
\begin{eqnarray}\label{comparing_character_formula}
\begin{split}
&|\lambda|^{-1}\left({\rm tr}\left( F_j(\lambda) \right)+(-1)^\omega{\rm tr}\left(F^\perp_j(\lambda)\right)\right)\;\;\;\;\\
=&
\left\{
\begin{array}{rl}
\#\left\{\alpha_j\in\sigma^{-1}(0):\alpha_j=\lambda\right\}+(-1)^\omega\#\left\{\alpha_j\in\sigma^{-1}(1):\alpha_j
=\lambda\right\},\text{ if } \lambda>0;\\
\#\left\{\beta_j\in\sigma^{-1}(0):\beta_j=|\lambda|\right\}+(-1)^\omega\#\left\{\beta_j\in\sigma^{-1}(1):\beta_j
=|\lambda|\right\},\text{ if } \lambda<0.\;\;
\end{array}
\right.
\end{split}
\end{eqnarray}
Let $q$ and $\zeta$ be the same as in lemma \ref{lemma_character_formula}.
 Applying (\ref{E_k_lambda_z}), lemmas \ref{lemma_character_formula} and \ref{spectral_decomposition}, we have for $p>1$
 \begin{eqnarray}\label{spectral_formula_CHARACTER}
 \chi_{\alpha\beta\gamma}^\sigma(q)={\rm tr}\left(\Pi_{\alpha\beta\gamma}^\sigma(q)\right)=
 \sum\limits_{\left\{\lambda\in\mathcal{S}(\mathcal{O}_j):\lambda\neq 0   \right\}}\lambda^{p-1}\left(  {\rm tr}(F_j(\lambda))+(-1)^{\zeta_j}\;{\rm tr}\left(F_j^\perp(\lambda) \right) \right).\;\;\;\;
 \end{eqnarray}
 Now, using lemma \ref{spectral_decomposition} again and lemma \ref{integer_wreath_product}, we can to rewrite the right  part of (\ref{character_formula}) as follows
 \begin{eqnarray}\label{alpha_beta_formula_character}
 \begin{split}
 \chi_{\alpha\beta\gamma}^\sigma(q)=\sum\limits_{\lambda\in S(\mathcal{O}_j):\lambda>0}\; \left(\#\left\{ \alpha_j\in \sigma^{-1}(0):\alpha_j=\lambda\right\}\right.\\
 +\left.(-1)^{\zeta_j}\;\#\left\{ \alpha_j\in \sigma^{-1}(1):\alpha_j=\lambda\right\}\right)\lambda^p\\
 +(-1)^{p-1}\sum\limits_{\lambda\in S(\mathcal{O}_j):\lambda<0}\; \left(\#\left\{ \beta_j\in \sigma^{-1}(0):-\beta_j=\lambda\right\}\right.\\
 +\left.(-1)^{\zeta_j}\;\#\left\{ \beta_j\in \sigma^{-1}(1):-\beta_j=\lambda\right\}\right)|\lambda|^p.
\end{split}
 \end{eqnarray}
 Comparing (\ref{spectral_formula_CHARACTER}) and (\ref{alpha_beta_formula_character}), we obtain (\ref{comparing_character_formula}).

 Therefore, for $\lambda\neq 0$, we have
 \begin{eqnarray}\label{trace_F_k}
{\rm tr }(F_j(\lambda)) =
\left\{
\begin{array}{rl}
\lambda\;\#\left\{\alpha_j\in\sigma^{-1}(0):\alpha_j=\lambda\right\},\text{ if } \lambda>0;\\
|\lambda|\;\#\left\{\beta_j\in\sigma^{-1}(0):\beta_j=|\lambda|\right\},\text{ if } \lambda<0,
\end{array}
\right.
 \end{eqnarray}
 \begin{eqnarray}\label{trace_F_k_perp}
 {\rm tr }(F^\perp_j(\lambda)) =
\left\{
\begin{array}{rl}
\lambda\;\#\left\{\alpha_j\in\sigma^{-1}(1):\alpha_j=\lambda\right\},\text{ if } \lambda>0;\\
|\lambda|\;\#\left\{\beta_j\in\sigma^{-1}(1):\beta_j=|\lambda|\right\},\text{ if } \lambda<0.
\end{array}
\right.
 \end{eqnarray}
 In the case $p=1$, we notice that
 \begin{eqnarray*}
 \begin{split}
 \chi_{\alpha\beta\gamma}^\sigma(q)={\rm tr}\left(\Pi_{\alpha\beta\gamma}^\sigma(q)\right)={\rm tr}\left(F_j(0)\right)+(-1)^{\zeta_j}\,{\rm tr}\left(F^\perp_j(0)\right)\\ +\sum\limits_{\left\{\lambda\in\mathcal{S}(\mathcal{O}_j):\lambda\neq 0   \right\}}\left(  {\rm tr}(F_j(\lambda))+(-1)^{\zeta_j}\;{\rm tr}\left(F_j^\perp(\lambda) \right) \right).
 \end{split}
 \end{eqnarray*}
 Hence, applying (\ref{character_formula}), (\ref{trace_F_k}) and (\ref{trace_F_k_perp}), we obtain
 \begin{eqnarray*}
 {\rm tr}\, \left(E_j(0)\, \Pi_{\alpha\beta\gamma}^\sigma(q)\right)={\rm tr}\,\left( F_j(0) \right)+(-1)^{\zeta_j}{\rm tr}\,\left( F_j^\perp(0) \right)=\gamma_0+(-1)^{\zeta_j}\gamma_1.
 \end{eqnarray*}
 Therefore,
 \begin{eqnarray}\label{normalization}
 {\rm tr}\,\left( F_j(0) \right)=\gamma_0 ~\text{ and }~ {\rm tr}\,\left( F_j^\perp(0) \right)=\gamma_1.
 \end{eqnarray}
We recall that $\Pi_{\alpha\beta\gamma}^\sigma$  acts in $\mathcal{H}_{\alpha\beta\gamma}^\sigma=L^2(M,{\rm tr})$ by the  operators of the left multiplication, where ${\rm tr}$ is a normal trace on factor $M=\left(\Pi_{\alpha\beta\gamma}^\sigma(B)\right)^{''}$ such that ${\rm tr}(I)=1$.

Let $k$ be the same as in section \ref{irreducible_repr_of_B_n}. Define sequence $D_j$ as follows  $$D_j=\left\{
\begin{array}{rl}
F_j(0), &\text{ if } j\leq k;\\
F_j^\perp(0),&\text{ if } j=k+1,k+2,\ldots n;\\
I, &\text{ if } j>n.
\end{array}
\right.$$
 Take an element $sz\in B$, where $s\in\mathfrak{S}_\infty, z=(z_1,z_2,\ldots)\in \;_0\mathbb{Z}_2^\infty$.  If $$\xi_{reg}=\gamma_0^{-k/2}\,\gamma_1^{(k-n)/2}\prod\limits_{j=1}^\infty D_j$$ then, applying lemmas \ref{multiplicative_lemma}, \ref{Main_lemma21}, \ref{lemma_character_formula}, \ref{spectral_decomposition} and equality \ref{normalization}, we have
\begin{eqnarray}\label{n_reg}
\left(\Pi_{\alpha\beta\gamma}^\sigma(sz)\xi_{reg},\xi_{reg}\right)=\left\{
\begin{array}{rl}
0, &\text{ if } (sz\notin B_\infty^{(n)})\lor (s\in\mathfrak{S}_n\setminus e); \\
 \chi_{\alpha\beta\gamma}^\sigma(sz),&\text{ if } sz\in B_{n\infty};\\
(-1)^{\sum\limits_{j=k+1}z_j}, &\text{ if }  z=(z_1,z_2,\ldots,z_n,0,0,\ldots)\lor (s=e).
\end{array}
\right.
\end{eqnarray}
Let
$\rho=\left( {\rm Irr}_{\;^0\!\lambda}\otimes {\rm Irr}_{\,^1\!\!\lambda} \right)$ be the same as in section \ref{irreducible_repr_of_B_n} and let ${\rm ch}_\rho$ be an ordinary character of the irreducible representation $\rho$ of Young subgroup $G_{\Omega_{_{kn}}}\cap\mathfrak{S}_n$. Set
\begin{eqnarray*}
 F_\rho = \frac{{\rm dim\,\rho}}{k!(n-k)!}\;\sum\limits_{s\in G_{\Omega_{_{kn}}}\cap\mathfrak{S}_n}\;\overline{{\rm ch}_\rho(s)}\,\Pi_{\alpha\beta\gamma}^\sigma(s).
\end{eqnarray*}
By definition,
\begin{eqnarray}
 F_\rho \in M\cap\Pi_{\alpha\beta\gamma}^\sigma\left((G_{\Omega_{_{kn}}}\cap\mathfrak{S}_n)\,B_{n\infty}. \right)',
\end{eqnarray}
and we conclude from (\ref{n_reg}) that
\begin{eqnarray}\label{nonzero_projection}
F_\rho\,\xi_{reg}\neq 0.
\end{eqnarray}
 Moreover, by lemmas \ref{asym_II_1} and \ref{spectral_decomposition}, we have
\begin{eqnarray}
F_\rho\,\xi_{reg}=\xi_{reg}\,F_{\rho}.
\end{eqnarray}
Define character $\tau$ on $G_{\Omega_{_{kn}}}$ by
\begin{eqnarray}\label{tau}
\tau(sz)=\frac{{\rm ch}_\rho(s)\left(\prod\limits_{j=k+1}^n(-1)^{z_j}\right)\;{\rm dim\,\rho}}{k!(n-k)!},
\end{eqnarray}
where $z\in\mathbb{Z}_2^n$ and $s\in G_{\Omega_{_{kn}}}\cap\mathfrak{S}_n$.
It follows from (\ref{n_reg}) that
\begin{eqnarray}\label{character_of_B_n}
\begin{split}
{\rm tr}\left(\Pi_{\alpha\beta\gamma}^\sigma(sz) F_\rho^2
\,\xi^2_{reg}\right)=\left\{
\begin{array}{rl}
&0, \text{ if } (sz\notin B_\infty^{(n)})\lor (s\in\mathfrak{S}_n\setminus (G_{\Omega_{_{kn}}}\cap\mathfrak{S}_n)); \\
& \frac{\chi_{\alpha\beta\gamma}^\sigma(sz)\;({\rm dim\,\rho})^2}{k!(n-k)!},\text{ if } sz\in B_{n\infty};\\
& \tau(sz), \text{ if }  (z\in\mathbb{Z}_2^n)\land (s\in G_{\Omega_{_{kn}}}\cap\mathfrak{S}_n),
\end{array}
\right.
\end{split}
\end{eqnarray}
where $\tau$
Now we consider the left coset space $X=B_n\diagup G_{\Omega_{_{kn}}}$. Let $\left\{\mathfrak{s}(x) \right\}_{x\in X}$ be the corresponding complete set of coset representatives, where $\mathfrak{s}(G_{\Omega_{_{kn}}})$ is identity element $e$ of group $B_n$. We will suppose without loss of generality that
\begin{eqnarray*}
\left\{\mathfrak{s}(x) \right\}_{x\in X}\subset \mathfrak{S}_n\subset B_n.
\end{eqnarray*}
Define a collection of normal positive functionals $\phi_x$, where $x\in X=B_n\diagup G_{\Omega_{_{kn}}}$, on $M=\left(\Pi_{\alpha\beta\gamma}^\sigma(B)\right)^{''}$ as follows
\begin{eqnarray*}
\begin{split}
&\phi_x(m)=\left<m\,\Pi_{\alpha\beta\gamma}^\sigma(\mathfrak{s}(x))F_\rho\xi_{reg}, \Pi_{\alpha\beta\gamma}^\sigma(\mathfrak{s}(x))F_\rho\xi_{reg} \right>_{L^2(M,{\rm tr})}\\
&={\rm tr}\left( \left(\Pi_{\alpha\beta\gamma}^\sigma(\mathfrak{s}(x))\right)^*\,m \,\Pi_{\alpha\beta\gamma}^\sigma(\mathfrak{s}(x))\,F^2_\rho\xi^2_{reg}\right)=
{\rm tr}\left( m \,\Pi_{\alpha\beta\gamma}^\sigma(\mathfrak{s}(x))\,F^2_\rho\xi_{reg}^2\,\left(\Pi_{\alpha\beta\gamma}^\sigma(\mathfrak{s}(x))\right)^*\right).\;\;\;
\end{split}
\end{eqnarray*}

Set $\phi(m)=\sum\limits_{x\in X}\phi_x(m)$ and $\widetilde{D}=\sum\limits_{x\in X}\Pi_{\alpha\beta\gamma}^\sigma(\mathfrak{s}(x))\,F_\rho\xi_{reg}\,\left(\Pi_{\alpha\beta\gamma}^\sigma(\mathfrak{s}(x))\right)^*$.  It is useful to notice that for $g\in B_n$, using (\ref{character_of_B_n}), we get
\begin{eqnarray}\label{induced_repr}
\phi\left(\Pi_{\alpha\beta\gamma}^\sigma(g)\right)={\rm tr}\left(\Pi_{\alpha\beta\gamma}^\sigma(g)\;\widetilde{D} \right)=\sum\limits_{x: \mathfrak{s}(x)^{-1}g\mathfrak{s}(x)\in G_{\Omega_{_{kn}}}}\tau\left(\mathfrak{s}(x)^{-1}g\mathfrak{s}(x) \right).
\end{eqnarray}
Thus the restriction $\phi$ to $B_n$ is a character of irreducible induced representation  ${\rm Ind}_{G_{\Omega_{_{kn}}}}^{B_n}\rho$.
Now, applying lemma \ref{multiplicative_lemma} and (\ref{character_of_B_n}), we have
\begin{eqnarray}
\phi\left(\Pi_{\alpha\beta\gamma}^\sigma(g)\right)=\left\{
\begin{array}{rl}
&0, \text{ if } g\notin B_\infty^{(n)}; \\
& \frac{|X|\,\chi_{\alpha\beta\gamma}^\sigma(g)\;({\rm dim\,\rho})^2}{k!(n-k)!},\text{ if } g\in B_{n\infty};\\
& \phi\left(\Pi_{\alpha\beta\gamma}^\sigma(g_1)\right)\,\chi_{\alpha\beta\gamma}^\sigma(g_2), \text{ if } g=g_1g_2,~\text{where}~g_1\in B_n,g_2\in   B_{n\infty}.
\end{array}
\right.
\end{eqnarray}
Hence, using (\ref{induced_repr}), we obtain equality (\ref{main_equality}), where $D=\frac{k!(n-k)!\widetilde{D}}{|X|\,({\rm dim}\,\rho)^2}$. This proves theorem \ref{II_1_theorem}. \qed

\subsection{The type of representation $\Pi={\rm Ind}^B_{B_\infty^{(n)}}\,^n\!\pi$ in the case $\gamma_0+\gamma_1>0$.}\label{gamma>0}
Here we use the notation introduced at the beginning of the section \ref{inducing_stable_representation}. Recall that $H_{\infty\Omega_{_{kn}}}=B_k\cdot B_{kn}\cdot B_{n\infty} =\mathfrak{S}_k\cdot\mathfrak{S}_{kn}\cdot\mathfrak{S}_{n\infty}\cdot \mathbb{Z}^k\cdot\mathbb{Z}^{\widetilde{kn}}\cdot \mathbb{Z}_2^{\widetilde{n\infty}}$.

According to Proposition \ref{Prop_Ind}, we have
\begin{eqnarray}
{\rm Ind}^B_{B_\infty^{(n)}}\,^n\!\pi \stackrel{\text{quasi}}{=}{\rm Ind}^B_{B_\infty^{(n)}}\,{\rm Ind}^{B_\infty^{(n)}}_{H_{\infty\Omega_{_{kn}}}}\,\Xi \stackrel{\text{quasi}}{=}{\rm Ind}^B_{H_{\infty\Omega_{_{kn}}}}\,\Xi,
\end{eqnarray}
where we use symbol $\stackrel{\text{quasi}}{=}$ to denote quasi-equivalent representations.
Therefore,
\begin{eqnarray}\label{successive_induction}
{\rm Ind}^B_{B_\infty^{(n)}}\,^n\!\pi \stackrel{\text{quasi}}{=}{\rm Ind}^B_{B_\infty^{(k)}}\,{\rm Ind}^{B_\infty^{(k)}}_{H_{\infty\Omega_{_{kn}}}}\,\Xi\stackrel{\text{quasi}}{=}{\rm Ind}^B_{H_{\infty\Omega_{_{kn}}}}\,\Xi, ~\text{where}~ k\leq n.
\end{eqnarray}
\begin{Th}\label{Ind^B_infty_(k)}
Let $\Xi$ be the same as in  Proposition \ref{Prop_Ind}. Then the following holds:
\begin{itemize}
  \item {\bf 1}) if $\gamma_1>0$ then $\,^k\!\pi={\rm Ind}^{B_\infty^{(k)}}_{H_{\infty\Omega_{_{kn}}}}\,\Xi$ is ${\rm II}_1$-factor-representation;
  \item {\bf 2}) if $\gamma_0>0$ then $\,^\mathbb{A}\!\pi={\rm Ind}^{B_\infty^{\mathbb{A}}}_{H_{\infty\Omega_{_{kn}}}}\,\Xi$, where $\mathbb{A}=\overline{k+1,n}$, is   ${\rm II}_1$-factor-representation.
\end{itemize}
\end{Th}
\begin{proof}
Take $z=(0,0,\ldots,0,z_{k+1},\ldots,z_n)\in Z_2^{\widetilde{kn}}$, $s\in \mathfrak{S}_{kn}$, $t\in B_{n\infty}$.
Denote by $\Xi_1$ the representation of group $H_{\infty\Omega_{_{kn}}}\cap B_{k\infty}= B_{kn}\cdot B_{n\infty}$, which acts
in the space $\mathcal{H}_{\,^1\!\!\lambda}\otimes \mathcal{H}_{\alpha\beta\gamma}^\sigma$ as follows
\begin{eqnarray}\label{rho_1}
\Xi_1(zst)=\Omega_{_{kn}}(z)\,{\rm Irr}_{\,^1\!\!\lambda} (s)\,\,\otimes\pi_{\alpha\beta\gamma}^\sigma(t),\;~\text{where}~\Omega_{_{kn}}(z)=(-1)^{\sum_{j=k+1}^n z_j}.
\end{eqnarray}
Let  $X=B_\infty^{(k)}\slash H_{\infty\Omega_{_{kn}}}$ be the space of left cosets of the subgroup $H_{\infty\Omega_{_{kn}}}$ in $B_\infty^{(k)}$. It is clear that we can to identify $X$ with $B_{k\infty}\slash \left(B_{kn}\cdot B_{n\infty} \right)$. Hence, we get
\begin{eqnarray}\label{Gamma}
\,^k\!\pi(zsh)={\rm Irr}_{\,^0\!\!\lambda}(s)\otimes\left({\rm Ind}^{B_{k\infty}}_{\left(B_{kn}B_{n\infty}\right)}\,\Xi_1\right)(h), \; \;\;
\end{eqnarray}
where $ z\in \mathbb{Z}_2^k, s\in \mathfrak{S}_k $ and  $h\in B_{k\infty}$.
Here we identify naturally the spaces  $l^2\left(\mathcal{H}_{\,^0\!\!\lambda}\otimes \mathcal{H}_{\,^1\!\!\lambda}\otimes \mathcal{H}_{\alpha\beta\gamma}^\sigma,X \right)$ and $\mathcal{H}_{\,^0\!\!\lambda}\otimes l^2\left( \mathcal{H}_{\,^1\!\!\lambda}\otimes \mathcal{H}_{\alpha\beta\gamma}^\sigma,X \right)$.

Hence, in order to prove the theorem true, it suffices to show that $\widetilde{\Gamma}_k={\rm Ind}^{B_{k\infty}}_{\left(B_{kn}B_{n\infty}\right)}$ is ${\rm II}_1$-factor representation of group $B_{k\infty}$.

We starts with the observation that, by theorem \ref{factor_ind},  $\widetilde{\Gamma}_k$ is factor representation in the space $l^2\left(\mathcal{H}_{\,^1\!\!\lambda}\otimes \mathcal{H}_{\alpha\beta\gamma}^\sigma,X \right)$. We further note that there exists positive normal functional $\,^k\!\phi$ on $\widetilde{\Gamma}_k\left(B_k\right)^{''}$, which is defined as follows
\begin{eqnarray}\label{state on factor}
\,^k\!\phi(\widetilde{\Gamma}_k(g))=\left\{\begin{array}{rl}
&0, \text{ if } g\notin B_{kn}B_{n\infty}; \\
& (-1)^{\sum_{j=k+1}^n z_j}\,\,{\rm tr}_{\,^1\!\!\lambda}(s)\;\chi_{\alpha\beta\gamma}^\sigma(t), ~\text{ if }~ g=zst,
\end{array}
\right.
\end{eqnarray}
where $z=(0,0,\ldots,0,z_{k+1},\ldots,z_n)\in Z_2^{\widetilde{kn}}$, $s\in \mathfrak{S}_{kn}$, $t\in B_{n\infty}$ and ${\rm tr}_{\,^1\!\!\lambda}$ is normalized trace of representation ${\rm Irr}_{\,^1\!\!\lambda}$ of group $\mathfrak{S}_{kn}$.

Let $\,^k\!\Pi_{\alpha\beta\gamma}^\sigma$ be a ${\rm II}_1$-factor-representation of $B_{k\infty}$ in Hilbert space $\,^k\!\mathcal{H}_{\alpha\beta\gamma}^\sigma$ with the same parameters $\alpha$, $\beta$, $\gamma$ and $\sigma$ as $\pi_{\alpha\beta\gamma}^\sigma$. Denote by $\,^k\!{\rm tr}$ a normal trace on factor $\,^k\!M=\left(\,^k\!\Pi_{\alpha\beta\gamma}^\sigma(B_{k\infty})\right)^{''}$ such that $\,^k\!{\rm tr}(I)=1$. We assume below that $\,^k\!\Pi_{\alpha\beta\gamma}^\sigma$ acts on $\,^k\!\mathcal{H}_{\alpha\beta\gamma}^\sigma=L^2(\,^k\!M,\,^k\!{\rm tr})$  by the  operators of the left multiplication.

It follows from (\ref{Gamma}) that to prove of our theorem, it sufficient to find $D\in \,^k\!M$ with the properrty
\begin{eqnarray}
\,^k\!{\rm tr}\left(D \,\Pi_{\alpha\beta\gamma}^\sigma(g)\right)=\,^k\!\phi(\widetilde{\Gamma}_k(g)) ~\text{ for all }~ g\in B_{k\infty}.
\end{eqnarray}
First, we note that if $\gamma_0>0$ then our statement follows from (\ref{Gamma}) and theorem  \ref{II_1_theorem}. Using the notations from the proof of theorem \ref{II_1_theorem}, consider the  opposite case.
Since $\gamma_0=0$, (\ref{E_k_lambda_z}) and (\ref{normalization}) show that
\begin{eqnarray}\label{E_j_sign}
E_j(0)\;\Pi_{\alpha\beta\gamma}^\sigma\left( \zeta^{(j)} \right)=(-1)^{\zeta_j}\,E_j(0) \;~\text{ for }~ j>k.
\end{eqnarray}
Let us introduce operator $\,^k\!\xi_{reg}=\gamma_1^{k-n}\;E_{k+1}(0)\cdots E_n(0)\,\in \,\,^k\!M$. It is clear that, by (\ref{normalization}),
\begin{eqnarray*}
\,^k\!{\rm tr}(\,^k\!\xi_{reg})=\gamma_1^{k-n}\;\,^k\!{\rm tr}\left(E_{k+1}(0) \right)\,\,^k\!{\rm tr}\left(E_{k+2}(0)\cdots \,^k\!{\rm tr}\left(E_{n}(0) \right) \right) =1.
\end{eqnarray*}
Now we note that orthogonal projector $F_{\,^1\!\!\lambda} = \frac{\left({\rm dim\,\,^1\!\!\lambda}\right)^2}{(n-k)!}\;\sum\limits_{s\in \mathfrak{S}_{kn}}\;\overline{{\rm tr}_{\,^1\!\!\lambda}(s)}\,^k\!\Pi_{\alpha\beta\gamma}^\sigma(s)$, where ${\rm dim\,\,^1\!\!\lambda}$ is dimension of irreducible representation ${\rm Irr}_{\,^1\!\!\lambda}$ of group $\mathfrak{S}_{kn}$. Since, by lemmas \ref{multiplicative_lemma} and \ref{Main_lemma21}, $\,^k\!{\rm tr}\left(\,^k\!\xi_{reg}\,\,^k\!\Pi_{\alpha\beta\gamma}^\sigma(s)\right)=0$ for $s\notin B_{n\infty}$, we have
\begin{eqnarray*}
\,^k\!{\rm tr}\left(\,^k\!\xi_{reg}\,F_{\,^1\!\!\lambda}\right)=\,^k\!{\rm tr}\left(\,^k\!\xi_{reg}\;{\rm tr}_{\,^1\!\lambda}(e)\right)=\frac{\left({\rm dim\,\,^1\!\!\lambda}\right)^2}{(n-k)!}.
\end{eqnarray*}
Thus, for the same reason, using (\ref{rho_1}) and (\ref{E_j_sign}), it is easy check that
\begin{eqnarray*}
\,^k\!\phi(\widetilde{\Gamma}_k(g))=\,^k\!{\rm tr}\left(\,^k\!\xi_{reg}\,F_{\,^1\!\!\lambda}\,^k\!\Pi_{\alpha\beta\gamma}^\sigma(g)\right)\;~\text{for all }~g\in B_{k\infty}.
\end{eqnarray*}
It follows from this that representations $\widetilde{\Gamma}_k$ and $\,^k\!\Pi_{\alpha\beta\gamma}^\sigma$ of $B_{k\infty}$ are quasi-equivalent.
\end{proof}
\begin{Th}\label{II_infty_case}
Let $\Xi$, $\gamma_0$, $\gamma_1$ be the same as in  Proposition \ref{Prop_Ind} and Theorem \ref{Ind^B_infty_(k)}. Then the following holds:
\begin{itemize}
  \item 1) if $\gamma_0=0$ and $k\geq 1$ then $\Pi={\rm Ind}^B_{B_\infty^{(n)}}\,^n\!\pi$ is ${\rm II}_\infty$-factor representation;
  \item 2) if $\gamma_1=0$ and $k< n$ then $\Pi={\rm Ind}^B_{B_\infty^{(n)}}\,^n\!\pi$ is ${\rm II}_\infty$-factor representation.
\end{itemize}
\end{Th}
First of all, we will prove two auxiliary lemmas.
\begin{Lm}\label{product_transpositions}
Let $c=(p_0\;p_1\;\ldots\; p_{m-1})$ be a cycle from $\mathfrak{S}_\infty\setminus\left(\mathfrak{S}_k\mathfrak{S}_{k\infty} \right)$; i.e. $c(p_i)=p_{i+1({\rm mod}\,m)}$.
Introduce subsets $\left\{i_1<i_2<\ldots<i_l \right\}=\left\{i:p_i>k \right\}$ and $\left\{j_1<\ldots< j_q \right\}=\left\{j:p_{i_j+1}\leq k \right\}$. If $\hat{c} =\left(p_{i_1}\;p_{i_{j_{_1}}+1} \right)\left(p_{i_{(j_{_1}+1)}}\;p_{i_{j_{_2}}+1} \right)\cdots$ $\cdots \left(p_{i_{_{\left(j_{_{(q-1)}}+1\right)}}}\;p_{i_{j_{_q}}+1({\rm mod}\,m)} \right)$  then $\hat{c}\,c\in \mathfrak{S}_k\mathfrak{S}_{k\infty}$.
\end{Lm}
\begin{proof}
Assume, without loss of generality, that $p_0\leq k$.
It easy to check that $c=\left(p_0\;\,p_{m-1} \right)\,\left(p_0\;\,p_{2} \right)\,\left(p_0\;\,p_{1} \right)$. Then for $ h= \left(p_0\;\,p_{i_{(j_{_1}+1)}-1} \right)\left(p_0\;\,p_{i_{(j_{_1}+1)}-2} \right)\cdots$ $\cdots\left(p_0\;\,p_{i_{j_{_1}}+2}\right)\in\mathfrak{S}_k$ we have
\begin{eqnarray*}
c=  \cdots     \left(p_0\;\,p_{{i_{(j_{_1}+1)}}} \right)\,h\,\left(p_0\;\,p_{i_{j_{_1}}+1}\right) \left(p_0\;\,p_{{i_{j_{_1}}}}\right)\cdots    \left(p_0\;\,p_{i_1+1}\right)\left(p_0\;\,p_{i_1}\right)  \left(p_0\;\,p_{i_1-1}\right)\cdots\\
\cdots     \left(p_0\;\,p_{{i_{(j_{_1}+1)}}} \right)\,h\,\left(p_0\;\,p_{{i_{j_{_1}}+1}}\right)\left(p_0\;\,p_{i_1}\right) \left(p_{i_1}\;\,p_{{i_{j_{_1}}}}\right)\cdots    \left(p_{i_1}\;\,p_{i_1+1}\right)  \left(p_0\;\,p_{i_1-1}\right)\cdots\\
=\cdots     \left(p_0\;\,p_{{i_{(j_{_1}+1)}}} \right)\,h\,\left(p_{{i_{j_{_1}}+1}}\;\,p_{i_1}\right)\left(p_0\;\,p_{{i_{j_{_1}}+1}}\right) \left(p_{i_1}\;\,p_{{i_{j_{_1}}}}\right)\cdots    \left(p_{i_1}\;\,p_{i_1+1}\right)  \left(p_0\;\,p_{i_1-1}\right)\cdots
\end{eqnarray*}
Hence,
\begin{eqnarray*}
c=\left(p_0\;\,p_{m-1} \right)\,\cdots     \left(p_0\;\,p_{{i_{(j_{_1}+1)}}} \right)\,h\,\left(p_{{i_{j_{_1}}+1}}\;\,p_{i_1}\right)g_1\\
=\,\left(p_{{i_{j_{_1}}+1}}\;\,p_{i_1}\right)\left(p_0\;\,p_{m-1} \right)\,\cdots     \left(p_0\;\,p_{{i_{(j_{_1}+1)}}} \right)\,h\, g_1,\;\;
\end{eqnarray*}
where $g_1=\left(p_0\;\,p_{{i_{j_{_1}}+1}}\right) \left(p_{i_1}\;\,p_{{i_{j_{_1}}}}\right)\cdots    \left(p_{i_1}\;\,p_{i_1+1}\right)  \left(p_0\;\,p_{i_1-1}\right)\cdots\in\mathfrak{S}_k\mathfrak{S}_{k\infty}$.
Thus
\begin{eqnarray*}
\left(p_{{i_{j_{_1}}+1}}\;\,p_{i_1}\right)\,c=\left(p_0\;\,p_{m-1} \right)\,\cdots     \left(p_0\;\,p_{{i_{(j_{_1}+1)}}} \right)\,h\, g_1,
\end{eqnarray*}
~{ where }~ ${\rm supp}\,(hg_1)\in{\rm supp}\,c$.
Further, applying the above reasoning  to the cycle  $\hat{c} =\left(p_0\;\,p_{m-1} \right)\,\cdots     \left(p_0\;\,p_{{i_{(j_{_1}+1)}}} \right)=\left( p_0\;\, p_{{i_{(j_{_1}+1)}}}\;\,p_{i_{(j_{_1}+1)}+1}\cdots p_{m-1}\right)$, we obtain
\begin{eqnarray*}
\hat{c}=\left(p_{{i_{(j_{_1}+1)}}}\;\,p_{i_{j_{_2}}+1}\;\, \right)\left(p_0\;\,p_{m-1} \right) \cdots\left(p_0\;\, p_{i_{(j_2+1)}}\right)g_2,
\end{eqnarray*}
where $g_2\in \mathfrak{S}_k\mathfrak{S}_{k\infty}$ and ${\rm supp}\,g_2\subset {\rm supp}\,\hat{c}$.

Therefore,
\begin{eqnarray*}
c=\left(p_{{i_{j_{_1}}+1}}\;\,p_{i_1}\right)\left(p_{{i_{(j_{_1}+1)}}}\;\,p_{i_{j_{_2}}+1}\;\, \right)\left(p_0\;\,p_{m-1} \right) \cdots\left(p_0\;\, p_{i_{(j_2+1)}}\right)hg_1g_2.
\end{eqnarray*}
Repeating the above transformations, we obtain the statement of lemma \ref{product_transpositions}.
\end{proof}
\begin{Co}
It follows from lemma \ref{product_transpositions} that
\begin{eqnarray*}
c\{1,2,\ldots k\}\cap\{k+1,k+2,\ldots\}=\hat{c}\left\{p_{i_{j_{_1}}+1},p_{i_{j_{_2}}+1},\ldots, p_{i_{j_{_q}}+1({\rm mod}\,m)} \right\}.
\end{eqnarray*}
 Any  permutation $s\in\mathfrak{S}_\infty$ can be expressed as the product of disjoint cycles. Therefore,  there exit two collections
$\mathcal{C}(s)=\left\{p_1,p_2,\ldots,p_q \right\}\subset\{1,2,\ldots k\}$ and $s(\mathcal{C}(s))=\left\{r_1,r_2,\ldots, r_q \right\}\subset\{k+1,k+2,\ldots\}$ such that
\begin{eqnarray}
\mathcal{C}(s)\cup s(\mathcal{C}(s))\subset {\rm supp}\, s\;\; ~\text{{\rm and}} ~\left(\prod\limits_{i=1}^q\left(p_i\;\, r_i \right)\right)\, \mathfrak{S}_k\mathfrak{S}_{k\infty}\,=\,s\,  \mathfrak{S}_k\mathfrak{S}_{k\infty}.
\end{eqnarray}
\end{Co}
\begin{Lm}\label{unique_involution}
Define  the elements $\hat{s},\hat{s}'\in \mathfrak{S}_\infty$ as follows $$\hat{s}=\prod\limits_{i=1}^q\left(p_i\;\, r_i \right), \hat{s}'=\prod\limits_{i=1}^{q'}\left(p'_i\;\, r'_i \right),$$
where $p_i,p_i'\in\{1,2,\ldots k\}$, $r_i,r_i'\in\{k+1,k+2,\ldots \}$, $p_i<p_{i+1}$ for $i=1,2,\ldots,q-1$ and ${p'}_i<{p'}_{i+1}$ for $i=1,2,\ldots,q'-1$.
 \begin{itemize}
   \item i) If $\hat{s}\;  \mathfrak{S}_k\mathfrak{S}_{k\infty}=  \hat{s}'\;  \mathfrak{S}_k\mathfrak{S}_{k\infty}$ then $\left\{p_1,p_2,\ldots,p_q \right\}=\left\{p_1',p_2',\ldots,p_{q'}' \right\}$, $\left\{r_1,r_2,\ldots, r_q \right\}=\left\{r_1',r_2',\ldots, r_{q'}' \right\}$ and there exists $g'\in \mathfrak{S}_\infty$ such that
\begin{eqnarray*}
{\rm supp}\,g'\subset\left\{r_1,r_2,\ldots, r_q \right\},\;\;g'\left(r_i' \right)=r_i ~\text{ for all } ~ i=1,2,\ldots q.
\end{eqnarray*}
   \item ii) If $g(p)=\left\{
\begin{array}{rl}
g'(p), \text{ if } p>k\\
\left(\hat{s}\;g'^{-1}\hat{s}\right)(p),\text{ if } p\leq k
\end{array}
\right.$ then $g\in \mathfrak{S}_k\mathfrak{S}_{k\infty}$ and $\hat{s}'=\hat{s}g$.
   \item iii) Given involution  $\hat{s}'=\prod\limits_{i=1}^{q}\left(p'_i\;\, r'_i \right)$, take $h'\in \mathfrak{S}_\infty$ such that ${\rm supp}\,h'\subset\left\{r_1',r_2',\ldots, r_q' \right\},\;\;h'\left(r_i' \right)<h'\left(r_{i+1}' \right)$  for all  $i=1,2,\ldots, q-1$.
       If $$h(p)=\left\{
\begin{array}{rl}
(h')^{-1}(p), \text{ if } p>k\\
\left(\hat{s}\;h'\hat{s}\right)(p),\text{ if } p\leq k
\end{array}
\right.$$ then $h\in \mathfrak{S}_k\mathfrak{S}_{k\infty}$ and $\left(\hat{s}'\,h \right)=\prod\limits_{i=1}^{q}\left(p_i\;\; h'(r'_i) \right)$.
 \end{itemize}
\end{Lm}
\begin{proof}
Since $\hat{s}\;\mathfrak{S}_k\mathfrak{S}_{k\infty}=  \hat{s}'\;  \mathfrak{S}_k\mathfrak{S}_{k\infty}$, there exists $g\in\mathfrak{S}_k\mathfrak{S}_{k\infty}$ such that $\hat{s}g=\hat{s}'$. Therefore, $\left\{r_i' \right\}_{i=1}^{q'}=\hat{s}g\left\{p_i' \right\}_{i=1}^{q'}=(\hat{s}g\{1,2,\ldots.k\})\cap\{k+1,k+2, \ldots\}$. Since $g\{1,2,\ldots.k\}=\{1,2,\ldots.k\}$, we have
\begin{eqnarray*}
\left\{r_i' \right\}_{i=1}^{q'}=(\hat{s}\{1,2,\ldots.k\})\cap\{k+1,k+2, \ldots\}=\left\{r_i \right\}_{i=1}^{q}.
\end{eqnarray*}
Hence, we obtain that $q=q'$ and
\begin{eqnarray*}
\left\{p_i' \right\}_{i=1}^{q}=\{1,2,\ldots.k\}\cap(\hat{s}g\{k+1,k+2, \ldots\})\\=\{1,2,\ldots.k\}\cap(\hat{s}\{k+1,k+2, \ldots\})=\left\{p_i \right\}_{i=1}^{q}.
\end{eqnarray*}
Thus we can to find $g'$ such that
\begin{eqnarray*}
{\rm supp}\,g'\subset\left\{r_1,r_2,\ldots, r_q \right\},\;\;g'\left(r_i' \right)=r_i ~\text{ for all } ~ i=1,2,\ldots q.
\end{eqnarray*}
This finishes the proof.
\end{proof}

\begin{proof}[Proof of theorem \ref{II_infty_case}]\label{proof_28}
 By Theorem \ref{Ind^B_infty_(k)}, $\,^k\!\pi={\rm Ind}^{B_\infty^{(k)}}_{H_{\infty\Omega_{_{kn}}}}\,\Xi$ is ${\rm II}_1$-factor. representation of $B_{k\infty}$. It follows from (\ref{successive_induction}) that
 \begin{eqnarray}
 \Pi={\rm Ind}^B_{B_\infty^{(n)}}\,^n\!\pi\stackrel{\text{quasi}}{=}{\rm Ind}^B_{B_\infty^{(k)}}\,^k\!\pi.
 \end{eqnarray}
 It is clear that the left coset space $X=B\slash{B_\infty^{(k)}}$ is identified naturally  with $\mathfrak{S}_\infty\slash \left(\mathfrak{S}_k\times\mathfrak{S}_{k\infty}\right)$. Namely,  there exists a map $X\ni x\stackrel{\mathfrak{s}}{\mapsto}\mathfrak{s}_x\in \mathfrak{S}_\infty\subset B$ with the properties $x=\mathfrak{s}_xB_\infty^{(k)}$ and $\mathfrak{s}_{_{\left( B_\infty^{(k)} \right)}}=e$, where $e$ is identity element of $\mathfrak{S}_\infty$. Denote by $\widetilde{g}$ the class in $X$ containing $g\in B$. By Lemmas \ref{product_transpositions} and \ref{unique_involution}, every coset $x=s\;\mathfrak{S}_k\mathfrak{S}_{k\infty}\in X$ contains the unique involution $\mathfrak{s}_x=\prod\limits_{i=1}^{\,^x\!l} \left(\,^x\!p_i\;\,\,^x\!r_i \right)$ such that $\,^x\!p_i<\,^x\!p_{i+1}\leq k$ and $k<\,^x\!r_i<\,^x\!r_{i+1}$ for all $i=1,2,\ldots,\,^x\!l-1<k$. In particular, $\mathfrak{s}_{\widetilde{e}}=e$ and $\,^{\widetilde{e}}\!l=0$; i.e. collection $\left\{\,^{\widetilde{e}}\!p_i \right\}$ is empty set in this case.

Denote by $\,^k\!M$ ${\rm II}_1$-factor generated by operators $\,^k\!\pi\left(B_\infty^{(k)}\right)$.
Without loss of generality we can to assume that $\,^k\!\pi$ is realized on $L^2\left( \,^k\!M, \,^k\!{\rm tr} \right)$, where $\,^k\!{\rm tr}$ is a normal, faithful trace on $\,^k\!M$ such that  $\,^k\!{\rm tr}(I)=1$, by the operators of left multiplication. Then    the identity $I$   of factor $\,^k\!M$ is cyclic and separating vector for $\,^k\!\pi$.

If $m>k$ then, by lemma \ref{asym_II_1}, there exist the operators
\begin{eqnarray}\label{^k_O}
\,^k\!\mathcal{O}_m=w-\lim\limits_{j\to\infty}\;^k\!\pi\left( (m\,\;j) \right).
\end{eqnarray}

The operators of the representation $\Pi={\rm Ind}^B_{B_\infty^{(k)}}\,^k\!\pi$ act on $v\in \mathcal{H}_\Pi$ $=l^2\left(L^2\left( \,^k\!M, \,^k\!{\rm tr} \right),X\right)$ by the formula
\begin{eqnarray}\label{ind1_action}
\left(\Pi(g)v\right)(x)=c(g,x)v\left(g^{-1}x\right),
\end{eqnarray}
where $c(g,x)= \,^k\!\pi\left(\left( \mathfrak{s}_x \right)^{-1}g\;\mathfrak{s}_{_{\left(g^{-1}x\right)}}\right)$. Notice that $\left(\mathfrak{s}_x \right)^{-1}g\;\mathfrak{s}_{_{\left(g^{-1}x\right)}}$  lies in $B_\infty^{(k)}$.

 Set
\begin{eqnarray}\label{trace1_vector}
\xi(x)=\left\{
\begin{array}{rl}
I, \text{ if } x=\widetilde{e};\\
0,\text{ if } x\neq\widetilde{e}.
\end{array}
\right.
\end{eqnarray}
Since $I$ is cyclic in  $L^2\left( \,^k\!M, \,^k\!{\rm tr} \right)$, then  $\left[\Pi(B)\xi\right]=\mathcal{H}_\Pi$. I.e., $\xi$ is cyclic vector for representation $\Pi$.  Now,  we conclude from (\ref{ind1_action}) and (\ref{trace1_vector})  that $$\left(\Pi(g)\xi,\xi\right)=\left\{
\begin{array}{rl}
0, \text{ if } g\notin B_\infty^{(k)}\\
\,^k\!{\rm tr}(g), \text{ if } g\in B_\infty^{(k)}.
\end{array}
\right.$$  We thus get
\begin{eqnarray}
\left(\Pi(g)\xi,\xi\right)=\left(\Pi(h)\Pi(g)\Pi(h^{-1})\xi,\xi\right) ~\text{ for all }~ g\in B  ~\text{ and }~ h\in B_\infty^{(k)}.
\end{eqnarray}
It follows from this that $\left(\Pi\left((m\,\;l) \right)\Pi(g_1)\xi, \Pi(g_2)\xi \right)=\left(\Pi\left((m\,\;l') \right)\Pi(g_1)\xi, \Pi(g_2)\xi \right)$ for all $g_1,g_2\in B$ and $l, l'\notin \;({\rm supp}\,g_1)\cup ({\rm supp}\,g_2)\cup \left\{1, 2,\ldots,k \right\}$. Hence, using the cyclicity of $\xi$, we obtain that sequence $\left\{ \Pi\left((m\,\;l) \right)\right\}_{l=1}^\infty$ converges in the weak operator topology. Set
\begin{eqnarray}
\mathcal{O}_m=w-\lim\limits_{l\to\infty}\Pi\left((m\,\;l) \right).
\end{eqnarray}
If $m\leq k$ then the operators $\mathcal{O}_m$ and $\Pi\left(1^{(m)}\right)$, where $1^{(m)}=(\underbrace{0,0,\ldots,0}_{m-1},1,0,\ldots)$ $\in \;_0\mathbb{Z}_2^\infty$, act on $\eta\in l^2\left(L^2\left( \,^k\!M, \,^k\!{\rm tr} \right),X\right)$ as follows
\begin{eqnarray}\label{asymp_transposition_2.55_operator}
\begin{split}
&\left(\mathcal{O}_m\eta\right)(x)=\left\{
\begin{array}{rl}
0, ~\text{ if }~ m\notin\left\{\,^x\!p_i \right\}_{i=1}^{\,^{^x}\!\!l};\\
\,^k\!\mathcal{O}_{\!\,{^x\!r_i}}\;\eta(x),~\text{ if }~ m=\,^x\!p_i,
\end{array}
\right.\\
&\left(\Pi\left(1^{(m)} \right)\;\eta(x) \right)= \,^k\!\pi\left(1^{(\mathfrak{s}_x(m))} \right)\;\eta(x), ~\text{where}~ \mathfrak{s}_x=\prod\limits_{i=1}^{\,^{^x}\!\!l} \left(\,^x\!p_i\;\,\,^x\!r_i \right).
\end{split}
\end{eqnarray}
In particular, \  \ by Lemma \ref{spectral_decomposition_Okounkov}, \ \ operator \ $\mathcal{O}_m$ \  has a discrete spectrum   with its only limit point at zero.\ \ \ \ \
 \ \ \ \ \ \ \
 Let $\mathcal{O}_m=\sum\limits_{\lambda\in S\left(\mathcal{O}_m \right)} \lambda\, E_m(\lambda)$ be a spectral decomposition of $\mathcal{O}_m$ in the space $l^2\left(L^2\left( \,^k\!M, \,^k\!{\rm tr} \right), X\right)$, and let $$\,^k\!\mathcal{O}_l=\sum\limits_{\lambda\in S\left(\,^k\!\mathcal{O}_l \right)} \lambda \;\,^k\!E_l(\lambda),$$ where $l>k$, be a spectral decomposition of $\,^k\!\mathcal{O}_l$ in the space $L^2\left( \,^k\!M, \,^k\!{\rm tr} \right)$. It follows from (\ref{asymp_transposition_2.55_operator}) that
\begin{eqnarray}\label{zero_projection}
\left( E_m(0)\;\eta\right)(x)=\left\{
\begin{array}{rl}
\eta(x), &~\text{ if }~ m\notin\left\{\,^x\!p_i \right\}_{i=1}^{\,^{^x}\!\!l};\\
\,^k\!E_{\,({^x\!r_i})}(0)\;\eta(x),&~\text{ if }~ m=\,^x\!p_i.
\end{array} \right.
\end{eqnarray}
Hence, using (\ref{asymp_transposition_2.55_operator}) and lemma \ref{lemma_commutative}, we obtain
\begin{eqnarray}\label{Product_zero_projection}
\begin{split}
&\left( \Pi\left(1^{(m)}\right)\,\,E_m(0)\eta\right)(x)=\left(E_m(0)\,\, \Pi\left(1^{(m)}\right)\eta\right)(x)\\
&\stackrel{(\gamma_0=0)}{=}
\left\{
\begin{array}{rl}
\eta(x), &~\text{ if }~ m\notin\left\{\,^x\!p_i \right\}_{i=1}^{\,^{^x}\!\!l};\\
-\,^k\!E_{\,({^x\!r_i})}(0)\;\eta(x),&~\text{ if }~ m=\,^x\!p_i.
\end{array} \right.
\end{split}
\end{eqnarray}
Applying \eqref{zero_projection} and \eqref{Product_zero_projection}, we obtain
\begin{eqnarray}\label{projection_on_unit}
\left(E_m(0)\left(I+\Pi\left(1^{(m)}\right)\right)\;\eta\right)(x)=\left\{
\begin{array}{rl}
2\,\eta(x), &~\text{ if }~ m\notin\left\{\,^x\!p_i \right\}_{i=1}^{\,^{^x}\!\!l};\\
0,&~\text{ if }~ m=\,^x\!p_i.
\end{array} \right.
\end{eqnarray}
For simplicity of notation, we denote orthogonal projection $\frac{1}{2}\, E_m(0)\left(I+\Pi\left(1^{(m)}\right)\right)$ by $Q_m$. At last, we obtain from (\ref{projection_on_unit}) that
\begin{eqnarray}\label{product_of_Q_m}
\left(\left(\prod\limits_{j=1}^k \,Q_j\right)\,\eta\right)(x)=\left\{
\begin{array}{rl}
\eta(x), ~\text{ if }~ x=B_\infty^{(k)}\\
0,~\text{ if }~x\neq B_\infty^{(k)}.
\end{array} \right.
\end{eqnarray}
If $Q=\prod\limits_{j=1}^k \,Q_j$ then the  projections  from the collection $\left\{\Pi(\mathfrak{s}_x)\,Q\, \Pi(\mathfrak{s}_x)\right\}_{x\in X}$ are equivalent and pairwise orthogonal. Hence, using the facts that $Q\in\,\Pi(B)''$ and $Q\,\Pi(B)''\,Q$ is ${\rm II}_1$-factor, we conclude that $\Pi(B)''$ is ${\rm II}_\infty$-factor.
\end{proof}

\subsection{The case $\gamma_0=\gamma_1=0$.}\label{all_gamma_are_zero}
Below we use the notations, introduced at the beginning of the section \ref{inducing_stable_representation}. Let us recall that $\,^n\!\pi$ is finite type factor-representation of group $\;B_\infty^{(n)}=B_nB_{n\infty}$. Let $\,^n\!\pi_{{|_{B_n}}}$ be the restriction of $\,^n\!\pi$ to subgroup $B_n$, and let  $\,^n\!\pi_{{|_{B_{n\infty}}}}$ be the restriction of $\,^n\!\pi$ to $B_{n\infty}$.   If $\;^n\!M=\,^n\!\pi\left( B_\infty^{(n)}\right)^{\prime\prime}$ and $\;^n\!{\rm tr}$ is a normal, faithful trace on $\;^n\!M$ such that  $\,^n\!{\rm tr}(I)=1$ then $\;^n\!{\rm tr}\left(\,^n\!\pi\left(st\right)\right)=\;^n\!{\rm tr}(\,^n\!\pi(s))\,\;^n\!{\rm tr}(\,^n\!\pi(t))$ for all $s\in B_n$ and $t\in B_{n\infty}$. Therefore, $\,^n\!\pi_{{|_{B_n}}}$ and $\,^n\!\pi_{{|_{B_{n\infty}}}}$ are  the factor-representations of the finite type. We recall that  $\,^n\!\pi_{{|_{B_n}}}$ is quasi-equivalent to representation ${\rm Ind}^{B_n}_{G_{\Omega_{_{kn}}}}\left(\,^{^{\Omega_{_{kn}}}}\!\!\!\rho\right)$, which is defined in subsection \ref{irreducible_repr_of_B_n}, and ${\rm tr}(\,^n\!\pi(t))$ calculate by (\ref{character_formula}); i.e. ${\rm tr}(\,^n\!\pi(t))=\chi_{\alpha\beta\gamma}^\sigma(t)$, where $\gamma=(\gamma_0,\,\gamma_1)$ and $\gamma_0+\gamma_1=1-\sum\alpha_i -\sum\beta_i$.
\begin{Rem}\label{one_dim}
Let $\pi_{\alpha\beta\gamma}^\sigma$ be a finite type factor-representation of $B_{n\infty}$, corresponding to the character $\chi_{\alpha\beta\gamma}^\sigma$.  Take $z=\left( 0,\ldots,0,z_{n+1},z_{n+2},\ldots \right)\in \mathbb{Z}_2^{n\infty}$ and  $s\in\mathfrak{S}_{n\infty}$. If $\pi_{\alpha\beta\gamma}^\sigma$ has type ${\rm I}$ then $\pi_{\alpha\beta\gamma}^\sigma$ it is one-dimensional. There exists four one-dimensional representations:
\begin{itemize}
  \item a) $\pi_{1\emptyset0}^\sigma$, where $\sigma(1)=0$; i. e. $\pi_{1\emptyset0}^\sigma(sz)=1$;
  \item b) $\pi_{1\emptyset0}^\sigma$, where $\sigma(1)=1$; i. e. $\pi_{1\emptyset0}^\sigma(sz)=(-1)^{\sum z_i}$;
  \item c)  $\pi_{\emptyset10}^\sigma$, where $\sigma(1)=0$; i. e. $\pi_{\emptyset10}^\sigma(sz)={\rm sign}(s)$;
  \item d)  $\pi_{\emptyset10}^\sigma$, where $\sigma(1)=1$; i. e. $\pi_{1\emptyset0}^\sigma(sz)=({\rm sign}(s)) (-1)^{\sum z_i}$.
\end{itemize}
\end{Rem}
Since $\;_0\mathbb{Z}_2^\infty\subset B_\infty^{(n)}$, we can identify naturally the left coset space $X=B\slash{B_\infty^{(n)}}$ with $\mathfrak{S}_\infty\slash \left(\mathfrak{S}_n\mathfrak{S}_{n\infty} \right)$. It follows from Lemmas \ref{product_transpositions} and \ref{unique_involution} that every left coset $x=gB_\infty^{(n)}$ contains unique involution $\mathfrak{s}_x=\prod\limits_{i=1}^{\,^x\!l} \left(\,^x\!p_i\;\,\,^x\!r_i \right)$, where $\,^x\!p_i<\,^x\!p_{i+1}\leq k$ and $k<\,^x\!r_i<\,^x\!r_{i+1}$ for all $i=1,2,\ldots,l-1$.
\begin{Th}\label{alpha+beta=1}
Let $\,^n\!\pi$ be a ${\rm II}_1$-factor-representation of $B_\infty^{(n)}$, and let $\sum\alpha_i +\sum\beta_i=1$. Suppose that   $n\geq 1$. If  there exists $i$ such that $\alpha_i<1$ or $\beta_i<1$ then $\Pi={\rm Ind}^B_{B_\infty^{(n)}}\,^n\!\pi$ is  stable factor-representation of $B$ of the type ${\rm II}_\infty$. In opposite case $\Pi$ is ${\rm I}_\infty$ stable factor-representation.
\begin{proof}
Here we follow the notations used in the proof of theorem \ref{II_infty_case}.
Namely, the operators of the representation $\Pi={\rm Ind}^B_{B_\infty^{(n)}}\,^n\!\pi$ act on $v\in \mathcal{H}_\Pi=l^2\left(L^2\left( \,^n\!M, \,^n\!{\rm tr} \right),X\right)$ by the formula
\begin{eqnarray}\label{ind2_action}
\left(\Pi(g)v\right)(x)=c(g,x)v\left(g^{-1}x\right),
\end{eqnarray}
where $c(g,x)= \,^n\!\pi\left(\left( \mathfrak{s}(x) \right)^{-1}g\;\mathfrak{s}\left(g^{-1}x\right)\right)$, $\left( \mathfrak{s}(x) \right)^{-1}g\;\mathfrak{s}\left(g^{-1}x\right)$ lies in $B_\infty^{(n)}$ and $\mathfrak{s}_x=\prod\limits_{i=1}^{\,^x\!l} \left(\,^x\!p_i\;\,\,^x\!r_i \right)$. Recall that $\left\{\,^x\!p_i\;\, \right\}_{i=1}^{\,^x\!l}\subset \left\{1,2,\ldots,n \right\}$,  $\left\{\,^x\!r_i\;\, \right\}_{i=1}^{\,^x\!l}\subset \left\{n+1,n+2,\ldots\right\}$. Let us define the operators  $\left\{\,^n\!\mathcal{O}_m \right\}_{m=n+1}^\infty$ in the same way as in \ref{^k_O}; i.e.
\begin{eqnarray*}
\,^n\!\mathcal{O}_m\,v=w-\lim\limits_{j\to\infty}\;^n\!\pi\left( (m\,\;j) \right)\,v, ~\text{ where }~ v\in L^2\left( \,^n\!M, \,^n\!{\rm tr} \right).
\end{eqnarray*}
Define $\xi\in \mathcal{H}_\Pi$ by
\begin{eqnarray}\label{trace2_vector}
\xi(x)=\left\{
\begin{array}{rl}
I, \text{ if } x=\widetilde{e};\\
0,\text{ if } x\neq\widetilde{e}.
\end{array}
\right.
\end{eqnarray}
It follows from this that
\begin{eqnarray*}
\left(\Pi(g)\xi,\xi\right)=\left\{
\begin{array}{rl}
0, \text{ if }g\notin B_\infty^{(n)};\\
\;^n\!{\rm tr}(\,^n\!\pi(g)),\text{ if } \; g\in B_\infty^{(n)}.
\end{array}
\right.
\end{eqnarray*}
Therefore,
\begin{eqnarray}\label{stability_equality}
\omega_\xi(g)=\left(\Pi(g)\xi,\xi\right)=\left(\Pi(hgh^{-1})\xi,\xi\right)\;~\text{ for all }~ g\in B ~\text{ and }~ h\in B_\infty^{(n)}.
\end{eqnarray}
This gives
\begin{eqnarray}\label{invariance_l_l'}
\left(\Pi((m\;\,l))\,\Pi(g_1)\xi,\Pi(g_2)\xi\right)=\left(\Pi((m\;\,l'))\,\Pi(g_1)\xi,\Pi(g_2)\xi\right)
\end{eqnarray}
for all $g_1,g_2\in B$ and $l,l'\notin ({\rm supp}\,g_1)\cup ({\rm supp}\,g_1)$.
Since $\xi$ is cyclic vector for $\Pi$, we conclude from (\ref{invariance_l_l'}) that a family $\left\{\Pi((m\;\,l))\right\}_{l=1}^\infty$ of operators converges in the weak operator topology. Set $\mathcal{O}_m=w-\lim\limits_{l\to\infty}\Pi\left((m\,\;l) \right)$.
Now we notice that under the conditions $1\leq m\leq n $ and $l>{\rm max}\,\left\{\,^x\!r_i \right\}^{^x\!l}_{i=1}$ the equality
 $\left( \mathfrak{s}(x) \right)^{-1}\;(m\;\,l)\;\mathfrak{s}\left(\;(m\;\,l)\;x\right)=\left\{
\begin{array}{rl}
( \,^x\!r_i \;\,l), ~\text{ if }~ m=\,^x\!p_i ;\\
e,~\text{ if }~ m\notin\left\{\,^x\!p_i \right\}^{^x\!l}_{i=1}
\end{array}
\right.$
holds. Hence, using (\ref{ind2_action}), we have
\begin{eqnarray}\label{mathcal{O}_m_gamma_1+gamma_2=0}
\left(\mathcal{O}_m\eta\right)(x)=\left\{
\begin{array}{rl}
0, ~\text{ if }~ m\notin\left\{\,^x\!p_i \right\}_{i=1}^{\,^x\!l},\\
\,^n\!\mathcal{O}_{\!\,{^x\!r_i}}\;\eta(x)~\text{ if }~ m=\,^x\!p_i
\end{array}
\right.,\;~\text{ where }~ m=1,2,\ldots, n.
\end{eqnarray}
By lemma \ref{spectral_decomposition},
\begin{eqnarray*}
{\rm Ker}\,^n\!\mathcal{O}_m\,=\left\{ v\in L^2\left( \,^n\!M, \,^n\!{\rm tr} \right): \,^n\!\mathcal{O}_m\,v=0 \right\}=0.
\end{eqnarray*}
Hence, using (\ref{mathcal{O}_m_gamma_1+gamma_2=0}), we have
\begin{eqnarray}
{\rm Ker}\,\mathcal{O}_m\,=\left\{\eta\in \mathcal{H}_\Pi:\eta(x)=0, ~\text{ if }~ m\in \left\{\,^x\!p_i \right\}_{i=1}^{\,^x\!l} \right\},\;\;~ m=1,2,\ldots, n.
\end{eqnarray}
Therefore,
\begin{eqnarray}\label{indicator_formula}
\bigcap\limits_{m=1}^n {\rm Ker}\,\mathcal{O}_m\,=\left\{\eta\in \mathcal{H}_\Pi: \eta(x)=0 ~\text{ for all }~ x\neq\widetilde{e}\right\}.
 \end{eqnarray}
 It follows from this that orthogonal projection $E$ onto subspace $H_{\widetilde{e}}=\left\{\eta\in \mathcal{H}_\Pi: \eta(x)=0 ~\text{ for all }~ x\neq\widetilde{e}\right\}$ lies in $w^*$-subalgebra $\Pi\left(\mathfrak{S}_\infty \right)''\subset\Pi\left(B\right)''$. Since projections $E_x=\Pi(\mathfrak{s}_x)E\Pi((\mathfrak{s}_x)^{-1})$ and $E_y$ are orthogonal for different $x,y\in X$ and $w^*$-algebra $E\Pi\left(B \right)''E=\,^n\!\pi\left(B_\infty^{(n)}\right)''$ is a finite type factor, $w^*$-algebra $\Pi\left(B \right)''$ is a semi-finite factor (of the type {\rm I} or {\rm II}) (see theorem \ref{factor_ind}).
\end{proof}
\end{Th}

  \section{The properties of the stable representations}\label{properties_of_stable_repr}
  Throughout this section, $\pi$ is stable factor-representation of  $\mathbb{Z}_2\wr\mathfrak{S}_\infty$ in the separable Hilbert space $\mathcal{H}_\pi$. We suppose that there exists the cyclic and separating vector $\xi\in \mathcal{H}_\pi$: $\left[ \pi(\mathbb{Z}_2\wr\mathfrak{S}_\infty)\xi \right]=\left[ \pi(\mathbb{Z}_2\wr\mathfrak{S}_\infty)^\prime\xi \right]=\mathcal{H}_\pi$.  Such realisation of $\pi$ is called {\it standard}. In opposite case, using corollary \ref{stable_quasi}, we will replace $\pi$ by the quasi-equivalent representation.
  \subsection{The standard realisation of $\pi$ and the corresponding representation of the inner automorphism group of $B$.}
  Set $M=\pi(B)^{\prime\prime}$.
  In this section we prove the next relevant statement (see page \pageref{Proof_th_ind}).
  \begin{Th}\label{induction_theorem}
    Let $M_*^+(n)=\left\{ \omega\in M_*^+\big|\omega\left( \pi(g)\cdot a\cdot \pi(g^{-1}) \right)=\omega\left( a \right)\text{ for all } \right.$ \newline $\left.g\in B_n\cdot B_{n\infty}\right.$ $\text{ and each }\left.  a\in M  \right\}$. Then the following hold:
     \begin{itemize}
       \item {\bf a}) there exists $n$ such that $M_*^+(n)\neq0$;
       \item {\bf b}) if ${\rm cd}\,\pi=\min\left\{ n\big|M_*^+(n)\neq0\right\}$, then $\left({\rm supp}\,\omega \right)\cdot \pi(s)\cdot \left( {\rm supp}\,\omega \right)=0$ for all $s\notin B_{{\rm cd}\,\pi}\cdot B_{({{\rm cd}\,\pi})\;\infty}$ and each nonzero $\omega\in M_*^+({\rm cd}\,\pi)$.
     \end{itemize}
 \end{Th}
 \begin{Def}
 We call ${\rm cd}\,\pi$ the central depth of $\pi$.
 \end{Def}
 \begin{Co}\label{collora_of_ind_Th}
   Let ${\rm cd}\,\pi$ and $\omega$ be the same as in Theorem \ref{induction_theorem}. Define functional $\varphi\in M_*^+({\rm cd}\,\pi)$ by
   \begin{eqnarray*}
   \varphi(a)=\left|B_n\right|^{-1}\sum\limits_{g\in B_n}\omega\left(\pi(g)\cdot a\cdot\pi(g^{-1})\right),\;\;\text{where } a\in M \text{ and } n={\rm   cd}\,\pi.
   \end{eqnarray*}
   Set $E={\rm supp}\,\varphi$. Then $E\in M\cap \pi\left( B_{\infty}^{(n)}\right)^\prime$. Thus, the operators $\pi_E(g)$ $=E\cdot\pi(g)\cdot E$, where $g\in B_{\infty}^{(n)}= B_n\cdot B_{n\infty}$, generate the unitary finite type factor-representation of the group $B_{\infty}^{(n)}$ in the space $E\mathcal{H}_\pi$. It follows from Theorem \ref{induction_theorem} that $\pi$ induced  from the representation $\pi_E$ of subgroup $B_{\infty}^{(n)}$; i. e. $\pi={\rm Ind}_{B_{\infty}^{(n)}}^B \pi_E$.
 \end{Co}
   To prove theorem \ref{induction_theorem} we need  some results  of  Tomita-Takesaki modular theory \cite{TAKES2}.

   Define operator $\widetilde{S}_\xi$ on $\mathcal{H}_\pi$ as follows
  \begin{eqnarray*}
 \mathcal{H}_\pi \supset M\xi\ni A\xi\stackrel{\widetilde{S}_\xi}{\mapsto}A^*\xi.
  \end{eqnarray*}
  Since $\xi$ is cyclic and separating, operator $\widetilde{S}_\xi$ extends to a closed anti-linear operator $S_\xi$ defined on a dense subset of $\mathcal{H}_\pi$.  Let $\Delta_\xi$ be the unique positive, self-adjoint operator, and let $J_\xi$ be the unique anti-unitary operator,
occurring in the polar decomposition
\begin{eqnarray*}
S_\xi=J_\xi\cdot\Delta_\xi^{1/2}=\Delta_\xi^{-1/2}J_\xi.
\end{eqnarray*}
$\Delta_\xi$ is called the {\it modular operator} and $J_\xi$ is called the {\it modular conjugation} or {\it modular involution} associated with the pair $(M,\xi)$. Tomita-Takesaki modular theory begins with the following remarkable theorem.
\begin{Th}[\cite{TAKES2}]
Let $\mathcal{N}$ be a $w^*$-algebra with the cyclic and separating  vector $\Omega$. Then $J_\Omega\Omega=\Delta_\Omega\Omega=\Omega$ and the following equalities hold:
\begin{eqnarray}\label{factor-its commutant}
J_\Omega^2=I,\;J_\Omega=J_\Omega^*, \;J_\Omega\,\mathcal{N}\,J_\Omega=\mathcal{N}^\prime.
\end{eqnarray}
\end{Th}
\begin{Rem}
Like above, we introduce operator $\widetilde{F}_\Omega$: $\mathcal{N}'\Omega\ni A^\prime\Omega\stackrel{\widetilde{F}_\Omega}{\mapsto}(A')^*\Omega$. If $F_\Omega$ is a closure of $\widetilde{F}_\Omega$ then next equalities hold
\begin{eqnarray*}
\Delta_\Omega=F_\Omega\, S_\Omega, \;\Delta_\Omega^{-1}=S_\Omega \,F_\Omega,\; F_\Omega=J_\Omega\Delta_\Omega^{-1/2}.
\end{eqnarray*}
\end{Rem}

The closure  of the set  of vectors $\left\{ AJ_\xi A\xi\right\}_{A\in M}$ forms a close cone $\mathfrak{P}_\xi$ in $\mathcal{H}_\pi$ which has several useful properties. Let us describe two of them, which are used below.
\begin{Prop}[\cite{TAKES2}]\label{vector-functional}
  For every positive normal\footnote{Normality means continuity with respect to the weak operator topology}  functional $\phi$ on $M$ there is a unique vector $\eta_\phi\in\mathfrak{P}_\xi$ such that $\phi(A)=\left( A\eta_\phi,\eta_\phi \right)$ for all $A\in M$. Let $\phi_1$, $\phi_2$ be the normal positive functionals on $M$, and let $\eta_{\phi_1}$,  $\eta_{\phi_2}$ be the corresponding vectors in $\mathfrak{P}_\xi$. Then the following equalities hold
  \begin{eqnarray}\label{inequality_vector_state}
  \left\| \eta_{\phi_1}-\eta_{\phi_2} \right\|\leq \left\|\phi_1-\phi_2 \right\|\leq\left\| \eta_{\phi_1}-\eta_{\phi_2} \right\|\cdot \left\| \eta_{\phi_1}+\eta_{\phi_2} \right\|.
  \end{eqnarray}
\end{Prop}
 Since $\pi$ is factor-representation of $B$, it follows from (\ref{factor-its commutant}) that the operators $\,^2\!\pi(g_1,g_2)=\pi(g_1)J_\xi\pi(g_2)J_\xi$, $g_1,g_2\in B$ form the irreducible representation of the group $B\times B$. Denote by ${\rm diag}\, B$ the subgroup of $B\times B$ consisting of elements of the view $(w,w)$. Let $D$ be the representation of $B$, which is the restriction of $\,^2\!\pi$ onto ${\rm diag}\, B$. Since $D(g)AJA\xi=\pi(g)AJ_\xi\pi(g)A\xi$, $g\in B$ then
\begin{eqnarray}\label{Cone_D_inv}
D(g)\mathfrak{P}_\xi=\mathfrak{P}_\xi \text{ for all } g\in B.
\end{eqnarray}
\begin{Prop}\label{tame_Prop}
  If $\pi$ is a stable representation of $B$ then $D$ is tame representation of $B$.
\end{Prop}
\begin{proof}
For $\eta\in\mathcal{H}_\pi$, define the positive functional $\omega_\eta$ on $\mathbf{C}[B]$ by $\omega_\eta(a)=\left( \pi(a)\eta,\eta \right)$. Then, using (\ref{factor-its commutant}), we obtain
\begin{eqnarray*}
\begin{split}
\omega_{D(g)\eta}(a)=\left( \pi(a)\pi(g)J\pi(g)J\eta, \pi(g)J\pi(g)J\eta \right)\\\stackrel{(\ref{factor-its commutant})}{=}\left( \pi\left( g^{-1}ag \right)\eta,\eta \right)=\omega_\eta\circ{\rm Ad}\,g(a).
\end{split}
\end{eqnarray*}
Hence, applying  the inequality (\ref{inequality_vector_state}), we have for $\eta\in\mathfrak{P}$
\begin{eqnarray}
\left\|D(g)\eta-\eta \right\|\leq\left\|\omega_\eta\circ{\rm Ad}\,g-\omega_\eta \right\|.
\end{eqnarray}
Thus, by the stability of $\omega_\eta$, the map $B\ni g\stackrel{\mathfrak{m}_\eta}{\mapsto} D(g)\eta\in \mathcal{H}_\pi$ is continuous with respect to the topology (\ref{topology}) on $B$ and norm topology on $\mathcal{H}_\pi$. Since the equality
\begin{eqnarray}\label{Cone_decom}
\begin{split}
M\xi\ni a\xi={\displaystyle\frac{1}{4}}\left[ \left( a+{\rm I}_{\mathcal{H}_\pi}\right)J_\xi\left( a+{\rm I}_{\mathcal{H}_\pi} \right)\xi- \left( a-{\rm I}_{\mathcal{H}_\pi}\right)J_\xi\left( a-{\rm I}_{\mathcal{H}_\pi} \right)\xi\right.\\
\left.i \left( a+i{\rm I}_{\mathcal{H}_\pi}\right)J_\xi\left( a+i{\rm I}_{\mathcal{H}_\pi} \right)\xi-i\left( a-i{\rm I}_{\mathcal{H}_\pi}\right)J_\xi\left( a-i{\rm I}_{\mathcal{H}_\pi} \right)\xi\right]
\end{split}
\end{eqnarray}
holds, the map $\mathfrak{m}_{a\xi}$ is continuous for each $a\in M$. Now proposition \ref{tame_Prop} follows from the density of the set $M\xi$ in $\mathcal{H}_\pi$.
\end{proof}

Fix the natural numbers $n$ and $N$, where $N>n$. Define an orthogonal projection $P_{nN}$ as follows: $P_{nN}=\left|B_{n\infty}\cap B_N\right|^{-1}\sum\limits_{g\in B_{n\infty}\cap B_N}D(g)$. Since $P_{n\;N}\geq P_{n\;N+1}\geq P_{n\;N+2}\ldots$ is a decreasing sequence, then there exists $\lim\limits_{N\to\infty}P_{n\;N}$ in the strong operator topology. Set $P_{n}=\lim\limits_{N\to\infty}P_{n\;N}$. By definition, $P_1\leq P_2\leq\ldots\leq P_n\leq\ldots$ and

\begin{Rem}\label{Rem_about_cone_inv}
  By definition of $P_n$ and \eqref{Cone_D_inv}, we have $P_{n}\mathfrak{P}_\xi\subset\mathfrak{P}_\xi$ and $\,^2\!\pi(g_1,g_2)P_n$ $=P_n\;^2\!\pi(g_1,g_2)$ for all $g_1,g_2\in B_n$.
\end{Rem}
Since $D$ is the tame representation, there exists the natural $n$ such that $P_n\neq 0$. We refer the least number $n\in\mathbb{N}\cup0$ for which $P_{n}\neq0$ as the {\it central depth} of the representation $\pi$ and denote it by ${\rm cd}(\pi)$. If $P_n\neq 0$ then there exists $a\xi\in M\xi$ such that $P_na\xi\neq 0$. Hence, using Remark \ref{Rem_about_cone_inv} and \eqref{Cone_decom} , we find nonzero $\eta \in\mathfrak{P}_\xi$ such that $P_n\eta=\eta$. Therefore, if $P_n\neq 0$ then $P_n\mathfrak{P}_\xi\neq 0$.

Let $M_*^+$ be the cone of the positive functionals from $M_*$, and let    $M_*^+(n)$ $=\left\{ \phi\in M_*^+:\phi=\phi\circ{\rm Ad}\,\pi(g) \text{ for all }g\in B_{n\,\infty}\right\}$. If $\phi\in M_*^+(n)$, then, by Proposition \ref{vector-functional}, there exists unique $\eta_\phi\in \mathfrak{P}_\xi$ for which  $\phi(A)=\left( A\eta_\phi,\eta_\phi \right)$ for all $A\in M$.
Since $\phi(A)=\phi\left( \pi(g^{-1})\,A\,\pi(g)\right)=\left( A\pi(g)\eta_\phi,\pi(g)\eta_\phi \right)=\left( AD(g)\eta_\phi,D(g)\eta_\phi \right)$ for all $A\in M$ and $g\in B_{n\,\infty}$, then, using \eqref{Cone_D_inv}, we obtain that $\eta_\phi \in P_n\mathfrak{P}_\xi$.

Therefore, ${\rm cd}(\pi)=\min\left\{ n: M_*^+(n)\neq 0 \right\}$.
The next statement is well known.
\begin{Prop}[\cite{Lie,Ol1}]\label{Prop_depth}
Let $n={\rm cd}(\pi)$. The projections $P_n$ and $D(s)\cdot P_n\cdot D(s)$ are orthogonal for each $s\notin \mathfrak{S}_n\cdot\mathfrak{S}_{n\infty}$.
\end{Prop}
For completeness we will give bellow the proof of this proposition.
First, we need the following auxiliary lemma.
\begin{Lm}\label{Tame_Lemma_existing}
  For each $k\in\mathbb{N}$ there exists $\lim\limits_{m\to\infty}D((k\; m))$ in the weak operator topology. Let $Q_k=\lim\limits_{m\to\infty}D((k\; m))$. Then $\left\{ Q_k  \right\}$ is the set of the pairwise commuting orthogonal projections such that $D(s)\cdot Q_k\cdot D(s)=Q_{s(k)}$ for all $k\in\mathbb{N}$ and $s\in\mathfrak{S}_\infty$.
\end{Lm}
\begin{proof}
By proposition \ref{tame_Prop}, $\left\{ \left( D((k\;m))\eta,\eta \right)  \right\}_{m\in\mathbb{N}}$ is Cauchy sequence    for each $\eta\in \mathcal{H}_\pi$.
Let us prove that $Q_k^2=Q_k$.

For this task,  we fix the unit vector $\eta\in\mathcal{H}_\pi$ and a number $\epsilon>0$. Then there exists  $m(\epsilon,\eta)\in\mathbb{N}$ such that
\begin{eqnarray}
\|D((l\;N))\xi-\xi\|<\epsilon \text{ (see proposition \ref{tame_Prop})},\label{inequality_1}\;\text{ and } \\
\left|\left( D((k\;l))\cdot Q_k\eta,\eta\right)-\left(Q_k^2\eta,\eta\right)\right|<\epsilon\;\text{ for all } l,N>m(\epsilon,\eta)\label{inequality_2}.
\end{eqnarray}
Now we fix $l>m(\epsilon,\eta)$ and find $N(\epsilon,l)>m(\epsilon,\eta)$ such that
\begin{eqnarray}
\left| \left(D((k\;l))\cdot D((k\;N))\eta,\eta\right)-\left(D((k\;l))\cdot Q_k\eta,\eta\right)\right|<\epsilon\text{ for all } N>N(\epsilon,l).\label{inequality_4}\;\;
\end{eqnarray}
Applying (\ref{inequality_2}) and (\ref{inequality_4}), we have
\begin{eqnarray*}
\left| \left(D((k\;l))\cdot D((k\;N))\eta,\eta\right)-\left(Q_k^2\eta,\eta\right)\right|<2\epsilon \text{ for all }N>N(\epsilon,m).
\end{eqnarray*}
Hence, using (\ref{inequality_1}), we obtain
\begin{eqnarray*}
\left| \left(D((k\;N))\eta,\eta\right)-\left(Q_k^2\eta,\eta\right)\right|<3\epsilon \text{ for all }N>N(\epsilon,n).
\end{eqnarray*}
Therefore,
$\left| \left(Q_k\eta,\eta\right)-\left(Q_k^2\eta,\eta\right)\right|<3\epsilon.$
The equalities $Q_kQ_j=Q_jQ_k$, $Q_k^*=Q_k$ and $D(s)\cdot Q_k\cdot D(s)=Q_{s(k)}$ are obvious.
\end{proof}
\begin{Lm}\label{Lemma_triviality}
  The following hold:
  \begin{itemize}
    \item {\rm a)} $D((k\;j))Q_kQ_j=Q_kQ_j$ for all $k,j$;
    \item {\rm b)} if $\zeta^{(k)}=\underbrace{(0,\ldots,0,}_{k-1}\zeta,0,0,\ldots)\in\;_0\mathbb{Z}_2^\infty$, then $D\left(\zeta^{(k)}\right)\cdot Q_k=Q_k$,
  \end{itemize}
\end{Lm}
\begin{proof}
 Take unit vector $\eta\in\mathcal{H}_\pi$ and $\epsilon>0$. Using proposition \ref{tame_Prop}, find $N(\epsilon)$ such  that
\begin{eqnarray}\label{<1}
\begin{split}
&\left|\left(D((k\;\,j))D((k\;\,n_1))D((j\;\,n_2)) \eta,\eta\right)-\left(D((k\;\,j))D((k\;\,N_1))D((j\;\,N_2)) \eta,\eta\right)\right|<\epsilon,\;\;\;\;\;\;\\
&\left|\left(D((k\;\,n_1))D((j\;\,n_2)) \eta,\eta\right)-\left(D((k\;\,N_1))D((j\;\,N_2)) \eta,\eta\right)\right|<\epsilon,\\
&\left\|D((n_1\;n_2))\eta-\eta \right\|<\epsilon
\end{split}
\end{eqnarray}
for all $n_1, n_2, N_1, N_2>N(\epsilon)$.
Hence, applying lemma  \ref{Tame_Lemma_existing},  we obtain
\begin{eqnarray*}
&\left|\left(D((k\;\,j))D((k\;\,n_1))D((j\;\,n_2)) \eta,\eta\right)-\left(D((k\;\,j))D((k\;\,N_1))Q_j \eta,\eta\right)\right|<\epsilon,\;\;\;\;\;\;\\
&\left|\left(D((k\;\,n_1))D((j\;\,n_2)) \eta,\eta\right)-\left(D((k\;\,N_1))Q_j \eta,\eta\right)\right|<\epsilon ~\text{ for all }~ n_1, n_2, N_1 >N(\epsilon).
\end{eqnarray*}
For the same reason
\begin{eqnarray}\label{<2}
\begin{split}
&\left|\left(D((k\;\,j))D((k\;\,n_1))D((j\;\,n_2)) \eta,\eta\right)-\left(D((k\;\,j))Q_kQ_j \eta,\eta\right)\right|<\epsilon,\;\;\;\;\;\;\\
&\left| \left(D((k\;n_1))\cdot D((j\;n_2))\eta,\eta\right)-\left(Q_k\cdot Q_j\eta,\eta\right)\right|<\epsilon ~\text{ for all }~ n_1, n_2 >N(\epsilon).\;\;\;\;\;\;
\end{split}
\end{eqnarray}
Since $(k\;j)(k\;n_1)(j\;n_2)=(k\;n_1)(j\;n_2)(n_1\;n_2)$, then, applying (\ref{<1}) and (\ref{<2}), we obtain
\begin{eqnarray*}
\left| \left(D((k\;n_1))\cdot D((j\;n_2))\eta,\eta\right)-\left(D((k\;j))\cdot Q_k\cdot Q_j\eta,\eta\right)\right|<2\epsilon ~\text{ for all }~ n_1,n_2>N(\epsilon).
\end{eqnarray*}
Hence, by (\ref{<2}), we have $\left|\left(Q_k\cdot Q_j\eta,\eta\right)-\left(D((k\;j))\cdot Q_k\cdot Q_j\eta,\eta\right)\right|<3\epsilon$. The equality {\rm a)} is proved.

Let us prove {\rm b)}. Again, given unit vector $\eta\in\mathcal{H}_\pi$ and $\epsilon>0$, using proposition \ref{tame_Prop} and lemma \ref{Tame_Lemma_existing}, we find $n\in\mathbb{N}$ such that
\begin{eqnarray}
 \left\|D\left( \zeta^{(n)} \right)\eta-\eta\right\|<\epsilon,\label{<<1}\\
\left| \left( D((k\;n))\eta,\eta \right)-\left(Q_k\eta,\eta\right) \right|<\epsilon\label{<<2}\text{ and } \\
\left| \left( D\left( \zeta^{(k)} \right)\cdot D((k\;n))\eta,\eta \right)-\left(D\left( \zeta^{(k)} \right)\cdot Q_k\eta,\eta\right) \right|<\epsilon.\label{<<3}
\end{eqnarray}
Then, by (\ref{<<1}) and (\ref{<<3}), we obtain $$\left| \left(  D((k\;n))\eta,\eta \right)-\left(D\left( \zeta^{(k)} \right)\cdot Q_k\eta,\eta\right) \right|<2\epsilon.$$ Hence, applying (\ref{<<2}), we conclude that $\left|\left(Q_k\eta,\eta\right) -\left(D\left( \zeta^{(k)} \right)\cdot Q_k\eta,\eta\right) \right|$ $<3\epsilon$.
\end{proof}
\begin{proof}[{\bf The proof of proposition \ref{Prop_depth}}] Lemma \ref{Tame_Lemma_existing} shows that there exists $\lim\limits_{N\to\infty}\prod\limits_{k=n+1}^N Q_k$ with respect to the strong operator topology, and $Q_n^{(\infty)}=s-\lim\limits_{N\to\infty}\prod\limits_{k=n+1}^N Q_k$ is an orthogonal projection. Set $\mathbb{A}=\left\{ k,s(k)  \right\}_{k=n+1}^\infty\subset\mathbb{N}$. By Lemma \ref{Lemma_triviality},
\begin{eqnarray}\label{formula_for_intersection}
Q_n^{(\infty)}=P_n\;\;\text{ and } \;\; D(s)\cdot P_n\cdot D(s)\cdot P_n=\prod\limits_{k\in\mathbb{A}}Q_k.
\end{eqnarray}
By assumption of proposition \ref{Prop_depth}, $s\notin \mathfrak{S}_n\cdot\mathfrak{S}_{n\infty}$. Therefore, there exists $m$ $\in\mathbb{A}\cap\{1,\ldots,n\}$. Without loss of generality we can assume that $m=n$. Thus, if $D(s)\cdot P_n\cdot D(s)\cdot P_n\neq0$, then, applying (\ref{formula_for_intersection}) and lemma \ref{Lemma_triviality}, we obtain that ${\rm cd}(\pi)\leq n-1$.
This contradicts the fact that ${\rm cd}(\pi)=n$.
\end{proof}
\begin{proof}[{\bf The proof of Theorem \ref{induction_theorem}}]\label{Proof_th_ind} To  prove {\bf a}) we fix the unit vector $\eta\in \mathcal{H}_\pi$ and, applying proposition \ref{tame_Prop}, find $n$ such that
\begin{eqnarray*}
\|D(g)\eta-\eta\|<1/2 \text{ for all } g\in B_{n\infty}.
\end{eqnarray*}
Hence,
\begin{eqnarray*}
\left\|\eta- \left|B_{n\infty}\cap B_N\right|^{-1}\sum\limits_{g\in B_{n\infty}\cap B_N}D(g)\eta\right\|<1/2 \text{ for all } N>n.
\end{eqnarray*}
Therefore, $\|P_n\eta-\eta\|<1/2$, where $P_n$ is the strong limit of the decreasing sequence of the projections $\left\{ P_{nN}=\left|B_{n\infty}\cap B_N\right|^{-1}\sum\limits_{g\in B_{n\infty}\cap B_N}D(g)  \right\}_{N=n+1}^\infty$. It follows that
$P_n\eta\neq0$ and $D(g)P_n\eta=P_n\eta$ for all $g\in B_{n\infty}$. Therefore, for each $a\in M$ and $g\in B_{n\infty}$,  we have
\begin{eqnarray*}
(aP_n\eta,P_n\eta)=\left( a\pi(g)J\pi(g)JP_n\eta,\pi(g)J\pi(g)JP_n\eta \right)=\left( \pi\left( g^{-1} \right)a\pi(g)P_n\eta,P_n\eta \right).
\end{eqnarray*}
Since $P_n\eta\neq0$, this gives {\bf a}).

Now we will prove {\bf b}). Let $n={\rm cd}\,(\pi)$. Fix nonzero $\omega\in M_*^+(n)$. Applying proposition \ref{vector-functional}, we find the unique vector $\eta_\omega\in\mathfrak{P}_\xi$ such that $\omega(a)=\left( a\eta_\omega,\eta_\omega \right)$ for each $a\in M$ and $P_n\eta_\omega=\eta_\omega$. By proposition \ref{Prop_depth}, the vectors $\eta_\omega$ and $D(s)\eta_\omega$ are orthogonal for each $s\in B\setminus B_n\cdot B_{n\infty}$. Since the vectors $\eta_\omega$ and $D(s)\eta_\omega$ belong to the cone $\mathfrak{P}_\xi$, we can  apply Lemma 1.12({\rm IX}) from \cite{TAKES2}, according to which the supports $({\rm supp}\,\omega)$ and $\pi(s^{-1})\left( {\rm supp}\,\omega \right)\pi(s)$ of the corresponding vector states are orthogonal.
\end{proof}
\section{Quasi-equivalence of the induced representations.}
First, we recall that two representations $\Pi$  and $\widetilde{\Pi}$ of group $G$ with representation spaces $\mathcal{H}$ and $\widetilde{\mathcal{H}}$, respectively, are said to be disjoint is there is no non-zero intertwining operator between $\mathcal{H}$ and $\widetilde{\mathcal{H}}$. Here, an intertwining operator between  $\Pi$  and $\widetilde{\Pi}$
 is a continuous linear operator $T:\mathcal{H}\rightarrow \widetilde{\mathcal{H}}$ such that $T\,\Pi(g)=\widetilde{\Pi}(g)\, T$ for all $g\in G$.

 Let $B(\mathcal{H}_i)$ $(i=1,2)$ be a set of all bounded linear operators on $\mathcal{H}_i$, and let $M_i$ be a $w^*$-subalgebra in $B(\mathcal{H}_i)$.
 Suppose that identity operator $I_{\mathcal{H}_i}$ lies in $M_i$.  The mapping $M_1\ni m_1\stackrel{\mathfrak{a}_{_{\mathcal{H}_2}}}{\mapsto}m_1\otimes I_{\mathcal{H}_2}\in B(\mathcal{H}_1\otimes\mathcal{H}_2)$ is an isomorphism of $M_1$ onto   $w^*$-subalgebra $M_1\otimes I_{\mathcal{H}_2}\subset B(\mathcal{H}_1\otimes\mathcal{H}_2)$, which is called the {\it ampliation} of $M_1$ onto $M_1\otimes I_{\mathcal{H}_2}$ \cite{Dixmier}.

 Take orthogonal projection $E'$ from $M_1'=\left\{m'\in B(\mathcal{H}_1)\big| \;m'\,m=mm' \right.  ~ \text{ for all }~$ $\left. m\in M_1 \right\}$. Then the mapping $M_1\ni m \stackrel{\mathfrak{i}\!_{E'}}{\mapsto} E'\,m=E'mE'\in E'M_1E'$ is a homomorphism, which is called the {\it induction} of $M_1$ onto $E'\,M_1\,E'$ \cite{Dixmier}.

Let $\left[M_1'E'\mathcal{H}_1 \right]$ be a closure of a set $M_1'E'\mathcal{H}_1$. Then orthogonal projection $\mathfrak{s}(E')$ onto $\left[M_1'E'\mathcal{H}_1 \right]$ lies in centrum $C(M_1)=M_1\cap M_1'$ of $M_1$. We call  $\mathfrak{s}(E')$ the {\it central support} of $E'$. If $F$ is an orthogonal projection from $C(M_1)$ such that $FE'=E'$ then $F\,\mathfrak{s}(E')=\mathfrak{s}(E')$.

Suppose that $mE'=0$ for some $m\in M_1$. Then $mm'E'=0$ for all $m'\in M_1'$. Therefore, $m\,\mathfrak{s}(E')=0$. Thus the induction $\mathfrak{i}\!_{E'}$ if and only  if is an isomorphism the equality    $\mathfrak{s}(E')=I$ is true.

 Consider two unitary representations $\Pi_1$ and $\Pi_2$ of a group $G$ in Hilbert spaces $\mathcal{H}_1$ and $\mathcal{H}_2$, respectively.  Set $M_i=\Pi_i(G)''$ $(i=1,2)$. If $M_i'\ni E_i'$ is an orthogonal   projection then denote by $\Pi_i^{E_i'}$ the unitary representation of $G$, generated by operators $\Pi_i(g)E_i'$, $g\in G$.
 \begin{Def}
The representation $\Pi_1$ and $\Pi_2$,    satisfying one of the following equivalent conditions:
\begin{itemize}
  \item A)  there exist orthogonal projections $E_1'\in M_1'$, $E_2'\in M_2'$, Hilbert spaces $\mathcal{K}_1$, $\mathcal{K}_2$ and the complex linear isometrical isomorphisms  $W: E_1'\mathcal{H}_1$ $\mapsto E_2'\mathcal{H}_2$, $W_i:(E_i'\mathcal{H}_i\otimes \mathcal{K}_i)\mapsto \mathcal{H}_i$, $i=1,2$  such that $W\Pi_1^{E_1'}(g)$ $=\Pi_2^{E_2'}(g)\;W$ and  $W_i\,\left(\Pi_i(g)E_i'\otimes I_{\mathcal{K}_i} \right) $ $=\Pi_i(g)\,W_i$, $i=1,2$ for all $g\in G$;
      \item B) any non-zero representation $\Pi_1^{E_1'}$ is not disjoint to $\Pi_2$, and any  non-zero representation $\Pi_2^{E_2'}$ is not disjoint to $\Pi_1$;
  \item C) there exists  an isomorphism  $\Phi: M_1\mapsto M_2$ such that $\Phi(\Pi_1(g))=\Pi_2(g)$ for all $g\in G$,
\end{itemize}
are called quasi-equivalent.
\end{Def}
The following statement  gives the general form of normal  isomorphism:
  \begin{Prop}[\cite{Dixmier}]\label{isomorphism_structure}
Let $\Phi$ normal isomorphism of $M_1$ onto $M_2$ Then exists  an ampliation $\Phi_1$ of $M_1$ onto $w^*$-algebra $\widetilde{M}_1$, an induction $\Phi_2 : \widetilde{M}_1\mapsto \widehat{M}_1$ of $\widetilde{M}_1$ onto $w^*$-algebra $\widehat{M}_1$, and spatial isomorphism $\Phi_3:\widehat{M}_1\mapsto M_2 $ of $\widehat{M}_1$ onto $M_2$ such that $\Phi=\Phi_3\circ\Phi_2\circ\Phi_1$.
\end{Prop}
\begin{Prop}\label{quasi_equivalent_ind}
Let $H$ be a subgroup of the countable group $G$, and let $\pi_1$ and $\pi_2$ be the quasi-equivalent unitary representations of  $H$. Then the representations $\Pi_1={\rm Ind}^G_H\,\pi_1$ and  $\Pi_2={\rm Ind}^G_H\,\pi_2$ are  quasi-equivalent.
\end{Prop}
\begin{proof}
Let $\pi_1$ and  $\pi_2$ act in the Hilbert spaces $\mathcal{K}_1$ and $\mathcal{K}_2$, respectively. Set $M_i=\pi_i(H)''$, $i=1,2$.
Since  let $\pi_1$ and $\pi_2$ are quasi-equivalent, there exists a normal isomorphism $\Phi$ of $M_1$ onto $M_2$ such that
\begin{eqnarray}\label{isomorphism_Phi}
\Phi(\pi_1(g))=\pi_2(g)  \;~\text{ for all }~ g\in H.
\end{eqnarray}
It follows from proposition \ref{isomorphism_structure} that
\begin{eqnarray}\label{decomposition_of_Phi}
\Phi=\Phi_3\circ\Phi_2\circ\Phi_1,
\end{eqnarray}
 where $\Phi_1$ is ampliation, $\Phi_2$ is induction and $\Phi_3$ is some spatial isomorphism.

Denote by $X=G/H$ the space of left cosets of the subgroup $H$ in $G$. Define mapping  $X\ni x\stackrel{\mathfrak{s}}{\mapsto}\mathfrak{s}(x)\in G$ with the properties $x=\mathfrak{s}(x)\, H$ and $\mathfrak{s}\left(H\right)=e$, where $e$ is identity  of $G$. Denote by $\widetilde{g}$ the class in $X$ containing $g\in G$.

The operators of  representation $\Pi_i={\rm Ind}^G_H\pi_i$, $i=1,2$ act on $v\in l^2\left(\mathcal{K}_i,X\right)$ as follows
\begin{eqnarray}\label{G_H_ind_action}
\left(\Pi_i(g)v\right)(x)=c(g,x)v\left(g^{-1}x\right),
\end{eqnarray}
where $c(g,x)= \pi_i\left(\left( \mathfrak{s}(x) \right)^{-1}g\;\mathfrak{s}\left(g^{-1}x\right)\right)$. Notice that $\left( \mathfrak{s}(x) \right)^{-1}g\;\mathfrak{s}\left(g^{-1}x\right)$ lies in $H$.

By proposition \ref{isomorphism_structure}, we see that $\Phi_1(m)=m\otimes I_\mathcal{L}$, where $ I_\mathcal{L}$ is identity operator in some Hilbert space $\mathcal{L}$; i. e. $\widetilde{M}_1=M_1\otimes \mathbb{C}I_\mathcal{L}$. There exists an orthogonal projection $\widetilde{E}'\in \widetilde{M}'_1$ such that  $\Phi_2=\mathfrak{i}_{_{\widetilde{E}'}}$. Let $U:\widetilde{E}'\,(\mathcal{K}_1\otimes\mathcal{L})\mapsto \mathcal{K}_2$ be an isometry of $\widetilde{E}'\,(\mathcal{K}_1\otimes\mathcal{L})$ onto $\mathcal{K}_2$ such that
\begin{eqnarray}\label{spatial}
U(\pi_1(g)\otimes I_\mathcal{L})\widetilde{E}'\,U^{-1}=\pi_2(g) \;~\text{ for all }~ g\in H.
\end{eqnarray}

 Let us introduce the notation ${\mathcal{M}}_i$ for $\Pi_i(G)''$, $i=1,2$.

 Let operator $\breve{\Phi}_1\left( \Pi_1(g)\right)$ acts on $\breve{v}$ $\in l^2\left(\mathcal{K}_1\otimes \mathcal{L},X\right)=l^2\left(\mathcal{K}_1,X\right)\otimes\mathcal{L}$ as follows
 \begin{eqnarray}\label{Phi_1}
 \left( \breve{\Phi}_1\left( \Pi_1(g)\right)\breve{v}\right)(x)=\Phi_1(c(g,x))\breve{v}\left(g^{-1}x\right),
 \end{eqnarray}
 where $\Phi_1(c(g,x))=c(g,x)\otimes I_\mathcal{L}$.
  It is clear that $ \breve{\Phi}_1$ defines ampliation  of ${\mathcal{M}}_1$ onto $\widetilde{\mathcal{M}}_1={\mathcal{M}}_1\otimes\mathbb{C}I_\mathcal{L}$.

 Let us define orthogonal projection $\widetilde{\mathcal{E}}'\in \widetilde{\mathcal{M}}_1'$ as follows
 \begin{eqnarray}\label{induction_tilde}
 \left( \widetilde{\mathcal{E}}'\breve{v}\right)(x)=\widetilde{E}'\,\breve{v}(x), ~\text{ where }~ \breve{v}\in l^2\left(\mathcal{K}_1\otimes \mathcal{L},X\right).
 \end{eqnarray}
Since $\Phi$ is an isomorphism, the induction $\Phi_2=\mathfrak{i}_{_{\widetilde{E}'}}$ is an isomorphism too. Therefore, central support of $\widetilde{E}'$ is equal to $I_{_{\mathcal{K}_1\otimes \mathcal{L}}}$; i. e. $$\left[\widetilde{M}_1'\widetilde{E}'\left( \mathcal{K}_1\otimes \mathcal{L}\right) \right]= \mathcal{K}_1\otimes \mathcal{L}.$$
Hence, using \eqref{induction_tilde}, we obtain that
\begin{eqnarray*}
\left[\widetilde{\mathcal{M}}_1' \widetilde{\mathcal{E}}'\; l^2\left(\mathcal{K}_1\otimes \mathcal{L},X\right)\right]= l^2\left(\mathcal{K}_1\otimes \mathcal{L},X\right).
\end{eqnarray*}
Thus the induction $\breve{\Phi}_2=\mathfrak{i}_{_{\widetilde{\mathcal{E}}'}}$ is an isomorphism.

At last, define the isometry $\mathcal{U}:\widetilde{\mathcal{E}}'\,l^2\left(\mathcal{K}_1\otimes \mathcal{L},X\right)\mapsto l^2\left(\mathcal{K}_2,X \right)$ as follows
\begin{eqnarray}\label{U_isometry}
\left(\mathcal{U}\breve{v} \right)(x)=U\breve{v}(x),\; ~\text{ where }~ \breve{v}\in\widetilde{\mathcal{E}}' l^2\left(\mathcal{K}_1\otimes \mathcal{L},X\right).
\end{eqnarray}
 Let us introduce the spatial isomorphism $\breve{\Phi}_3: \widetilde{\mathcal{E}}' \widetilde{\mathcal{M}}_1\mapsto \mathcal{M}_2$ (see \eqref{spatial}) by $\breve{\Phi}_3(\widetilde{m})=\mathcal{U}\, \widetilde{m}\,\mathcal{U}^{-1}$, $\widetilde{m}\in \widetilde{\mathcal{E}}' \widetilde{\mathcal{M}}_1$. It follows from above that $\breve{\Phi}=\breve{\Phi}_3\circ\breve{\Phi}_2\circ\breve{\Phi}_1$ is an isomorphism of $\mathcal{M}_1$ onto $\mathcal{M}_2$ and the equality $\breve{\Phi}\circ\Phi_1=\Phi_2$ holds.
 \end{proof}
 \subsection{Asymptotic character of the stable factor representation.}
Here we give a canonical construction of a indecomposable character for a stable factor representation of $B$.
Let ${\rm Cl}(g)=\left\{ hgh^{-1}  \right\}_{h\in B}$ be the  conjugacy class of an element $g\in B$. We will denote by ${\rm Cl}_{n\infty}(g)$ the set ${\rm Cl}(g)\cap B_{n\infty}$.
\begin{Prop}\label{asymp_char}
Let $\pi$ be a stable factor-representation of $B$.
  Fix the sequence $g_n\in {\rm Cl}_{n\infty}(g)$. Then there exists $\lim\limits_{n\to\infty}\pi\left(g_n\right)$ in the weak operator topology. Since $\pi$ is factor-representation, we have $w-\lim\limits_{n\to\infty}\pi\left(g_n\right)=\chi_{_\pi}^{as}(g)\cdot {\rm I}_{\mathcal{H}_\pi}$, where ${\rm I}_{\mathcal{H}_\pi}$ stands for the identity operator on $\mathcal{H}_\pi$, and  $\chi_{_\pi}^{as}$ is positive definite function on  $B$ with the properties:
  \begin{itemize}
    \item {\it i}) $\chi_{_\pi}^{as}(g)$ does not depend on the choice of the sequence $\left\{ g_n  \right\}$;
    \item {\it ii}) $\chi_{_\pi}^{as}$ is central; i. e. $\chi_{_\pi}^{as}(gh)=\chi_{_\pi}^{as}(hg)$ for all $g,h\in B$;
    \item {\it iii}) $\chi_{_\pi}^{as}$ is an indecomposable character; i.e   $\chi_{_\pi}^{as}$ satisfies the relations (\ref{II_1-mult_general}) and (\ref{character_formula}).
  \end{itemize}
\end{Prop}
  We will call $\chi_{_\pi}^{as}$ as an {\it asymptotical character} of the representation $\pi$.
  \begin{proof}[Proof of proposition \ref{asymp_char}]
  For convenience we assume that $\pi$ acts in Hilbert space $\mathcal{H}_\pi$ with the cyclic and separating  vector $\xi$; i. e. $\left[\pi(B)\xi \right]=\left[\pi(B)'\xi \right]=\mathcal{H}_\pi$.
  Take two sequences $\{g_n\}$, $\{h_n\}$ such that $g_n,h_n\in {\rm Cl}_{n\infty}(g)$.  Since $\pi$ is stable representation (see Definition \ref{stable_Def}), we have
  \begin{eqnarray}\label{norm_estimate}
  \left| \omega_\eta\left( \pi(b_n)x\pi(b_n)^*\right)-\omega_\eta(x)\right|<\epsilon_n(\eta)\|x\|,\;~\text{ where }~ x\in \pi(B)''
    \end{eqnarray}
    and $\lim\limits_{n\to\infty}\epsilon_n(\eta)=0$ for all $\eta\in\mathcal{H}_\pi$, $b_n\in B_{n\infty}$.  If $N>n$ then there exists $b_n\in B_{n\infty}$ such that $g_{_N}=b_ng_nb_n^{-1}$. Hence, applying \eqref{norm_estimate}, we have
      \begin{eqnarray*}
  \left| \omega_\eta\left( \pi(g_n)\right)-\omega_\eta(\pi(g_{_N}))\right|<\epsilon_n(\eta).
    \end{eqnarray*}
    Therefore, a sequence  $\pi(g_n)$ converges to operator $A\in\pi(B)''$ in weak operator topology. Since $\pi$ is factor-representation, we obtain that $A=\chi_{_{\pi}}(g)^{as}\cdot {\rm I}_{\mathcal{H}_\pi}$, where $\chi_{_{\pi}}^{as}(g)\in\mathbb{C}$. Further we note that there exists $b_n\in B_{n\infty}$ such that $h_n=b_ng_nb_n^{-1}$. Hence, using \eqref{norm_estimate}, we have
    \begin{eqnarray*}
    \left| \omega_\eta\left(\pi(h_n)\right)-\omega_\eta(\pi(g_n))\right|<\epsilon_n(\eta).
    \end{eqnarray*}
  This gives {\it i)}. To prove {\it ii)} it is sufficient to note  that ${\rm Cl}_{n\infty}(gh)={\rm Cl}_{n\infty}(hg)$.

  Let us prove {\it iii)}. In this case it suffices to show that $\chi_{_\pi}^{as}$ has a multiplicativity property \cite{DN}; i.e. the equality $\chi_{_{\pi}}^{as}(gh)=\chi_{_{\pi}}^{as}(g)\,\chi_{_{\pi}}^{as}(h)$ holds for all $g,h\in B$ such that ${\rm supp}\,g\;\cap\; {\rm supp}\,h=\emptyset$. Suppose that
  \begin{eqnarray}
  {\rm supp}\,g\;\cup\; {\rm supp}\,h\;\subset\;\{1,2,\ldots, K\}.
  \end{eqnarray}
  Set $n>K$, $t_n=\prod\limits_{i=1}^K(i\;\, i+n)$, $g_n=t_ngt_n^{-1}$,  $h_n=t_nht_n^{-1}$ and $(gh)_n=t_nght_n^{-1}$.  It is
clear that $g_n\in{\rm Cl}_{n\infty}(g), h_n\in{\rm Cl}_{n\infty}(h)$ and $g_nh_n=(gh)_n\in {\rm Cl}_{n\infty}(gh)$. Take sequence $\left\{\,^n\!s_p=\prod\limits_{j\in{\rm supp}\,h_n}(j\;\,j+pK)\right\}\subset B_{n\infty}$. Then $\,^n\!s_p\,g_n(\,^n\!s_p^{-1})$
 and, by \eqref{norm_estimate}, we have
  \begin{eqnarray*}
 \left| \omega_\eta\left( \pi(^n\!s_p)\pi(g_n)\pi(h_n)\pi(^n\!s_p^{-1})\right)-\omega_\eta(\pi(g_n)\,\pi(h_n))\right|<\epsilon_n(\eta).
\end{eqnarray*}
Hence, using definition of $\,^n\!s_p$, we conclude
\begin{eqnarray*}
\left| \omega_\eta\left( \pi(g_n)\pi(^n\!s_p)\pi(h_n)\pi(^n\!s_p^{-1})\right)-\omega_\eta(\pi(g_n)\,\pi(h_n))\right|<\epsilon_n(\eta).
\end{eqnarray*}
Since $\,^n\!s_p\,h_n\,^n\!s_p^{-1}\,\in  {\rm Cl}_{(n+Kp)\infty}(h)$, it follows from above that
\begin{eqnarray*}
\lim\limits_{p\to\infty} \left| \omega_\eta\left( \pi(g_n)\Pi(^n\!s_p)\pi(h_n)\pi(^n\!s_p^{-1})\right)-\omega_\eta(\pi(g_n)\,\pi(h_n))\right|\\
=\left| \chi_{_\pi}^{as}(h)\omega_\eta\left( \pi(g_n))\right)-\omega_\eta(\pi((gh)_n))\right|<\epsilon_n(\eta).
\end{eqnarray*}
After passing to the limit $n\to\infty$, we obtain
$
\left|\chi_{_\pi}^{as}(g)\, \chi_{_\pi}^{as}(h)-\chi_{_\pi}^{as}(gh)\right|=0.
$  This gives {\it iii)}.
\end{proof}

\subsection{The asymptotical transposition.}
Let $\pi$ be the unitary stable representation of group $B$ in Hilbert space $\mathcal{H}_{_\pi}$.
An important characteristic of the stable representations is the existence of the weak limits $\lim\limits_{k\to\infty}\pi((n\;\,k))$. These operators were introduced by A. Okounkov in \cite{Ok1},\cite{Ok2} when studying  of the finite type factor - representation of the infinite symmetric group $\mathfrak{S}\infty$ and admissible representation of the pair $(\mathfrak{S}_\infty\times\mathfrak{S}_\infty,{\rm diag}(\mathfrak{S}_\infty)$, where ${\rm diag}(\mathfrak{S}_\infty))$ $= \left\{ (s,s)\right\}_{s\in\mathfrak{S}_\infty}$. In \cite{DN} this operators were used in the classification of a finite characters on infinite wreath product. The following statement is the generalisation of Lemma 17 from \cite{DN} on the case of the stable representations of group $B$.
\begin{Lm}\label{o_transp}
The sequence $\{\pi((n\,\;k))\}_{k=1}^\infty$ converges to operator $\mathcal{O}_n=\lim\limits_{k\to\infty}\pi((n\,\;k))$ in the weak operator topology. $w^*$-algebra $\mathfrak{A}$, generated by $\pi\left( \;_0\mathbb{Z}_2^\infty\right)$ and $\left\{ \mathcal{O}_n\right\}_{n=1}^\infty$ is abelian. The equality $\pi(s)\, \mathcal{O}_j\,\pi(s^{-1})=\mathcal{O}_{s(j)}$ holds for all $s\in\mathfrak{S}_\infty$ and $j\in \left\{1,2,\ldots \right\}$.
\end{Lm}
\begin{proof}
It follows from definition \ref{stable_Def} that $\left\{ \left(\pi((n\;\,k))\eta,\eta \right)\right\}_{k=1}^\infty$ is Cauchy sequence for each $\eta\in\mathcal{H}_{_\Pi}$. Therefore, the sequence $\left\{ \pi((n\;\,k))\right\}_{k=1}^\infty$ converges to operator $\mathcal{O}_n$ in the weak operator topology. The equality $\mathcal{O}_j\mathcal{O}_l=\mathcal{O}_l\mathcal{O}_j$ follows from definition $\mathcal{O}_k$.

In the case $m\neq l$ the equality $\pi\left(1^{(m)}\right)\,\mathcal{O}_l=\mathcal{O}_l\,\pi\left(1^{(m)}\right)$, where       $1^{(m)}=(\underbrace{0,0,\ldots,0}_{m-1},1,0,\ldots)\in \;_0\mathbb{Z}_2^\infty$, follows from definition of $\mathcal{O}_l$. If $m=l$ then
\begin{eqnarray*}
\left(\mathcal{O}_m \,\pi\left(1^{(m)}\right) \,\eta,\eta\right)
=\lim\limits_{j\to\infty}\left( \pi((m\;\,j))\,\pi\left(1^{(m)}\right)\,\eta,\eta\right)\\
\lim\limits_{j\to\infty}\left( \pi\left(1^{(j)}\right)\pi((m\;\,j))\,\eta,\eta\right).
\end{eqnarray*}
Hence, by stability of $\pi$, we have
\begin{eqnarray*}
\left(\mathcal{O}_m \,\pi\left(1^{(m)}\right) \,\eta,\eta\right)=\lim\limits_{j\to\infty}\left( \pi\left(1^{(j)}\right)\pi((m\;\,j))\, \pi\left(1^{(j)}\right)\, \pi\left(1^{(j)}\right)\eta,\eta\right)\\
\stackrel{stability}{=} \lim\limits_{j\to\infty}\left(\pi((m\;\,j))\, \pi\left(1^{(j)}\right)\,\eta,\eta\right)= \lim\limits_{j\to\infty}\left(\pi\left(1^{(m)}\right)\,\pi((m\;\,j))\, \eta,\eta\right)\\
=\left(\pi\left(1^{(m)}\right) \,\mathcal{O}_m\,\eta,\eta\right)\; ~\text{ for all }~ \;\eta\in \mathcal{H}_{_\pi}.
\end{eqnarray*}
\end{proof}
\begin{Lm}\label{as_trans_j_leq_n}
Let $\pi$, $n$ be a stable factor-representation of $B$, and let $E$ be the same as in Corollary \ref{collora_of_ind_Th}.  Then $\mathcal{O}_j\,E=0$ for all $j\in\{1,2,\ldots, n={\rm cd}\,\pi\}$.
\end{Lm}
\begin{proof}[First proof.]
If $\{\mathfrak{s}(x)\}_{x\in X}$ is coset representatives of left coset $X=B/{B_\infty^{(n)}}$ then
by  Corollary \ref{collora_of_ind_Th}
$
\sum\limits_{x\in X}\,\pi(\mathfrak{s}(x))\,E\,\pi\left( \mathfrak{s}(x)\right)^* =I.
$
Thus is sufficient to prove that
\begin{eqnarray}
E\,\pi(\mathfrak{s}(x))\,\mathcal{O}_j\,E=0\;~\text{ for all }~ x\in X,\; j\in\{1,2,\ldots, n={\rm cd}\,\pi\}.
\end{eqnarray}
But $\mathfrak{s}(x)\,(j\;\; m)\notin B_n\,B_{n\infty}=B_\infty^{(n)}$ for all $m\in\{n+1,n+2,\ldots\}\setminus ({\rm supp}\,\mathfrak{s}(x))$. Hence, using Corollary \ref{collora_of_ind_Th}, we have $E\,\pi(\mathfrak{s}(x)\,\pi((j\;\;m))\,E=0$. Now our lemma follows from definition of $\mathcal{O}_j$. \qed

\noindent{\it Second proof.} If $j\leq n$ and $n<m_1<m_2$ then $(j\;\;m_1)(j\;\;m_2)\notin B_\infty^{(n)}$. Hence, applying Corollary \ref{collora_of_ind_Th}, we obtain
\begin{eqnarray*}
E\pi((j\;\;m_1)(j\;\;m_2))E=0.
\end{eqnarray*}
Therefore, $0=w-\lim\limits_{m_1\to\infty}\lim\limits_{m_2\to\infty}E\pi((j\;\;m_1)(j\;\;m_2))E=E\mathcal{O}_j^2\,E=E\mathcal{O}_j\,\left( E\mathcal{O}_j\right)^*$.
\end{proof}
\begin{Co}
Since $\mathcal{O}_j=\mathcal{O}_j^*$ for all $j$, according to Lemma \ref{as_trans_j_leq_n} and Corollary \ref{collora_of_ind_Th}, we have
$\mathcal{O}_j\,E=E\,\mathcal{O}_j$ for all $j$.
\end{Co}
By Proposition \ref{asymp_char},  there exist parameters $\alpha,\beta,\gamma,\sigma$ such that
$\chi_{\pi}^{as}=\chi_{\alpha\beta\gamma}^\sigma$, where $\chi_{\alpha\beta\gamma}^\sigma$ is defined by \ref{character_formula}.
\begin{Lm}\label{E=F}
Let $\pi$ and $E$ be the same as in Lemma \ref{as_trans_j_leq_n}, and let $F$ be an orthogonal projection onto subspace $\bigcap\limits_{j=1}^n\,{\rm Ker}\,\mathcal{O}_j$.  If $\gamma_0=\gamma_1=0$ then
$E=F$.
\end{Lm}
\begin{proof}
It follows from Lemma \ref{as_trans_j_leq_n} that $E\leq F$.

Now we assume that $f=F-E\neq 0$.

By Lemma \ref{unique_involution}, we can assume without loss of generality that coset representatives $\{\mathfrak{s}_x\}_{x\in X}$ of left coset $X=B/{B_\infty^{(n)}}$ looks like
\begin{eqnarray*}
\mathfrak{s}_x=\left(\,{^x}\!p_1\;\;\,{^x}\!r_1\right)\left(\,{^x}\!p_2\;\;\,{^x}\!r_2\right)\cdots\left(\,{^x}\!p_q\;\;\,{^x}\!r_q\right),
\end{eqnarray*}
where $\left\{{^x}\!p_1<{^x}\!p_2<\ldots<{^x}\!p_q \right\}\subset\{1,2,\ldots n\}$ and $\left\{{^x}\!r_1<{^x}\!r_2<\ldots<{^x}\!r_q \right\}$ $\subset \{n+1,n+2, \ldots\}$.

We recall that normalized character $\chi_{_E}$ of the finite type factor-representation $\pi_{E}$ of group $B_\infty^{(n)}$ from  Corollary \ref{collora_of_ind_Th} has the form
\begin{eqnarray*}
\chi_{_E}(b_1\,b_2)=\chi_n(b_1)\chi_{_\pi}^{as}(b_2),\; ~\text{ where }~\; b_1\in B_n, b_2\in B_{n\infty}.
\end{eqnarray*}
Here $\chi_n$ is normalize character of the irreducible representation $\mathfrak{Ir}_{_{\,^0\!\lambda\,^1\!\!\lambda}}$ of $B_n$ (see section \ref{repr_subgroup}) and $\chi_{_\pi}^{as}=\chi_{\alpha\beta\gamma}^\sigma$ is indecomposable character of group $B_{n\infty}$.

Let $j>n$.
Since $\pi_{_E}$ is a finite type factor-representation of $B_\infty^{(n)}$, by Lemma \ref{spectral_decomposition},   operator $ E\,O_j\,E =E\,\lim\limits_{m\to\infty}\pi((j\;\;m))\,E$ has the following spectral decomposition
\begin{eqnarray*}
 E\,O_j\,E =\sum\limits_{\lambda\in S\left(E\mathcal{O}_j E\right)} \lambda E_j(\lambda).
\end{eqnarray*}
Applying Lemma \ref{spectral_decomposition} again and condition $\gamma_0=\gamma_1=0$ $\Leftrightarrow$ $\sum\alpha_i+\sum\beta_i=1$, we obtain
\begin{eqnarray}\label{zero_KER}
{\rm Ker}\,E\,O_j\,E=0.
\end{eqnarray}
Since $f=F-E>0$, we find $\mathfrak{s}_x\neq e$ such that $f\,\pi(\mathfrak{s}_x)\,E\neq 0$. Hence, using \eqref{zero_KER}, we have
\begin{eqnarray*}
O_j\,E\,\pi(\mathfrak{s}_x)\,f\neq 0 \;~\text{ for all }~ j>n.
\end{eqnarray*}
From this, applying Lemmas \ref{o_transp} and \ref{as_trans_j_leq_n}, we have
\begin{eqnarray}\label{neq_0}
O_j\,E\,\pi(\mathfrak{s}_x)\,f=E\,\pi(\mathfrak{s}_x)\,\mathcal{O}_{\mathfrak{s}_x(j)}\,f\neq0 \;~\text{ for all}~ j>n.
 \end{eqnarray}
 Since $\mathfrak{s}_x\neq e$, there exists $j>n$ such that $\mathfrak{s}_x(j)\leq n$. Therefore, $E\,\pi(\mathfrak{s}_x)\,\mathcal{O}_{\mathfrak{s}_x(j)}\,f$ $=0$. This contradicts \eqref{neq_0}.
\end{proof}
\begin{Prop}\label{prop_cd}
Let $\Pi$ be the same as in Theorem \ref{alpha+beta=1}. Then ${\rm cd}\,\Pi =n$.
\end{Prop}
\begin{proof}
We will follow the notation used  in the proof of  Theorem \ref{alpha+beta=1}. It follows from \eqref{stability_equality} that ${\rm cd}\,\Pi \leq n$.
Suppose that ${\rm cd}\,\Pi=d<n$. Then there exists a nonzero normal nonnegative functional $\varphi$  on $\Pi(B)''$ such that
\begin{eqnarray*}
  \varphi\left( \Pi(b) m\Pi(b^{-1})\right)=\varphi(m) ~\text{for all }~ m\in \Pi(B)'' ~\text{ and }~ b\in B_\infty^{(d)} =B_dB_{d\infty}.
\end{eqnarray*}
If $\widetilde{E}$ = ${\rm supp}\, \varphi$, then, applying   Corollary \ref{collora_of_ind_Th}, we obtain that  $\widetilde{E}\Pi(B)''\widetilde{E}=\left(\widetilde{E}\,\Pi\left(B_\infty^{(d)} \right)\,\widetilde{E} \right)''$ is a finite type factor and $\widetilde{E}\in \Pi(B_\infty^{(d)})'$. By Theorem \ref{induction_theorem},
\begin{eqnarray}
\widetilde{E}\Pi(b)\widetilde{E}=0\;~\text{ for all }~ b\notin B_\infty^{(d)}.
\end{eqnarray}
It follows from Lemma \ref{E=F} that
\begin{eqnarray}\label{E<tilde_E}
\bigcap\limits_{j=1}^d{\rm Ker}\,\mathcal{O}_j=\widetilde{E}.
\end{eqnarray}
Denoting by $E$ the projection onto subspace $\left\{\eta\in \mathcal{H}_\Pi: \eta(x)=0 ~\text{ for all } x\neq \widetilde{e}\right\}$, we obtain, by \eqref{indicator_formula},
\begin{eqnarray}
E\mathcal{H}_\Pi=\bigcap\limits_{m=1}^n {\rm Ker}\,\mathcal{O}_m.
\end{eqnarray}
 Hence, using \eqref{E<tilde_E}, we have
\begin{eqnarray}
E\leq\widetilde{E},\;E\in \Pi(B)'' \cap \Pi(B_\infty^{(n)})'\cap\left\{ \left\{ \mathcal{O}_j\right\}_{j=1}^\infty\right\}'.
\end{eqnarray}
Therefore, $E\in \widetilde{E}\Pi(B)''\widetilde{E}=\left(\widetilde{E}\,\Pi\left(B_\infty^{(d)} \right)\,\widetilde{E} \right)''$.  Since $d<n$, the collection $\left\{ (n\;\;m)\right\}_{m=n+1}^\infty\subset B_{d\infty}\setminus \left(B_\infty^{(n)} \right)$.  According to equality  $$E\mathcal{H}_\Pi=\left\{\eta\in \mathcal{H}_\Pi: \eta(x)=0 ~\text{ for all } x\neq \widetilde{e}\right\},$$
we have that the projections $\Pi((n\;\;m_1))\,E\,\Pi((n\;\;m_1))$ and  $\Pi((n\;\;m_2))\,E\,\Pi((n\;\;m_2))$ are orthogonal for different $m_1$ and $m_2$. Therefore, $\left\{  \Pi((n\;\;m))\,E\,\Pi((n\;\;m))\right\}_{m=n+1}^\infty$ is an infinite collection of pairwise orthogonal equivalennt projections in finite type factor  $\widetilde{E}\,\Pi\left( B_\infty^{(d)}\right)''\widetilde{E}$, a contradiction.
\end{proof}
\subsection{The central depth of representation ${\rm Ind}^B_{B_\infty^{(n)}}\,^n\!\pi$ in the case $\gamma_0$ $+\gamma_1>0$.}
We recall that $\,^n\!\pi$ is a finite type factor-representation of subgroup $B_\infty^{(n)}\subset B$ (see sections \ref{irreducible_repr_of_B_n}, \ref{repr_subgroup}). By Proposition \ref{Prop_Ind}, $\,^n\!\pi={\rm Ind}_{H_{\infty\Omega_{_{kn}}}}^{B_\infty^{(n)}}\,\Xi$, where $\,\Xi$ is finite type factor representation of subgroup $H_{\infty\Omega_{_{kn}}}=B_k\cdot B_{kn}\cdot B_{n\infty}$, which defined by \eqref{tilde_rho}. It follows from this that
\begin{eqnarray*}
{\rm Ind}^B_{B_\infty^{(n)}}\,^n\!\pi={\rm Ind}^B_{B_\infty^{(n)}}\,{\rm Ind}_{H_{\infty\Omega_{_{kn}}}}^{B_\infty^{(n)}}\,\Xi={\rm Ind}_{H_{\infty\Omega_{_{kn}}}}^B\,\Xi.
 \end{eqnarray*}
 On the other side, considering the chain of subgroups $H_{\infty\Omega_{_{kn}}}=B_k\cdot B_{kn}\cdot B_{n\infty}$ $\subset B_\infty^{(k)}=B_k\cdot B_{k\infty}\subset B$, we have
 \begin{eqnarray} \label{ind_through}
 {\rm Ind}^B_{B_\infty^{(n)}}\,^n\!\pi={\rm Ind}_{H_{\infty\Omega_{_{kn}}}}^B\,\Xi={\rm Ind}^B_{B_\infty^{(k)}}{\rm Ind}_{H_{\infty\Omega_{_{kn}}}}^{B_\infty^{(k)}}\,\Xi.
 \end{eqnarray}
By Theorem \ref{II_1_theorem},  it is sufficient to find the central depth only in two cases:
\begin{itemize}
  \item {\it i)} $\gamma_0=0$ and $\gamma_1>0$;
  \item  {\it ii)} $\gamma_0>0$ and $\gamma_1=0$.
\end{itemize}
\begin{Th}\label{cd_gamma_0=0_gamma_1>0}
Let $\Pi={\rm Ind}_{H_{\infty\Omega_{_{kn}}}}^B\,\Xi$. If $\gamma_0+\gamma_1>0$ then  we have the following conclusions:
\begin{itemize}
  \item {\it a)} if $\gamma_0>0$ and $\gamma_1>0$ then ${\rm cd}\,\Pi=0$, i. e. $\Pi$ is ${\rm II}_1$-factor-representation (see Theorem \ref{II_1_theorem});
  \item {\it b)} if  $\gamma_0=0$ and $\gamma_1>0$ then ${\rm cd}\,\Pi=k$;
  \item {\it c)} if  $\gamma_0>0$ and $\gamma_1=0$ then ${\rm cd}\,\Pi=n-k$.
\end{itemize}
\end{Th}
\begin{proof}
Property {\it a)} follows from Theorem \ref{II_1_theorem}.

\noindent  {\it b)}\ \  By \eqref{ind_through}, $\Pi={\rm Ind}^B_{B_\infty^{(k)}}{\rm Ind}_{H_{\infty\Omega_{_{kn}}}}^{B_\infty^{(k)}}\,\Xi$. According to Theorem \ref{Ind^B_infty_(k)}, $\,^k\!\pi={\rm Ind}_{H_{\infty\Omega_{_{kn}}}}^{B_\infty^{(k)}}\,\Xi$ is ${\rm II}_1$-factor-representation of subgroup $B_\infty^{(k)}$. Take $z=(z_1,\ldots,z_k,0,0,\ldots)$ $\in \mathbb{Z}_2^k$, $s\in\mathfrak{S}_k$ and $b\in B_{k\infty}$. Then, applying Propositions  \ref{Prop_Ind} and \ref{quasi_equivalent_ind}, we obtain that up to quasi-equivalence $\,^k\!\pi$  has a form
\begin{eqnarray}
\,^k\!\pi(zsb)\stackrel{Proposition \ref{Prop_Ind}}{=}\Omega_{_{kn}}(z)\left( {\rm Irr}_{\,^0\!\lambda}(s)\otimes \pi_{\alpha\beta\gamma}^\sigma(b)\right)\stackrel{\eqref{mult_character}}{=}{\rm Irr}_{\,^0\!\lambda}(s)\otimes \pi_{\alpha\beta\gamma}^\sigma(b),
\end{eqnarray}
where $\pi_{\alpha\beta\gamma}^\sigma$ is a finite type factor-representation, corresponding to character $\chi_{\alpha\beta\gamma}^\sigma$ from \eqref{character_formula}. Below,   we shall use for convenience some of the notation introduced in the proof of Theorem
\ref{II_infty_case}.

Representation $\Pi$ acts in $\mathcal{H}_\Pi=l^2\left(L^2\left( \,^k\!M, \,^k\!{\rm tr} \right),X\right)$ by \eqref{ind1_action}. Denote by $F_k$ the orthogonal projection onto subspace $\left\{ \eta\in\mathcal{H}_\Pi: \eta(x)=0 \,~\text{ for all}~x\neq \widetilde{e}\right\}$.
In the proof of Theorem \ref{II_infty_case} , we showed that the orthogonal projections
\begin{eqnarray}\label{k^Q}
\,^{^k}\!\!Q_m=\frac{1}{2}\, E_m(0)\left(I+\Pi\left(1^{(m)}\right)\right),
\end{eqnarray}
 where $m=1,2,\ldots,k$, satisfy the condition
\begin{eqnarray}\label{prod^k}
F_k=\prod\limits_{m=1}^k \,^{^k}\!\!Q_m.
\end{eqnarray}
In particular, it follows from this that $F_k\in\Pi(B)''$. It is clear that ${\rm cd}\,\Pi\leq k$.

Suppose that $d={\rm cd}\,\Pi<k$. Applying Theorem \ref{induction_theorem} and Corollary \ref{collora_of_ind_Th}, we find orthogonal projection $E\in \Pi(B)''\cap \Pi\left( B_{\infty}^{(d)}\right)^\prime$ with the properties:
\begin{itemize}
  \item {\it 1)} $E\,\Pi(b)\, E=0$ for all $b\notin  B_{\infty}^{(d)}$;
  \item {\it 2)} the unitary representation $\Pi_{_E}$ of the group $ B_{\infty}^{(d)}$ generated  by the operators $\Pi_{_E}(b)=E\,\Pi(b)\,E$, where $b\in  B_{\infty}^{(d)}$, in the subspace $E\mathcal{H}_\Pi$  is a finite type factor-representation;
\end{itemize}
It follows from this that $\Pi$ induced  from the representation $\Pi_{_E}$ of subgroup $B_{\infty}^{(d)}$; i. e. $\Pi={\rm Ind}_{B_{\infty}^{(d)}}^B \Pi_{_E}$. Now, repeating the reasoning from the proof of Theorem  \ref{II_infty_case} (see \eqref{zero_projection}, \eqref{Product_zero_projection}, \eqref{projection_on_unit}), we obtain
\begin{eqnarray}\label{d^Q}
E=\prod\limits_{m=1}^d \,^{^d}\!\!Q_m, \;~\text{ where}~ \,^{^d}\!\!Q_m=\frac{1}{2}\, E_m(0)\left(I+\Pi\left(1^{(m)}\right)\right).
\end{eqnarray}
Since $d<k$ and, by \eqref{k^Q}, \eqref{d^Q} $, \,^{^d}\!\!Q_m= \,^{^k}\!\!Q_m$ for $m=1,2,\ldots,d$, it follows from \eqref{prod^k} that $F_k\leq E$. Therefore, projection $F_k$ belongs to finite type factor $E\Pi(B)''\,E = \Pi_{_E}(B_{\infty}^{(d)})''$. But nonzero projections $\,^m\!\!F_k$ $=\Pi_{_E}((k\;\;m))\,F_k\,\Pi_{_E}((k\;\;m))\in \Pi_{_E}(B_{d\infty}^{(d)})'' $ are pairwise equivalent and orthogonal for different $m=k+1, k+2,\ldots$. For this reason, algebra $ \Pi_{_E}(B_{\infty}^{(d)})''$ cannot be a finite type factor (see property {\it 2)}). This contradicts our assumption that $d<k$.

 To prove the {\it c)} we need to repeat the above reasoning for representation $\widetilde{\Pi}$, which defined as follows:
\begin{eqnarray*}
 \widetilde{\Pi}(b)= \left\{\begin{array}{rl}
\Pi(s), ~&\text{ if }~ s\in\mathfrak{S}_\infty\subset B;\\
(-1)^{\sum z_i}\Pi(z),&\text{ if } \; z=(z_1,z_2,\ldots)\in \;_0\mathbb{Z}_2^\infty \subset B.
\end{array}
\right.
\end{eqnarray*}
\end{proof}
 \subsection{Quasi-equivalence of the induced representations in the case $\gamma_0=\gamma_1=0$.}\label{gamma_zero}
We will use below the notation from section \ref{all_gamma_are_zero}.
\begin{Th}\label{quasi-equivalence _in_case_alpha+beta=1}
  Let $\,^n\!\pi$ and $\,^n\!\breve{\pi}$  be the finite type factor-representation of $B_\infty^{(n)}$, and let $\sum\alpha_i +\sum\beta_i=1$.  The representations $\Pi={\rm Ind}^B_{B_\infty^{(n)}}\,^n\!\pi$ and $\breve{\Pi}={\rm Ind}^B_{B_\infty^{(n)}}\,^n\!\breve{\pi}$ are quasi-equivalent if and only if the representations $\,^n\!\pi$ and $\,^n\!\breve{\pi}$ are quasi-equivalent.
\end{Th}
\begin{proof}
Denote the factors $\,^n\!\pi(B_\infty^{(n)})''$ and $\,^n\!\breve{\pi}(B_\infty^{(n)})''$ briefly by $\,^n\!M$ and $\,^n\!\breve{M}$, respectively. Without loss of generality we may assume   that $\,^n\!\pi$ and $\,^n\!\breve{\pi}$ act in $L^2\left(\,^n\!M, \,^n\!{\rm tr} \right)$ and $L^2\left(\,^n\!M, \,^n\!\breve{{\rm tr}} \right)$ by the operators of left multiplication. Here $\,^n\!{\rm tr}$ and $\,^n\!\breve{{\rm tr}}$ are faithful normal traces on the finite type factors  $\,^n\!M$ and $\,^n\!\breve{M}$. Suppose that $\Pi$ and $\breve{\Pi}$ are realized in the spaces $l^2(L^2\left(\,^n\!M, \,^n\!{\rm tr} \right),X)$ and $l^2(L^2\left(\,^n\!\breve{M}, \,^n\!\breve{{\rm tr}} \right),X)$ by \eqref{ind2_action}.

Since $\Pi$ and $\breve{\Pi}$ are the stable representations, it follows from Lemma \ref{o_transp} that there exist the operators $\mathcal{O}_j$ and $\breve{\mathcal{O}}_j$:
\begin{eqnarray*}
\mathcal{O}_j=w-\lim\limits_{m\to\infty}\Pi((j\;\;m)), \;
\breve{\mathcal{O}}_j=w-\lim\limits_{m\to\infty}\breve{\Pi}((j\;\;m)).
\end{eqnarray*}
Denote by $E_m(0)$ and $\breve{E}_m(0)$ the projections onto ${\rm Ker}\,\mathcal{O}_m$ and ${\rm Ker}\,\breve{\mathcal{O}}_m$, respectively.

Suppose that there exists an isomorphism $\theta: \Pi(B)''\mapsto \breve{\Pi}$ such that $\theta(\Pi(b))$ $=\breve{\Pi}(b)$ for all $b\in B$. Since $\theta(\mathcal{O}_m)=\breve{\mathcal{O}}_m$, the  equality $\theta(E_m(0))=\breve{E}_m(0)$ holds. Thus, introducing the notations  $E=\prod\limits_{m=1}^n\, E_m(0)$  and $\breve{E}=\prod\limits_{m=1}^n\, \breve{E}_m(0)$, we obtain
\begin{eqnarray*}
\theta(E)=\breve{E}.
\end{eqnarray*}
Applying \eqref{indicator_formula}, we have
\begin{eqnarray*}
\theta(E\,\Pi(b)\,E)=\breve{E}\,\breve{\Pi}(b)\,\Pi \breve{E}\;~\text{for all }~ b\in B.
\end{eqnarray*}
It means that representations
\begin{eqnarray*}
B_\infty^{(n)}\ni b\stackrel{\Pi_{E}}{\mapsto} E\,\Pi(b)\, E\; ~\text{and }~ B_\infty^{(n)}\ni b\stackrel{\breve{\Pi}_{{\breve{E}}}}{\mapsto} \breve{E}\,\breve{\Pi}(b)\, \breve{E}
\end{eqnarray*}
are quasi-equivalent.  But $\Pi_{E}$ is quasi-equivalent to $\,^n\!\pi$, and $\breve{\Pi}_{E}$ is quasi-equivalent to $\,^n\!\breve{\pi}$. Therefore,  $\,^n\!\pi$ and $\,^n\!\breve{\pi}$ are quasi-equivalent.

If $\,^n\!\pi$ and $\,^n\!\breve{\pi}$ are quasi-equivalent then $\Pi$ and $\breve{\Pi}$ are also quasi-equivalent by virtue of Proposition \ref{quasi_equivalent_ind}.
\end{proof}
\subsection{Quasi-equivalence of the induced representations in the case $\gamma_0+\gamma_1>0$.}
By Proposition \ref{Prop_Ind}, the  finite type factor-representations $\,^n\!\pi$ and $\,^n\!\breve{\pi}$  of $B_{\infty}^{(n)}$ $=B_n\cdot B_{n\infty}$ are induced from the representations $\Xi$ and $\breve{\Xi}$ of  subgroup $H_{\infty \Omega_{_{kn}}}$ $=B_k\cdot B_{kn}\cdot B_{n\infty}$, which are defined by \eqref{tilde_rho}. Therefore,
\begin{eqnarray}\label{chain_ind}
\begin{split}
\Pi={\rm Ind}_{B_\infty^{(n)}}^B\,\,^n\!\pi={\rm Ind}_{B_\infty^{(n)}}^B\,\,{\rm Ind}^{B_\infty^{(n)}}_{H_{\infty \Omega_{_{kn}}}}\,\Xi={\rm Ind}_{H_{\infty \Omega_{_{kn}}}}^B\,\Xi,\\
\breve{\Pi}={\rm Ind}_{B_\infty^{(m)}}^B\,\,^m\!\breve{\pi}={\rm Ind}_{ H_{\infty \Omega_{_{lm}}}}^B\,\breve{\Xi},\;~\text{ where }~k\leq n, l\leq m \\
\;~\text{ and }~  H_{\infty \Omega_{_{lm}}}=B_l\cdot B_{lm}\cdot B_{m\infty}.
\end{split}
\end{eqnarray}
Let
$\;^0\!\lambda\vdash k, \;^1\!\!\lambda\vdash (n-k),\alpha, \beta, \gamma, \sigma$ and $\;^0\!\breve{\lambda}\vdash l,\;^1\!\!\breve{\lambda}\vdash (m-l), \breve{\alpha}, \breve{\beta}, \breve{\gamma}, \breve{\sigma}$ be the parameters of ${\rm II}_1$-factor-representations $\Xi$ and $\breve{\Xi}$ of the group $ H_{\infty \Omega_{_{kn}}}=B_k\cdot B_{kn}\cdot B_{n\infty}$ and $ H_{\infty \Omega_{_{lm}}}=B_l\cdot B_{lm}\cdot B_{m\infty}$, respectively.

Next statement follows from Proposition \ref{asymp_char}.
\begin{Prop}\label{as_character_equality}
  Let $\chi^{as}_{_{\Pi}}$ and $\chi^{as}_{_{\breve\Pi}}$ be the asymptotical characters of the stable factor-representations $\Pi$ and $\breve\Pi$.  If $\Pi$ and $\breve\Pi$ are quasi-equivalent then $\chi^{as}_{_{\Pi}}=\chi^{as}_{_{\breve\Pi}}$; i.e. $\alpha=\breve\alpha$,$\beta=\breve\beta$, $\gamma=\breve\gamma$ and $\sigma=\breve\sigma$.
\end{Prop}
The following Proposition is a consequence of Theorems \ref{II_1_theorem},   \ref{Ind^B_infty_(k)} and Proposition \ref{as_character_equality}
\begin{Prop}
Let us assume that $\alpha=\breve\alpha,\beta=\breve\beta,\sigma=\breve\sigma, \gamma=\breve\gamma$ and  one of the following conditions holds:
\begin{itemize}
  \item {\it 1)} $ \gamma_0>0$ and $ \gamma_1>0$,;
  \item {\it 2)} $(\gamma_0=0) \& (\gamma_1>0)\& (k=0)\&(l=0)$;
  \item {\it 3)}  $(\gamma_0>0) \& (\gamma_1=0)\& (k=n)\&(l=m)$.
\end{itemize}
Then $\Pi$ and $\breve\Pi$ have ${\rm II}_1$-type. It follows from this that $\Pi$ and $\breve\Pi$ are quasi-equivalent.
\end{Prop}
\subsubsection{The case $\gamma_0=0, \gamma_1>0$.}
If representations $\Pi$ and $\breve\Pi$ are quasi-equivalent then ${\rm cd}\,\Pi={\rm cd}\,\breve\Pi$. Therefore, applying Theorem \ref{cd_gamma_0=0_gamma_1>0}, we can assume without less of generality that $k=l$. Thus, we conclude from \eqref{chain_ind} that
\begin{eqnarray}
\begin{split}
\Pi={\rm Ind}^B_{H_{\infty\Omega_{_{kn}}}}\,\Xi={\rm Ind}^B_{B_\infty^{(k)}}\,{\rm Ind}^{B_\infty^{(k)}}_{H_{\infty\Omega_{_{kn}}}}\,\Xi\\
\breve\Pi\stackrel{(l=k)}{=}{\rm Ind}^B_{{H}_{\infty\Omega_{_{k}}}}\,\breve\Xi={\rm Ind}^B_{B_\infty^{(k)}}\,{\rm Ind}^{B_\infty^{(k)}}_{H_{\infty\Omega_{_{km}}}}\,\breve\Xi.
\end{split}
\end{eqnarray}
By Theorem \ref{Ind^B_infty_(k)}, the representations $\,^k\!\pi=\,{\rm Ind}^{B_\infty^{(k)}}_{H_{\infty\Omega_{_{kn}}}}\,\Xi$ and $\,^k\!\breve\pi=\,{\rm Ind}^{B_\infty^{(k)}}_{\breve{H}_{\infty\Omega_{_{km}}}}\,\breve\Xi$
are defined on $b=zsg$, where $z\in\mathbb{Z}_2^k, s\in\mathfrak{S}_k, g\in\mathfrak{S}_{k\infty}$, by
\begin{eqnarray}
\begin{split}\label{Gamma_k_II_1}
^k\!\pi(b)=  {\rm Irr}_{\,^0\!\lambda}(s)\otimes \pi_{\alpha\beta\gamma}^\sigma(g),\\
^k\!\breve\pi(b)={\rm Irr}_{\,^0\!\breve\lambda}(s)\otimes \pi_{\alpha\beta\gamma}^\sigma(g).
\end{split}
\end{eqnarray}
Here ${\rm Irr}_{\,^0\!\lambda}$ (${\rm Irr}_{\,^0\!\breve\lambda}$) is irreducible representation of $\mathfrak{S}_k$, corresponding $\,^0\!\lambda\vdash k$ ($\,^0\!\breve\lambda\vdash k$); ${\rm II}_1$-factor-representation $\pi_{\alpha\beta\gamma}^\sigma$ is defined by indecomposable character $\chi_{\alpha\beta\gamma}^\sigma$ (see \eqref{character_formula}).
\begin{Th}\label{quasi equivalence gamma=0 gamma=1}
Let $\gamma_0=0$ and $\gamma_1>0$. The representations $\Pi={\rm Ind}^B_{B_\infty^{(k)}}\,^k\!\pi$ and $\breve{\Pi}={\rm Ind}^B_{B_\infty^{(k)}}\,^k\!\breve{\pi}$ are quasi-equivalent if and only if ${\rm II}_1$-factor-representations $\,^k\!\pi$ and $\,^k\!\breve{\pi}$ are quasi-equivalent.
\end{Th}
\begin{proof}
We follow the notation used in the proof of Theorem \ref{II_infty_case} (see page \pageref{proof_28}).
Let $\Pi$ and $\breve\Pi$ are realized in $\mathcal{H}_\Pi$ $=l^2\left(L^2\left( \,^k\!M, \,^k\!{\rm tr} \right),X\right)$ by \eqref{ind1_action}, where $ \,^k\!M$ denotes $\,^k\!\pi(B_\infty^{(k)})''$ or $\,^k\!\breve\pi(B_\infty^{(k)})''$.

Suppose that $\Pi$ and $\breve\Pi$ are quasi-equivalent. Then there exists isomorphism $\theta: \Pi(B)'' \mapsto \breve\Pi(B)''$ with the property: $\theta(\Pi(b))=\breve\Pi(b)$ for all $b\in B$. By Lemma \ref{o_transp}, there exist the following weak limits
\begin{eqnarray*}
\mathcal{O}_m=w-\lim\limits_{j\to\infty}\Pi((m\;\;j))\;~\text{ and }~ \breve{\mathcal{O}}_m=w-\lim\limits_{j\to\infty}\breve\Pi((m\;\;j)).
\end{eqnarray*}
Let be $E_m(0)$ ($\breve E_m(0)$) be a orthogonal projection onto ${\rm Ker}\,\mathcal{O}_m$ (${\rm Ker}\,\breve{\mathcal{O}}_m$).
Then
\begin{eqnarray}\label{Theta_E_breve_E}
\theta( E_m(0))=\breve{E}_m(0).
\end{eqnarray}
Set $Q_m(\breve{Q}_m)=\frac{1}{2}E_m(0)\left(I+\Pi\left(1^{(m)}\right)\right)$ ($\frac{1}{2}\breve{E}_m(0)\left(I+\breve\Pi\left(1^{(m)}\right)\right)$). It follows from \eqref{projection_on_unit} and \eqref{product_of_Q_m} that projections $E=\prod\limits_1^k Q_m$ and $\breve E=\prod\limits_1^k \breve{Q}_m$ have the following properties:
\begin{itemize}
  \item i) $E\Pi(g)E=0$ and $\breve E \breve\Pi(g)\breve E=0$ for all $g\notin B\setminus B_\infty^{(k)}$;
  \item ii) the operators $\Pi_{_E}(g)=E\Pi(g)E$  and  $\breve{\Pi}_{_{\breve{E}}}(g)=\breve{E}\breve\Pi(g)\breve{E}$  generated ${\rm II}_1$-factor-representations of $B_\infty^{(k)}$;
  \item iii) $\Pi_{_E}$ is quasi-equivalent to $\,^k\!\pi$ and  $\breve{\Pi}_{_E}$ is quasi-equivalent to $\,^k\!\breve\pi$.
\end{itemize}
But, by \eqref{Theta_E_breve_E}, $\Pi_{_E}$ and $\breve{\Pi}_{_{\breve{E}}}$ are quasi-equivalent. Therefore, the representations $\,^k\!\pi$ and $\,^k\!\breve\pi$ are also quasi-equivalent.

If $\,^k\!\pi$ and $\,^k\!\breve\pi$ are quasi-equivalent then quasi equivalence of representations  $\Pi$ and $\breve\Pi$ follows from  Proposition \ref{quasi_equivalent_ind}.
\end{proof}
{}
B. Verkin ILTPE of NASU - B.Verkin Institute for Low Temperature Physics and Engineering
of the National Academy of Sciences of Ukraine

 n.nessonov@gmail.com
\end{document}